\documentclass[aop,preprint]{imsart}

\usepackage{amsmath,amsthm,amsfonts,amssymb,bm,wasysym}
\usepackage{hyperref}
\usepackage[utf8]{inputenc}
\usepackage{float}
\startlocaldefs

\newtheorem{theorem}{Theorem}[section]

\numberwithin{equation}{section}
\newtheorem{lemma}[theorem]{Lemma}
\newtheorem{proposition}[theorem]{Proposition}
\newtheorem{remark}[theorem]{Remark}

\numberwithin{equation}{section}

\newtheorem*{theoremA}{Theorem A}
\newtheorem*{theoremB}{Theorem B}
\newtheorem*{theoremC}{Theorem C}

\def\N{\mathbb{N}}
\def\Z{\mathbb{Z}}
\def\E{\mathbb{E}}

\def\R{\mathbb{R}}

\def\bP{\mathbb{P}}

\def\bP{\mathbb{P}}

\def\CC{\mathcal{C}}

\def\BB{\mathcal{B}}

\def\NN{\mathcal{N}}

\def\JJ{\mathcal{J}}

\def\GG{\mathcal{G}}
\def\RR{\mathcal{R}}
\def\OO{\mathcal{O}}

\renewcommand{\phi}{\varphi}
\renewcommand{\hat}{\widehat}
\renewcommand{\epsilon}{\varepsilon}
\renewcommand{\tilde}{\widetilde}

\newcommand{\1}{{\text{\Large $\mathfrak 1$}}}

\newcommand{\var}{\operatorname{Var}}
\newcommand{\cov}{\operatorname{Cov}}

\renewcommand{\emptyset}{\varnothing}

\newcommand{\cp}{\mathrm{Cap}}

\endlocaldefs

\begin{document}

\begin{frontmatter}

% "Title of the paper"
\title{Capacity of the range in dimension $5$}
\runtitle{Capacity of the range in dimension $5$}

% indicate corresponding author with \corref{}
% \author{\fnms{John} \snm{Smith}\corref{}\ead[label=e1]{smith@foo.com}\thanksref{t1}}
% \thankstext{t1}{Thanks to somebody} 
% \address{line 1\\ line 2\\ printead{e1}}
% \affiliation{Some University}

\author{\fnms{Bruno} \snm{Schapira}\ead[label=e1]{bruno.schapira@univ-amu.fr}}
\address{Aix-Marseille Universit\'e, CNRS, Centrale Marseille, I2M, UMR 7373, 13453 Marseille, France\\ \printead{e1}}
%\and
%\author{\fnms{???} \snm{???}\ead[label=e2]{???}}
%\address{\printead{e2}}
%\affiliation{???}

\runauthor{Bruno Schapira}

\begin{abstract}
We prove a Central limit theorem for the capacity of the range of a symmetric random walk on $\mathbb Z^5$,  
under only a moment condition on the step distribution. The result is analogous to 
the central limit theorem for the size of the range in dimension three, obtained by Jain and Pruitt in 1971. In particular an atypical logarithmic correction appears in the scaling of the variance. 
The proof is based on new asymptotic estimates, which hold in any dimension $d\ge 5$, for the probability that the ranges of two independent random walks intersect. 
The latter are then used for computing covariances of some intersection events, at the leading order.
\end{abstract}



\end{frontmatter}

% AOS,AOAS: If there are supplements please fill:
%\begin{supplement}[id=suppA]
%  \sname{Supplement A}
%  \stitle{Title}
%  \slink[doi]{10.1214/00-AOASXXXXSUPP}
%  \sdatatype{.pdf}" 
%  \sdescription{Some text}
%\end{supplement}

\section{Introduction}\label{sec:intro}
Consider a random walk $(S_n)_{n\ge 0}$ on $\Z^d$, that is a process of the form $S_n=S_0+X_1+\dots + X_n$, where the $(X_i)_{i\ge 1}$ are independent and identically distributed. A general question is to understand the geometric properties of its range, that is the random set $\RR_n:=\{S_0,\dots,S_n\}$, and more specifically to analyze its large scale limiting behavior as the time $n$ is growing. 
In their pioneering work, Dvoretzky and Erd\'os \cite{DE51} proved a strong law of large numbers for the number of distinct sites in $\RR_n$, in any dimension $d\ge 1$.    
Later a central limit theorem was obtained first by Jain and Orey \cite{JO69} in dimensions $d\ge 5$, then by Jain and Pruitt \cite{JP71} in dimension $3$ and higher, and finally by Le Gall \cite{LG86} in dimension $2$, under fairly general hypotheses on the common law of the $(X_i)_{i\ge 1}$. Furthermore, a lot of activity has been focused on analyzing the large and moderate deviations, which we will not discuss here.

More recently some papers were concerned with other functionals of the range, including its entropy \cite{BKYY}, and its boundary \cite{AS17, BKYY, BY, DGK, Ok16}. 
Here we will be interested in another 
natural way to measure the size of the range, which also captures some properties of its shape. 
Namely we will consider its Newtonian capacity, defined for a finite subset $A\subset \Z^d$, as 
\begin{equation}\label{cap.def}
\cp(A) :=\sum_{x\in A} \bP_x[H_A^+ = \infty],
\end{equation}
where $\bP_x$ is the law of the walk starting from $x$, and $H_A^+$ denotes the first return time to $A$ (see \eqref{HAHA+} below). 
Actually the first study of the capacity of the range goes back to the earlier work by Jain and Orey \cite{JO69}, 
who proved a law of large numbers in any dimension $d\ge 3$; and more precisely that almost surely, as $n\to \infty$, 
\begin{equation}\label{LLN.cap}
\frac 1n \cp(\RR_n)\to \gamma_d,
\end{equation}
for some constant $\gamma_d$, which is nonzero if and only if $d\ge 5$ -- the latter observation being actually directly related to the fact that it is only in dimension $5$ 
and higher that two independent ranges have a positive probability not to intersect each other. However, until 
very recently to our knowledge there were no other work on the capacity of the range, even though the results of Lawler on the intersection of random walks 
incidentally gave a sharp asymptotic behavior of the mean in dimension four, see \cite{Law91}.

In a series of recent papers \cite{Chang, ASS18, ASS19}, the central limit theorem has been established for the simple random walk in any dimension $d\ge 3$, 
except for the case of dimension $5$, which remained unsolved so far. The main goal of this paper is to fill this gap, but in the mean time we obtain general results on 
the probability that the ranges of two independent walks intersect, which might be of independent interest.  
We furthermore obtain estimates for the covariances between such events, which is arguably one of the main novelty of our work; but we shall come back on this point a bit later.

Our hypotheses on the random walk are quite general: we only require that the distribution of the $(X_i)_{i\ge 1}$ is a symmetric and irreducible probability measure on $\Z^d$, which has a finite $d$-th moment. 
Under these hypotheses our first result is the following.  
\begin{theoremA}\label{theoremA}
Assume $d=5$. There exists a constant $\sigma>0$, such that as $n\to \infty$, 
$$\var(\cp(\RR_n)) \sim \sigma^2 \, n\log n.$$
\end{theoremA}
We then deduce a central limit theorem. 
\begin{theoremB}\label{theoremB}
Assume $d=5$. Then,  
$$\frac{\cp(\RR_n) - \gamma_5 n}{\sigma \sqrt{n\log n}}\quad \stackrel{(\mathcal L)}{\underset{n\to \infty}{\Longrightarrow}} \quad \mathcal N(0,1).$$
\end{theoremB}
As already mentioned, along the proof we also obtain a precise asymptotic estimate for the probability that the ranges of two independent walks starting from far away intersect.  
Previously to our knowledge only the order of magnitude up to multiplicative constants had been established, see \cite{Law91}.  
Since our proof works the same in any dimension $d\ge 5$, we state our result in this general setting. Recall that to each random walk one can associate a norm (see below for a formal definition), which we denote here by $\mathcal J(\cdot)$ (in particular in the case of the simple random walk it coincides with the Euclidean norm). 
\begin{theoremC}\label{theoremC}
Assume $d\ge 5$. Let $S$ and $\widetilde S$ be two independent random walks starting from the origin (with the same distribution). 
There exists a constant $c>0$, such that as $\|x\|\to \infty$, 
$$\bP\left[\RR_\infty \cap (x+\widetilde \RR_\infty)\neq \emptyset\right]\sim \frac{c}{\mathcal J(x)^{d-4}}.$$
\end{theoremC}
In fact we obtain a stronger and more general result. Indeed, first we get some control on the second order term, and show that it is $\OO (\|x\|^{4-d-\nu})$, for some constant $\nu>0$. Moreover, we also consider some functionals of the position of one of the two walks at its hitting time of the other range. More precisely, we obtain asymptotic estimates for quantities of the form $\E[F(S_\tau)\1\{\tau<\infty\}]$, 
with $\tau$ the hitting time of the range $x+\tilde \RR_\infty$, 
for functions $F$ satisfying some regularity property, see \eqref{cond.F}. In particular, it applies to functions of the form $F(x)=1/\mathcal J(x)^\alpha$, for any $\alpha\in [0,1]$, for which we obtain that for some constants $\nu>0$, and $c>0$, 
$$ \E\left[\frac{ \1\{\tau<\infty\} }{1+\mathcal J(S_\tau)^\alpha}\right] = \frac{c}{\mathcal J(x)^{d-4+\alpha}} + \OO\left(\|x\|^{4-\alpha-d- \nu}\right).$$
Moreover, the same kind of estimates is obtained when one considers rather $\tau$ as the hitting time of $x+\tilde \RR[0,\ell]$, with $\ell$ a finite integer. These results are then 
used to derive asymptotic estimates for covariances of hitting events in the following four situations: let $S$, $S^1$, $S^2$, and $S^3$, be four independent random walks on $\Z^5$, all starting from the origin and consider either 
\begin{itemize}
\item[$(i)$] $A=\{\RR_\infty^1 \cap \RR[k,\infty)\neq \emptyset\}, \ \text{and}\ B=  \{\RR_\infty^2 \cap (S_k + \RR_\infty^3 )\neq \emptyset\}$, 
\item[$(ii)$] $A=\{\RR_\infty^1 \cap \RR[k,\infty)\neq \emptyset\}, \text{ and } B=  \{(S_k+\RR_\infty^2) \cap \RR[k+1,\infty)\neq \emptyset\}$, 
\item[$(iii)$] $A=\{\RR_\infty^1 \cap \RR[k,\infty)\neq \emptyset\}, \ \text{and}\ B=  \{(S_k+\RR_\infty^2) \cap \RR[0,k-1] \neq \emptyset\}$, 
\item[$(iv)$] $A=\{\RR_\infty^1 \cap \RR[1,k] \neq \emptyset\}, \ \text{and}\ B=  \{(S_k+\RR_\infty^2) \cap \RR[0,k-1] \neq \emptyset\}$. 
\end{itemize}
In all these cases, we show that for some constant $c>0$, as $k\to \infty$, 
$$\cov(A,B) \sim \frac{c}{k}. $$
Case $(i)$ is the easiest, and follows directly from Theorem C, since actually one can see that in this case both $\bP[A\cap B]$ and $\bP[A]\cdot \bP[B]$ are asymptotically equivalent to a constant times the inverse of $k$. However, the other cases are more intricate, partly due to some cancellations that occur between the two terms, 
which, if estimated separately, are both of order $1/\sqrt{k}$ in cases $(ii)$ and $(iii)$, or even of order $1$ in case $(iv)$. In these cases, we rely on the extensions of Theorem C, that we just mentioned above. 
More precisely in case $(ii)$ we rely on the general result applied 
with the functions $F(x)=1/\|x\|$, and its convolution with the distribution of $S_k$, while in cases $(iii)$ and $(iv)$ we use the extension to hitting times of 
finite windows of the range.  We stress also that showing the positivity of the constants $c$ here is a delicate part of the proof, especially in case $(iv)$, where it relies 
on the following inequality:  
$$\int_{0\le s \le t\le 1}\left( \E\left[\frac{1}{\|\beta_s-\beta_1\|^3\cdot \|\beta_t\|^3}\right] -\E\left[\frac{1}{\|\beta_s-\beta_1\|^3}\right] \E\left[\frac{1}{\|\beta_t\|^3}\right]\right) \, ds\, dt>0,$$
with $(\beta_u)_{u\ge 0}$ a standard Brownian motion in $\R^5$.

The paper is organized as follows. The next section is devoted to preliminaries, in particular we fix the main notation, recall known results on the transition kernel 
and the Green's function, and derive some basic estimates. In Section 3 we give the plan of the proof of Theorem A, which is cut into a number of intermediate results: 
Propositions \ref{prop.error}--\ref{prop.phipsi.2}. Propositions \ref{prop.error}--\ref{prop.phi0} are then proved in Sections $4$--$6$. 
The last one, which is also the most delicate one, requires 
Theorem C and its extensions. Its proof is therefore postponed to Section 8, while we first prove our general results on the intersection of two independent ranges in Section $7$, which is written in the general setting of random walks on $\Z^d$, for any $d\ge 5$, and can be read independently of the rest of the paper. 
Finally Section 9 is devoted to the proof 
of Theorem B, which is done by following a relatively 
well-established general scheme, based on the Lindeberg-Feller theorem for triangular arrays.

 %%%%%%%%%%%%%%%%%%%%%%%%%%%%%%%%%%%%%%%%%%%%%%%%%%%%%%%%%%%%%%%%%%%%%%%%%%%%%%

\section{Preliminaries}
\subsection{Notation}
We recall that we assume the law of the $(X_i)_{i\ge 1}$ to be a symmetric and irreducible probability measure\footnote{symmetric means that for all $x\in \Z^d$, $\bP[X_1=x]=\bP[X_1=-x]$, and irreducible means that for all $x$, 
$\bP[S_n=x]>0$, for some $n\ge 1$.} on $\Z^d$, $d\ge 5$, with a finite $d$-th moment\footnote{this means that $\E[\|X_1\|^d]<\infty$, with $\|\cdot \|$ the Euclidean norm.}. 
The walk is called aperiodic if the probability to be at the origin at time $n$ is nonzero for all $n$ large enough, and it is called bipartite if this probability is nonzero only when $n$ is even. 
Note that only these two cases may appear for a symmetric random walk.

Recall also that for $x\in \Z^d$, we denote by $\bP_x$ the law of the walk starting from $S_0=x$. When $x=0$, we simply write it as $\bP$.  
We denote its total range as $\RR_\infty :=\{S_k\}_{k\ge 0}$, and for $0\le k\le n\le +\infty$, set $\RR[k,n]:=\{S_k,\dots,S_n\}$.

For an integer $k\ge 2$, the law of $k$ independent random walks (with the same step distribution) starting from some $x_1,\dots, x_k\in \Z^5$, is denoted by 
$\bP_{x_1,\dots,x_k}$, or simply by $\bP$ when they all start from the origin.

We define 
\begin{equation}\label{HAHA+}
H_A:=\inf\{n\ge 0\ : \ S_n\in A\},\quad \text{and} \quad H_A^+ :=\inf\{n\ge 1\ :\ S_n\in A\},
\end{equation} 
respectively for the hitting time and first return time to a subset $A\subset \Z^d$, that we abbreviate respectively as $H_x$ and $H_x^+$ when $A$ is a singleton $\{x\}$.

We let $\|x\|$ be the Euclidean norm of $x\in \Z^d$.  
If $X_1$ has covariance matrix $\Gamma= \Lambda \Lambda^t$, we define its associated norm as 
$$\mathcal J^*(x) := |x\cdot \Gamma^{-1} x|^{1/2} = \|\Lambda^{-1} x\|,$$
and set $\mathcal J(x)= d^{-1/2}\mathcal J^*(x)$ (see \cite{LL} p.4 for more details).

For $a$ and $b$ some nonnegative reals, we let $a\wedge b:=\min(a,b)$ and $a\vee b:= \max(a,b)$.  
We use the letters $c$ and $C$ to denote constants (which could depend on the covariance matrix of the walk), whose values might change from line to line. We also use standard notation for the comparison of functions: we write $f=\OO(g)$, or sometimes $f\lesssim g$, if there exists a constant $C>0$, such that $f(x) \le Cg(x)$, for all $x$. Likewise, $f=o(g)$ means that $f/g \to 0$, and $f\sim g$ means that $f$ and $g$ are equivalent, that is if $|f-g| = o(f)$. Finally we write $f\asymp g$, when both $f=\OO(g)$, and $g=\OO(f)$. 

\subsection{Transition kernel and Green's function}
We denote by $p_n(x)$ the probability that a random walk starting from the origin ends up at position $x\in \Z^d$ after $n$ steps, that is $p_n(x):=\bP[S_n=x]$, and note that for any $x,y\in \Z^d$, one has $\bP_x[S_n=y] = p_n(y-x)$. Recall the definitions of $\Gamma$ and $\mathcal J^*$ from the previous subsection, and define 
\begin{equation}\label{pnbar}
\overline p_n(x) := \frac {1}{(2\pi n)^{d/2} \sqrt{\det \Gamma}} \cdot e^{-\frac{\mathcal J^*(x)^2}{2n}}.
\end{equation}
The first tool we shall need is a local central limit theorem, roughly saying that $p_n(x)$ is well approximated by $\overline p_n(x)$, under appropriate hypotheses. 
Such result has a long history, see in particular the standard books by Feller \cite{Feller} and Spitzer \cite{Spitzer}. We refer here to the more recent book of 
Lawler and Limic \cite{LL}, and more precisely to their Theorem 2.3.5 in the case of an 
aperiodic random walk, and to (the proof of) their Theorem 2.1.3 in the case of bipartite walks, which provide the result we need under minimal hypotheses (in particular it only requires a finite fourth-moment for $\|X_1\|$).  
\begin{theorem}[{\bf Local Central Limit Theorem}]\label{LCLT}
There exists a constant $C>0$, such that for all $n\ge 1$, and all $x\in \Z^d$,  
\begin{eqnarray*}
|p_n(x)-\overline p_n(x)| \le  \frac{C}{n^{(d+2)/2}}, 
\end{eqnarray*}
in the case of an aperiodic walk, and for bipartite walks, 
$$|p_n(x)+p_{n+1}(x)-2\overline p_n(x)| \le  \frac{C}{n^{(d+2)/2}}.$$
\end{theorem}
In addition, under our hypotheses (in particular assuming $\E[\|X_1\|^d]<\infty$), there exists a constant $C>0$, such that for any $n\ge 1$ and any 
$x\in \Z^d$ (see Proposition 2.4.6 in \cite{LL}), 
\begin{equation}\label{pn.largex}
p_n(x)\le C\cdot \left\{ 
\begin{array}{ll}
n^{-d/2} & \text{if }\|x\|\le \sqrt n,\\
\|x\|^{-d} & \text{if }\|x\|>\sqrt n.  
\end{array}
\right.
\end{equation}
It is also known (see the proof of Proposition 2.4.6 in \cite{LL}) that 
\begin{equation}\label{norm.Sn}
\E[\|S_n\|^d] =\OO(n^{d/2}).
\end{equation}
Together with the reflection principle (see Proposition 1.6.2 in \cite{LL}), and Markov's inequality, this gives that for any $n\ge 1$ and $r\ge 1$, 
\begin{equation}\label{Sn.large}
\bP\left[\max_{0\le k\le n}  \|S_k\|\ge r\right] \le C \cdot \left(\frac{\sqrt n}{r}\right)^{d}.
\end{equation}
Now we define for $\ell \ge 0$, $G_\ell(x) := \sum_{n\ge \ell} p_n(x)$. 
The {\bf Green's function} is the function $G:=G_0$. A union bound gives 
\begin{equation}\label{Green.hit}
 \bP[x\in \RR[\ell,\infty)] \le G_\ell(x).
 \end{equation}
By \eqref{pn.largex} there exists a constant $C>0$, such that for any $x\in \Z^d$, and $\ell \ge 0$, 
\begin{equation}\label{Green}
G_\ell(x) \le \frac{C}{\|x\|^{d-2} + \ell^{\frac{d-2}{2}} + 1}. 
\end{equation}   
It follows from this bound (together with the corresponding lower bound $G(x)\ge c\|x\|^{2-d}$, which can be deduced from Theorem \ref{LCLT}), and the fact that $G$ is 
harmonic on $\Z^d\setminus\{0\}$, that the hitting probability of a ball is bounded as follows (see the proof of \cite[Proposition 6.4.2]{LL}): 
\begin{equation}\label{hit.ball}
\bP_x\left[\eta_r<\infty \right] =\OO\left(\frac{r^{d-2}}{1+\|x\|^{d-2}}\right), \quad \text{with}\quad \eta_r:=\inf\{n\ge 0\ :\ \|S_n\|\le r\}. 
\end{equation}
We shall need as well some control on the overshoot. We state the result we need as a lemma and provide a short proof for the sake of completeness.  
\begin{lemma}[{\bf Overshoot Lemma}]\label{hit.ball.overshoot}
There exists a constant $C>0$, such that for all $r\ge 1$, and all $x\in \Z^d$, with $\|x\|\ge r$,  
\begin{equation*}
\bP_x[\eta_r<\infty,\, \|S_{\eta_r}\| \le r/2] \le \frac{C}{1+\|x\|^{d-2}}.
\end{equation*}
\end{lemma}
\begin{proof}
We closely follow the proof of Lemma 5.1.9 in \cite{LL}. Note first that one can alway assume that $r$ is large enough, for otherwise the result follows from \eqref{hit.ball}. Then define for $k\ge 0$, 
$$Y_k:= \sum_{n= 0}^{\eta_r} \1\{r+k \le \|S_n\|< r+(k+1)\}.$$
Let 
$$g(x,k) = \E_x[Y_k] = \sum_{n=0}^\infty \bP_x[r+k \le \|S_n\|\le r+k+1,\, n< \eta_r].$$
One has 
\begin{align*}
& \bP_x[\eta_r<\infty, \, \|S_{\eta_r}\| \le r/2] = \sum_{n=0}^\infty \bP_x[\eta_r=n+1, \, \|S_{\eta_r}\| \le r/2] \\
& = \sum_{n=0}^\infty \sum_{k=0}^\infty 
\bP_x[\eta_r=n+1, \, \|S_{\eta_r}\| \le r/2,\, r+k \le \|S_n\|< r+k+1]\\
& \le \sum_{k=0}^\infty \sum_{n=0}^\infty \bP_x\left[\eta_r>n,\, r+k \le \|S_n\|\le r+k+1,\, \|S_{n+1}-S_n\| \ge \frac r2 + k\right]\\
& = \sum_{k=0}^\infty g(x,k) \bP\left[\|X_1\| \ge \frac r2 + k \right] = \sum_{k=0}^\infty g(x,k) \sum_{\ell = k}^\infty \bP\left[\frac r2 +\ell  \le \|X_1\|<  \frac r2 + \ell+1 \right]\\
& = \sum_{\ell = 0}^\infty \bP\left[\frac r2 +\ell  \le \|X_1\|< \frac r2 + \ell+1 \right]\sum_{k=0}^\ell g(x,k). 
\end{align*}
Now Theorem \ref{LCLT} shows that one has $\bP_z[\|S_{\ell^2}\| \le r]\ge \rho$, for some constant $\rho>0$, uniformly in $r$ (large enough), $\ell\ge 1$, and $r\le \|z\|\le r+\ell$. It follows, exactly as in the
proof of Lemma 5.1.9 from \cite{LL}, that for any $\ell\ge 1$,  
$$\max_{\|z\|\le r+\ell} \sum_{0\le k< \ell} g(z,k) \le \frac{\ell^2}{\rho}.$$
Using in addition \eqref{hit.ball}, we get with the Markov property, 
$$\sum_{0\le k< \ell} g(x,k) \le C \frac{(r+\ell)^{d-2}}{1+\|x\|^{d-2}} \cdot \ell^2,$$
for some constant $C>0$. 
As a consequence one has 
\begin{align*}
& \bP_x[\eta_r<\infty, \, \|S_{\eta_r}\| \le r/2] \\
& \le \frac C{1+\|x\|^{d-2}} \sum_{\ell = 0}^\infty \bP\left[\frac r2+ \ell  \le \|X_1\|< \frac r2 +\ell + 1\right](r+\ell)^{d-2}(\ell+1) ^2\\
& \le \frac C{1+\|x\|^{d-2}} \E\left[\|X_1\|^{d-2}(\|X_1\| - r/2)^2\1\{\|X_1\|\ge r/2\}\right] \le \frac{C}{1+\|x\|^{d-2}},
 \end{align*}
 since by hypothesis, the $d$-th moment of $X_1$ is finite. 
\end{proof}

\subsection{Basic tools}
We prove here some elementary facts, which will be needed throughout the paper, and which are immediate consequences of the results from the previous subsection. 

\begin{lemma}\label{lem.upconvolG}
There exists $C>0$, such that for all $x\in \Z^d$, and $\ell \ge 0$, 
$$  \sum_{z\in \Z^d} G_\ell(z) G(z-x) \le \frac{C}{\|x\|^{d-4}+\ell^{\frac{d-4}{2}} + 1}.$$
\end{lemma}
\begin{proof}
Assume first that $\ell =0$. Then by \eqref{Green}, 
\begin{align*}
\sum_{z\in \Z^d} G(z) G(z-x) & \lesssim \frac{1}{1+\|x\|^{d-2}} \left(\sum_{\|z\|\le 2\|x\|} \frac{1}{1+\|z\|^{d-2}} +\sum_{\|z-x\|\le \frac{\|x\|}{2}} \frac{1}{1+\|z-x\|^{d-2}} \right)\\
& \quad  + \sum_{\|z\|\ge 2\|x\|}\frac{1}{1+\|z\|^{2(d-2)}} \lesssim \frac{1}{1+\|x\|^{d-4}}. 
\end{align*}
Assume next that $\ell\ge 1$. We distinguish two cases: if $\|x\|\le \sqrt \ell$, then by using \eqref{Green} again we deduce, 
$$\sum_{z\in \Z^d} G_\ell(z) G(z-x) \lesssim  \frac{1}{\ell^{d/2}}\cdot \sum_{\|z\|\le 2\sqrt \ell} \frac{1}{1+ \|z-x\|^{d-2}} + \sum_{\|z\|\ge 2\sqrt \ell} \frac 1{\|z\|^{2(d-2)}} 
\lesssim \frac{1}{ \ell^{\frac{d-4}2}}.$$
When $\|x\| >\sqrt \ell$, the result follows from case $\ell =0$, since $G_\ell(z) \le G(z)$. 
\end{proof}

\begin{lemma}\label{lem.sumG}
One has, 
\begin{equation}\label{exp.Green}
\sup_{x\in \Z^d} \, \E[G(S_n-x)] = \OO\left(\frac 1{n^{\frac{d-2}{2}}}\right),
\end{equation}
and for any $\alpha \in [0,d)$, 
\begin{equation}\label{exp.Green.x}
\sup_{n\ge 0} \, \E\left[\frac 1{1+\|S_n-x\|^\alpha } \right] = \OO\left(\frac 1{1+\|x\|^\alpha}\right). 
\end{equation}
Moreover, when $d=5$, 
\begin{equation} \label{sumG2}
\E\left[\Big(\sum_{n\ge k} G(S_n)\Big)^2\right] = \OO\left(\frac 1k\right).
\end{equation}
\end{lemma}
\begin{proof}
For \eqref{exp.Green}, we proceed similarly as in the proof of Lemma \ref{lem.upconvolG}. If $\|x\| \le \sqrt{n}$, one has using \eqref{pn.largex} and \eqref{Green}, 
\begin{align*}
 & \E[G(S_n-x)]  = \sum_{z\in \Z^d}p_n(z) G(z-x) \\
 &  \lesssim \frac 1{n^{d/2}} \sum_{\|z\|\le 2\sqrt n} \frac{1}{1+\|z-x\|^{d-2}} + \sum_{\|z\|>2\sqrt n} \frac {1}{\|z\|^{2d-2}} \lesssim n^{\frac{2-d}{2}},
\end{align*}
while if $\|x\|>\sqrt{n}$, we get as well
$$\E[G(S_n-x)]\lesssim \frac 1{n^{d/2}} \sum_{\|z\|\le \sqrt n/2} \frac{1}{\|x\|^{d-2}} + \sum_{\|z\|>\sqrt n/2} \frac {1}{\|z\|^d(1+\|z-x\|)^{d-2}} \lesssim n^{\frac{2-d}{2}}.$$
Considering now \eqref{exp.Green.x}, we write
\begin{align*}
& \E\left[\frac 1{1+\|S_n-x\|^\alpha} \right]  \le \frac{C}{1+\|x\|^\alpha} + \sum_{\|z-x\|\le \|x\|/2} \frac{p_n(z)}{1+\|z-x\|^\alpha} \\
& \stackrel{\eqref{pn.largex}}{\lesssim}  \frac{1}{1+\|x\|^\alpha} + \frac{1}{1+\|x\|^d} \sum_{\|z-x\|\le \|x\|/2} \frac{1}{1+\|z-x\|^\alpha} \lesssim \frac{1}{1+\|x\|^\alpha}.
\end{align*}
Finally for \eqref{sumG2}, one has using the Markov property at the second line, 
\begin{align*}
&\E\left[\Big(\sum_{n\ge k} G(S_n)\Big)^2\right] = \sum_{x,y} G(x)G(y)\E\left[\sum_{n,m\ge k}\1\{S_n=x,S_m=y\}\right]\\
&\le 2\sum_{x,y} G(x)G(y)\sum_{n\ge k}\sum_{\ell \ge 0}  p_n(x)p_\ell(y-x)= 2\sum_{x,y} G(x)G(y)G_k(x)G(y-x)\\
& \stackrel{\text{Lemma }\ref{lem.upconvolG}}{\lesssim} \sum_x \frac{1}{\|x\|^4}G_k(x)  \stackrel{\eqref{Green}}{\lesssim} \frac 1k.
\end{align*}
\end{proof}

The next result  deals with the probability that two independent ranges intersect. Despite its proof is a rather straightforward  
consequence of the previous results, it already provides upper bounds of the right order (only off by a multiplicative constant). 
\begin{lemma}\label{lem.simplehit}
Let $S$ and $\tilde S$ be two independent walks starting respectively from the origin and some $x\in \Z^d$. Let also $\ell$ and $m$ be two given nonnegative integers (possibly infinite for $m$). Define   
$$\tau:=\inf\{n\ge 0\ :\ \tilde S_n\in \RR[\ell,\ell+m]\}.$$
Then, for any function $F:\Z^d\to \R_+$, 
\begin{equation}\label{lem.hit.1}
\E_{0,x}[\1\{\tau <\infty\}F(\tilde S_\tau)] \le \sum_{i=\ell}^{\ell + m} \E[G(S_i-x)F(S_i)].  
\end{equation}
In particular, uniformly in $\ell$ and $m$, 
\begin{equation} \label{lem.hit.2}
\bP_{0,x}[\tau<\infty] = \OO\left(\frac{1}{1+\|x\|^{d-4} }\right).
\end{equation}
Moreover, uniformly in $x\in \Z^d$, 
\begin{equation}\label{lem.hit.3}
\bP_{0,x}[\tau<\infty] = \left\{ 
\begin{array}{ll}
 \OO\left(m\cdot \ell^{\frac{2-d}2}\right) & \text{if }m<\infty \\
\OO\left(\ell^{\frac{4-d}2} \right) & \text{if }m=\infty.
\end{array}
\right.
\end{equation}
\end{lemma} 
\begin{proof}
The first statement follows from  \eqref{Green.hit}. Indeed using this, and the independence between $S$ and $\tilde S$, we deduce that 
\begin{align*}
\E_{0,x}[\1\{\tau <\infty\}F(\tilde S_\tau)] & \le \sum_{i=\ell}^{\ell + m} \E_{0,x}[\1\{S_i\in \tilde \RR_\infty \} F(S_i)] \stackrel{\eqref{Green.hit}}{\le} \sum_{i=\ell}^{\ell +m} \E[G(S_i-x)F(S_i)]. 
\end{align*}
For \eqref{lem.hit.2}, note first that it suffices to consider the case when $\ell=0$ and $m=\infty$, as otherwise the probability is just smaller. Taking now $F\equiv 1$ in \eqref{lem.hit.1}, and using Lemma \ref{lem.upconvolG} gives the result. Similarly \eqref{lem.hit.3} directly follows from \eqref{lem.hit.1} and \eqref{exp.Green}. 
\end{proof}

\section{Scheme of proof of Theorem A}
\subsection{A last passage decomposition for the capacity of the range}
We provide here a last passage decomposition for the capacity of the range, in the same fashion as the well-known decomposition for the size of the range, which goes 
back  to the seminal paper by 
Dvoretzky and Erd\'os \cite{DE51}, and which was also used by Jain and Pruitt \cite{JP71} for their proof of the central limit theorem. We note that Jain and Orey \cite{JO69} used as well a similar decomposition in their analysis of the capacity of the range (in fact they used instead a first passage decomposition).

So let $(S_n)_{n\ge 0}$ be some random walk starting from the origin, and set  
$$\varphi_k^n:= \bP_{S_k}[H_{\RR_n}^+=\infty\mid \RR_n], \text{ and } Z_k^n := \1\{S_\ell \neq S_k,\ \text{for all } \ell = k+1,\dots, n\},$$
for all $0\le k\le n$, 
By definition of the capacity \eqref{cap.def}, one can write by recording the sites of $\RR_n$ according to their last visit,   
$$\cp(\RR_n)=\sum_{k=0}^n Z_k^n\cdot  \varphi_k^n.$$
A first simplification is to remove the dependance in $n$ in each of the terms in the sum. To do this, we need  some additional notation: we 
consider $(S_n)_{n\in \Z}$ a two-sided random walk starting from the origin (that is $(S_n)_{n\ge 0}$ and $(S_{-n})_{n\ge 0}$ are two independent walks starting from the origin), and denote its total range 
by $\overline \RR_\infty :=\{S_n\}_{n\in \Z}$. Then for $k\ge 0$, let 
$$\varphi(k):=\bP_{S_k}[H_{\overline \RR_\infty}^+ = \infty \mid (S_n)_{n\in \Z}], \text{ and } Z(k):=  \1\{S_\ell \neq S_k,\text{ for all }\ell \ge k+1 \}.$$
We note that $\phi(k)$ can be zero with nonzero probability, but that $\E[\phi(k)]\neq 0$ (see the proof of Theorem 6.5.10 in \cite{LL}).  We then define  
$$\mathcal C_n : = \sum_{k=0}^n Z(k)\varphi(k),\quad \text{ and }\quad W_n:=\cp(\RR_n) - \mathcal C_n.$$
We will prove in a moment the following estimate. 
\begin{lemma} \label{lem.Wn} One has 
$$\E[W_n^2] = \mathcal O(n).$$ 
\end{lemma}
Given this result, Theorem A reduces to an estimate of the variance of $\mathcal C_n$.  
To this end, we first observe that 
$$\var(\mathcal C_n) = 2 \sum_{0\le \ell < k\le n} \cov( Z(\ell)\varphi(\ell), Z(k)\varphi(k)) + \OO(n).$$
Furthermore, by translation invariance, for any $\ell < k$, 
$$\cov( Z(\ell)\varphi(\ell), Z(k)\varphi(k)) = \cov(Z(0)\varphi(0), Z(k-\ell)\varphi(k-\ell)),$$
so that in fact
$$\var(\mathcal C_n) = 2\sum_{\ell = 1}^n \sum_{k=1}^{\ell} \cov( Z(0)\varphi(0), Z(k)\varphi(k)) + \OO(n).$$ 
Thus Theorem A is a direct consequence of the following theorem. 
\begin{theorem}\label{prop.cov}
There exists a constant $\sigma>0$, such that 
\begin{equation*}
\cov( Z(0)\varphi(0), Z(k)\varphi(k)) \sim \frac{\sigma^2}{2k}.
\end{equation*}
\end{theorem}
This result is the core of the paper, and uses in particular Theorem C (in fact some more general statement, see  Theorem \ref{thm.asymptotic}). 
More details about its proof will be given in the next subsection, but first we 
show that $W_n$ is negligible by giving the proof of Lemma \ref{lem.Wn}. 

\begin{proof}[Proof of Lemma \ref{lem.Wn}]
Note that $W_n=W_{n,1} + W_{n,2}$, with 
$$W_{n,1} = \sum_{k=0}^n (Z_k^n - Z(k))\varphi_k^n, \quad \text{and}\quad W_{n,2} = \sum_{k=0}^n(\varphi_k^n- \varphi(k)) Z(k).$$
Consider first the term $W_{n,1}$ which is easier. Observe that $Z_k^n-Z(k)$ is nonnegative and bounded by the indicator function of the event 
$\{S_k\in \RR[n+1,\infty)\}$. Bounding also $\varphi_k^n$ by one, we get   
\begin{align*}
\E[W_{n,1}^2] & \le  \sum_{\ell = 0}^n \sum_{k=0}^n \E[(Z_\ell^n-Z(\ell))(Z_k^n -Z(k))] \\
&\le  \sum_{\ell = 0}^n \sum_{k=0}^n \bP\left[S_\ell \in \RR[n+1,\infty), \, S_k\in \RR[n+1,\infty)\right]. 
\end{align*}
Then noting that $(S_{n+1-k}-S_{n+1})_{k\ge 0}$ and $(S_{n+1+k}-S_{n+1})_{k\ge 0}$ are two independent random walks starting from the origin, we obtain 
\begin{align*}
 \E[W_{n,1}^2] &  \le  \sum_{\ell =1}^{n+1} \sum_{k=1}^{n+1} \bP[H_{S_\ell} <\infty, \, H_{S_k}<\infty]\le 2 \sum_{\ell = 1}^{n+1} \sum_{k=1}^{n+1} \bP[H_{S_\ell} \le H_{S_k} <\infty] \\
&\le 2 \sum_{1\le \ell \le k\le n+1} \bP[H_{S_\ell} \le H_{S_k} <\infty] + \bP[H_{S_k} \le H_{S_\ell} <\infty]. 
\end{align*}
Using next the Markov property and \eqref{Green.hit}, we get with $S$ and $\widetilde S$ two independent random walks starting from the origin,  
\begin{align*}
 \E[W_{n,1}^2] & \le 2 \sum_{1\le \ell \le k\le n+1} \E[G(S_\ell) G(S_k-S_\ell)] + \E[G(S_k) G(S_k- S_\ell)]\\
&\le 2 \sum_{\ell = 1}^{n+1} \sum_{k=0}^n \E[G(S_\ell)] \cdot \E[G(S_k)] + \E[G(S_\ell + \widetilde S_k) G(\widetilde S_k)]\\
&\le 4 \left(\sup_{x\in \Z^5} \sum_{\ell \ge 0} \E[G(x+S_\ell)]\right)^2 \stackrel{\eqref{exp.Green}}{=} \mathcal O(1).  
\end{align*}
We proceed similarly with $W_{n,2}$. Observe first that for any $k\ge 0$, 
$$0 \le \varphi_k^n - \varphi(k) \le \bP_{S_k}[H_{\RR(-\infty,0]}<\infty\mid S]  + \bP_{S_k}[H_{\RR[n,\infty)}<\infty\mid S].$$
Furthermore, for any $0\le \ell \le k\le n$, the two terms $\bP_{S_\ell}[H_{\RR(-\infty,0]}<\infty\mid S]$ and $\bP_{S_k}[H_{\RR[n,\infty)}<\infty\mid S]$ are independent. 
Therefore,   
\begin{align}\label{Wn2}
 \nonumber  \E[W_{n,2}^2]   \le & \sum_{\ell = 0}^n \sum_{k=0}^n \E[(\varphi_\ell^n-\varphi(\ell))(\varphi_k^n-\varphi(k))] \le  2\left(\sum_{\ell = 0}^n \bP\left[ H_{\RR[\ell, \infty)}<\infty\right]\right)^2 \\
& + 4 \sum_{0\le \ell \le k\le n} \bP\left[\RR^3_\infty \cap (S_\ell + \RR^1_\infty) \neq \emptyset, \, \RR^3_\infty \cap (S_k+\RR_\infty^2)\neq \emptyset \right],
\end{align}
where in the last term $\RR^1_\infty$, $\RR^2_\infty$ and $\RR^3_\infty$ are the ranges of three (one-sided) independent walks, independent of $(S_n)_{n\ge 0}$, starting from the origin (denoting here $(S_{-n})_{n\ge 0}$ as another walk $(S^3_n)_{n\ge 0}$). 
Now \eqref{lem.hit.3} already shows that the first term on the right hand side of \eqref{Wn2} is $\OO(n)$. 
For the second one, note that for any $0\le \ell \le k\le n$, one has  
\begin{align*}
&\bP\left[\RR^3_\infty \cap (S_\ell + \RR^1_\infty) \neq \emptyset, \, \RR^3_\infty \cap (S_k+\RR_\infty^2)\neq \emptyset \right]\\
  \le & \  
\E\left[|\RR^3_\infty \cap (S_\ell + \RR^1_\infty)| \cdot | \RR^3_\infty \cap (S_k+\RR_\infty^2)| \right]\\
 = &\  \E\left[\E[|\RR^3_\infty \cap (S_\ell + \RR^1_\infty)|\mid S,\, S^3] \cdot \E[| \RR^3_\infty \cap (S_k+\RR_\infty^2)|\mid S,\, S^3] \right] \\
 \stackrel{\eqref{Green.hit}}{\le} & \E\left[ \Big(\sum_{m\ge 0} G(S^3_m - S_\ell) \Big)  \Big(\sum_{m\ge 0} G(S^3_m - S_k) \Big) \right] =  \E\left[ \Big(\sum_{m\ge k} G(S_m - S_{k-\ell}) \Big)  \Big(\sum_{m\ge k} G(S_m) \Big) \right]\\
\le & \ \E\left[ \Big(\sum_{m\ge \ell} G(S_m) \Big)^2\right]^{1/2} \cdot  \E\left[ \Big(\sum_{m\ge k} G(S_m)\Big)^2\right]^{1/2} \stackrel{\eqref{sumG2}}{=} \OO\left(\frac{1}{1+\sqrt{k\ell}}\right), 
\end{align*} 
using invariance by time reversal at the penultimate line, and Cauchy-Schwarz at the last one. 
This concludes the proof of the lemma. 
\end{proof}

\subsection{Scheme of proof of Theorem \ref{prop.cov}}
We provide here some decomposition of $\phi(0)$ and $\phi(k)$ into a sum of terms involving intersection and non-intersection probabilities of different parts of the path $(S_n)_{n\in \Z}$. 
For this, we consider some sequence of integers $(\varepsilon_k)_{k\ge 1}$ satisfying $k>2\epsilon_k$, for all $k\ge 3$, and whose value will be fixed later.  
A first step in our analysis is to reduce the influence of the random variables $Z(0)$ and $Z(k)$, which play a 
very minor role in the whole proof. 
Thus we define  
$$Z_0:=\1\{S_\ell\neq 0,\, \forall \ell=1,\dots,\varepsilon_k\}, \text{ and } 
Z_k:=\1\{S_\ell\neq S_k, \, \forall \ell =k+1,\dots,k+\varepsilon_k\}.$$
Note that these notation are slightly misleading (as in fact $Z_0$ and $Z_k$ depend on $\epsilon_k$, but this shall hopefully not cause any confusion). 
One has 
$$\E[|Z(0)- Z_0|]= \bP[0\in \RR[\varepsilon_k+1,\infty)] \stackrel{\eqref{Green.hit}}{\le} G_{\varepsilon_k}(0) \stackrel{\eqref{Green}}{=} \OO(\varepsilon_k^{-3/2}),$$ 
and the same estimate holds for $\E[|Z(k)-Z_k|]$, by the Markov property. Therefore, 
$$\cov(Z(0)\varphi(0),Z(k)\varphi(k)) = \cov(Z_0\varphi(0),Z_k\varphi(k)) + \OO(\varepsilon_k^{-3/2}).$$
Then recall that we consider a two-sided walk $(S_n)_{n\in \Z}$, and that $\varphi(0) = \bP[H_{\RR(-\infty, \infty)}^+=\infty\mid S]$. Thus one can decompose $\varphi(0)$ as follows: 
$$\varphi(0) =\varphi_0- \varphi_1- \varphi_2-\varphi_3 + \varphi_{1,2} +\varphi_{1,3} + \varphi_{2,3} - \varphi_{1,2,3},$$
with 
$$\varphi_0:=\bP[H_{\RR[-\varepsilon_k,\varepsilon_k]}^+=\infty\mid S],\quad \varphi_1: = \bP[H^+_{\RR(-\infty,-\varepsilon_k-1]}<\infty, \, H_{\RR[-\varepsilon_k,\varepsilon_k]}^+=\infty \mid S], $$
$$\varphi_2 := \bP[H^+_{\RR[\varepsilon_k+1,k]}<\infty, H_{\RR[-\varepsilon_k,\varepsilon_k]}^+=\infty \mid S], \,  \varphi_3 := \bP[H^+_{\RR[k+1,\infty)}<\infty ,  H_{\RR[-\varepsilon_k,\varepsilon_k]}^+=\infty\mid S],$$
$$\varphi_{1,2}: = \bP[H^+_{\RR(-\infty,-\epsilon_k-1]}<\infty , \, H^+_{\RR[\epsilon_k+1,k]}<\infty , \, H_{\RR[-\varepsilon_k,\varepsilon_k]}^+=\infty \mid S], $$
$$\varphi_{1,3} := \bP[H^+_{\RR(-\infty,-\epsilon_k-1]}<\infty , \, H^+_{\RR[k+1,\infty)}<\infty, \, H_{\RR[-\varepsilon_k,\varepsilon_k]}^+=\infty  \mid S],$$ 
$$\varphi_{2,3} := \bP[H^+_{\RR[\epsilon_k+1,k]}<\infty,\, H^+_{\RR[k+1,\infty)}<\infty, \, H_{\RR[-\varepsilon_k,\varepsilon_k]}^+=\infty  \mid S],$$
$$\varphi_{1,2,3} := \bP[H^+_{\RR(-\infty, -\varepsilon_k-1]}<\infty,  H^+_{\RR[\varepsilon_k+1,k]}<\infty, H^+_{\RR[k+1,\infty)}<\infty,  H_{\RR[-\varepsilon_k,\varepsilon_k]}^+=\infty  \mid S].$$
We decompose similarly 
$$\varphi(k)=\psi_0 - \psi_1 - \psi_2 - \psi_3 + \psi_{1,2} + \psi_{1,3} + \psi_{2,3} - \psi_{1,2,3},$$
where index $0$ refers to the event of avoiding $\RR[k-\epsilon_k,k+\epsilon_k]$, index $1$ to the event of hitting $\RR(-\infty,-1]$, index $2$ to the event of hitting $\RR[0,k-\epsilon_k-1]$ and index $3$ to the event of hitting $\RR[k+\epsilon_k+1,\infty)$ (for a walk starting from $S_k$ this time).  
Note that $\varphi_0$ and $\psi_0$ are independent. Then write   
\begin{align}\label{main.dec}
  & \cov(Z_0\varphi(0),Z_k\varphi(k)) =  -\sum_{i=1}^3 \left(\cov(Z_0\varphi_i,Z_k\psi_0)+ \cov(Z_0\varphi_0,Z_k\psi_i)\right) \\
\nonumber &  +  \sum_{i,j=1}^3 \cov(Z_0\varphi_i , Z_k\psi_j)  +\sum_{1\le i <j\le 3} \left(\cov(Z_0\phi_{i,j},Z_k\psi_0) + \cov(Z_0\phi_0,Z_k\psi_{i,j})\right)+ R_{0,k},
\end{align}
where $R_{0,k}$ is an error term. 
Our first task will be to show that it is negligible. 
\begin{proposition}\label{prop.error}
One has $|R_{0,k}| = \OO\left(\epsilon_k^{-3/2}\right)$.  
\end{proposition}
The second step is the following. 
\begin{proposition}\label{prop.ij0}
One has
\begin{itemize}
\item[(i)]  
$|\cov(Z_0\phi_{1,2},Z_k\psi_0)|+|\cov(Z_0\phi_0,Z_k\psi_{2,3})| = \OO\left(\frac{\sqrt{\epsilon_k}}{k^{3/2}}\right)$,
\item[(ii)] $|\cov(Z_0\phi_{1,3},Z_k\psi_0)| + |\cov(Z_0\phi_0,Z_k\psi_{1,3})| = \OO\left(\frac{\sqrt{\epsilon_k}}{k^{3/2}}\cdot \log(\frac k{\epsilon_k}) + \frac{1}{\epsilon_k^{3/4} \sqrt k}\right)$,
\item[(iii)] 
$|\cov(Z_0\phi_{2,3},Z_k\psi_0)| +|\cov(Z_0\phi_0,Z_k\psi_{1,2})|= \OO\left(\frac{\sqrt{\epsilon_k}}{k^{3/2}}\cdot \log(\frac k{\epsilon_k})+ \frac{1}{\epsilon_k^{3/4} \sqrt k}\right)$. 
\end{itemize}
\end{proposition}
In the same fashion as Part (i) of the previous proposition, we show: 
\begin{proposition}\label{prop.phipsi.1}
For any $1\le i<j\le 3$, 
$$|\cov(Z_0\phi_i,Z_k\psi_j)| =\OO\left(\frac{\sqrt{\epsilon_k}}{k^{3/2}}\right), \quad |\cov(Z_0\phi_j,Z_k\psi_i)| =\OO\left(\frac{1}{\epsilon_k}\right).$$
\end{proposition}
The next step deals with the first sum in the right-hand side of \eqref{main.dec}. 
\begin{proposition}\label{prop.phi0}
There exists a constant $\alpha\in (0,1)$, such that 
$$\cov(Z_0\varphi_1, Z_k\psi_0) = \cov(Z_0\varphi_0,Z_k\psi_3) = 0,$$
$$|\cov(Z_0\varphi_2,Z_k\psi_0)|+ |\cov(Z_0\varphi_0,Z_k\psi_2)| = \OO\left(\frac{\sqrt{\varepsilon_k}}{k^{3/2}} \right),$$
$$|\cov(Z_0\varphi_3,Z_k\psi_0)|+ |\cov(Z_0\varphi_0,Z_k\psi_1)| = \OO\left(\frac{ \varepsilon_k^\alpha}{k^{1+\alpha}} \right).$$
\end{proposition}
At this point one can already deduce the bound $\var(\cp(\RR_n))= \OO(n \log n)$, just applying the previous propositions with say $\epsilon_k:=\lfloor k/4\rfloor$.

In order to obtain the finer asymptotic result stated in Theorem \ref{prop.cov}, it remains to identify the leading terms in \eqref{main.dec},   
which is the most delicate part. The result reads as follows.   
\begin{proposition}\label{prop.phipsi.2}
There exists $\delta>0$, such that if $\epsilon_k\ge k^{1-\delta}$ and $\epsilon_k = o(k)$, then for some positive constants $(\sigma_{i,j})_{1\le i\le j\le 3}$, 
$$\cov(Z_0\phi_j,Z_k\psi_i)\sim \cov(Z_0\phi_{4-i},Z_k\psi_{4-j}) \sim \frac{\sigma_{i,j}}{k}.$$
\end{proposition}
Note that Theorem \ref{prop.cov} is a direct consequence of \eqref{main.dec} and Propositions \ref{prop.error}--\ref{prop.phipsi.2}, which we prove now in the following sections.  

%%%%%%%%%%%%%%%%%%%%%%%%%%%%%%%%%%%%%%%%%%%%%%%%%%%%%%%%%%%%%%%%%%%%%%%%%%%%%%%%%%%%%%%%%%%%

\section{Proof of Proposition \ref{prop.error}}
We divide the proof into two lemmas. 
\begin{lemma}\label{lem.123}
One has 
$$\E[\varphi_{1,2,3}] = \OO\left(\frac{1}{\varepsilon_k \sqrt k}\right),\quad \text{and}\quad \E[\psi_{1,2,3}] = \OO\left(\frac{1}{\epsilon_k\sqrt k}\right).$$
\end{lemma}

\begin{lemma}\label{lem.ijl}
For any $1\le i<j\le 3$, and any $1\le \ell\le 3$, 
$$ \E[\varphi_{i,j}\psi_\ell] =\OO\left(\epsilon_k^{-3/2}  \right), \quad \text{and} \quad \E[\varphi_{i,j}] \cdot \E[\psi_\ell]  =\OO\left(\epsilon_k^{-3/2}  \right) .$$
\end{lemma}
Observe that the $(\varphi_{i,j})_{i,j}$ and $(\psi_{i,j})_{i,j}$ have the same law (up to reordering), and similarly for the $(\varphi_i)_i$ and $(\psi_i)_{i}$. 
Furthermore, $\varphi_{i,j}\le \varphi_i$ for any $i,j$. 
Therefore by definition of $R_{0,k}$ the proof of Proposition \ref{prop.error} readily follows from these two lemmas. 
For their proofs, we will use the following fact.  
\begin{lemma}\label{lem.prep.123}
There exists $C>0$, such that for any $x,y\in \Z^5$, $0\le \ell \le m$, 
\begin{equation*}
\sum_{i=\ell}^m \sum_{z\in \Z^5} p_i(z) G(z-y) p_{m-i}(z-x) \le 
\frac{C}{(1+\|x\|+\sqrt m)^5} \left(\frac 1{1+\|y-x\|} + \frac{1}{1+\sqrt{\ell}+\|y\|}\right).
\end{equation*}   
\end{lemma}
\begin{proof}
Consider first the case $\|x\|\le \sqrt m$. By \eqref{pn.largex} and Lemma \ref{lem.upconvolG}, 
$$
\sum_{i=\ell}^{\lfloor m/2\rfloor } \sum_{z\in \Z^5} p_i(z) G(z-y) p_{m-i}(z-x) \lesssim \frac{1}{1+m^{5/2}} \sum_{z\in \Z^5} G_{\ell}(z) G(z-y)
\lesssim \frac{(1+ m)^{-5/2}}{1+\sqrt{\ell}+\|y\|}, $$
with the convention that the first sum is zero when $m<2\ell$, and 
$$\sum_{i=\lfloor m/2\rfloor }^m \sum_{z\in \Z^5} p_i(z) G(z-y) p_{m-i}(z-x) \lesssim \frac{1}{1+m^{5/2}} \sum_{z\in \Z^5} G(z-y) G(z-x)
\lesssim \frac{(1+m)^{-5/2}} {1+\|y-x\|}. $$
Likewise, when $\|x\|>\sqrt m$, applying again \eqref{pn.largex} and Lemma \ref{lem.upconvolG}, we get 
\begin{align*}
& \sum_{i=\ell}^{m} \sum_{\|z-x\| \ge \frac{\|x\|}{2}}  p_i(z) G(z-y) p_{m-i}(z-x) \lesssim \frac{1}{\|x\|^5}  \sum_{z\in \Z^5} G_{\ell}(z)G(z-y) \lesssim \frac{\|x\|^{-5}}{ 1+\sqrt{\ell}+\|y\|},\\
& \sum_{i=\ell}^{m} \sum_{\|z-x\| \le \frac{\|x\|}{2}}  p_i(z) G(z-y) p_{m-i}(z-x) \lesssim \frac{1}{\|x\|^5}  \sum_{z\in \Z^5} G(z-y)G(z-x)  \lesssim \frac{\|x\|^{-5}}{1+\|y-x\|},
\end{align*}
which concludes the proof of the lemma. 
\end{proof}

One can now give the proof of Lemma \ref{lem.123}.
\begin{proof}[Proof of Lemma \ref{lem.123}]
Since $\varphi_{1,2,3}$ and $\psi_{1,2,3}$ have the same law, it suffices to prove the result for $\varphi_{1,2,3}$. Let $(S_n)_{n\in \Z}$ and $(\widetilde S_n)_{n\ge 0}$ be two independent random walks starting from the origin. Define 
$$\tau_1:=\inf\{n\ge 1\, :\, \tilde S_n \in \RR(-\infty,-\varepsilon_k-1]\},\ \tau_2:=\inf\{n\ge 1\, :\, \tilde S_n \in \RR[\varepsilon_k+1,k]\},$$
and 
$$\tau_3:= \inf\{n\ge 1\, :\, \tilde S_n\in \RR[k+1,\infty)\}.$$
One has 
\begin{equation}\label{phi123.tauij}
\E[\phi_{1,2,3}] \le \sum_{i_1\neq i_2 \neq i_3} \bP[\tau_{i_1}\le \tau_{i_2}\le \tau_{i_3}].
\end{equation}
We first consider the term corresponding to $i_1=1$, $i_2=2$, and $i_3=3$. 
One has by the Markov property, 
\begin{align*}
\bP[\tau_1\le \tau_2\le \tau_3<\infty] \stackrel{\eqref{lem.hit.2}}{\lesssim}  \E\left[ \frac{\1\{\tau_1\le \tau_2<\infty\}}{1+\|\tilde S_{\tau_2} - S_k\|}\right]\stackrel{\eqref{lem.hit.1}}{\lesssim} \sum_{i=\varepsilon_k}^k \E\left[ \frac{G(S_i-\tilde S_{\tau_1})\1\{\tau_1<\infty\}}{1+\| S_i - S_k\|}\right].
\end{align*}
Now define $\GG_i:=\sigma((S_j)_{j\le i})\vee \sigma((\tilde S_n)_{n\ge 0})$, and note that $\tau_1$ is $\GG_i$-measurable for any $i\ge 0$. 
Moreover, the Markov property and \eqref{pn.largex} show that 
$$\E\left[\frac{1}{1+\|S_i - S_k\|}\mid \GG_i\right] \lesssim \frac{1}{\sqrt{k-i}}.$$
Therefore,
\begin{align*}
& \bP[\tau_1\le \tau_2\le \tau_3<\infty]  \lesssim \sum_{i=\varepsilon_k}^k \E\left[\1\{\tau_1<\infty\}\cdot \frac{G(S_i-\tilde S_{\tau_1})}{1+\sqrt{k-i}}\right]\\
& \lesssim \sum_{z\in \Z^5}\bP[\tau_1<\infty, \, \tilde S_{\tau_1}=z] \cdot\left( \sum_{i=\epsilon_k}^{k/2} \frac{\E[G(S_i-z)]}{\sqrt k} + \sum_{i=k/2}^k \frac{\E[G(S_i-z)]}{1+\sqrt{k-i}}\right)\\
& \stackrel{\eqref{exp.Green}}{\lesssim} \frac{1}{\sqrt{k\epsilon_k}}\cdot  \bP[\tau_1<\infty] \stackrel{\eqref{lem.hit.2}}{\lesssim}\frac 1{\epsilon_k\sqrt k}. 
\end{align*}
We consider next the term corresponding to $i_1=1$, $i_2=3$ and $i_3=2$, whose analysis slightly differs from the previous one.  
First Lemma \ref{lem.prep.123} gives 
\begin{align}\label{tau132}
 & \bP[\tau_1\le \tau_3\le \tau_2<\infty]  =\sum_{x,y\in \Z^5} \E\left[\1\{\tau_1\le \tau_3<\infty, \tilde S_{\tau_3}=y, S_k=x\} \sum_{i=\epsilon_k}^k G(S_i-y) \right]\\
\nonumber = &  \sum_{x,y\in \Z^5} \left(\sum_{i=\epsilon_k}^k \sum_{z\in \Z^5} p_i(z)G(z-y) p_{k-i}(x-z)\right) \bP\left[\tau_1\le \tau_3<\infty, \tilde S_{\tau_3}=y\mid S_k=x\right]\\
\nonumber \lesssim & \sum_{x\in \Z^5} \frac{1}{(\|x\| + \sqrt k)^5}\left(\frac{\bP[\tau_1\le \tau_3 <\infty\mid S_k=x]}{\sqrt{\epsilon_k}} +
 \E\left[ \frac{\1\{\tau_1\le \tau_3<\infty\}}{1+\|\tilde S_{\tau_3} - x\|}\ \Big| \ S_k=x\right] \right).
\end{align}
We then have 
\begin{align*}
& \bP[\tau_1\le \tau_3<\infty\mid S_k=x] \stackrel{\eqref{lem.hit.2}}{\lesssim}  
\E\left[ \frac { \1\{\tau_1<\infty\}}{1+\|\tilde S_{\tau_1}-x\|}\right] \\
&\stackrel{\eqref{lem.hit.1}}{\lesssim} \sum_{y\in \Z^5} \frac{G_{\epsilon_k}(y) G(y)}{1+\|y-x\|} \stackrel{\text{Lemma }\ref{lem.upconvolG}}{\lesssim} \frac{1}{(1+\|x\|) \sqrt \epsilon_k} + \sum_{\|y-x\|\le \frac{\|x\|}{2}}\frac{G_{\epsilon_k}(y) G(y)}{1+\|y-x\|}  .
\end{align*}
Moreover, when $\|x\|\ge \sqrt {\varepsilon_k}$, one has 
\begin{align*}
\sum_{\|y-x\|\le \frac{\|x\|}{2}} \frac{G_{\epsilon_k}(y) G(y)}{1+\|y-x\|}  \stackrel{\eqref{Green}}{\lesssim} \frac{1}{\|x\|^6} \sum_{\|y-x\|\le \frac{\|x\|}{2}} \frac{1}{1+\|y-x\|} 
\lesssim  \frac{1}{\|x\|^2},
\end{align*}
while, when $\|x\|\le \sqrt {\varepsilon_k}$,  
$$\sum_{\|y-x\|\le \frac{\|x\|}{2}} \frac{G_{\epsilon_k}(y) G(y)}{1+\|y-x\|}  \stackrel{\eqref{Green}}{\lesssim} (1+\|x\|) \epsilon_k^{-3/2} \lesssim \frac {1}{\varepsilon_k}.$$
Therefore, it holds for any $x$, 
\begin{equation}\label{tau132.1}
\bP[\tau_1\le \tau_3<\infty\mid S_k=x]  \lesssim \frac{1}{(1+\|x\|)\sqrt \varepsilon_k}.
\end{equation}
Similarly, one has 
\begin{align}\label{tau132.2}
\nonumber & \E\left[ \frac{\1\{\tau_1\le \tau_3<\infty\}}{1+\|\tilde S_{\tau_3} - x\|}\ \Big| \ S_k=x\right] \le \E\left[ \sum_{y\in \Z^5} \frac{G(y-\tilde S_{\tau_1}) G(y-x)}{1+\|y-x\|}\1\{\tau_1<\infty\}\right]\\ 
& \le  \E\left[  \frac{\1\{\tau_1<\infty\}}{1+\|\tilde S_{\tau_1} - x\|^2}\right]  \le \sum_{y\in \Z^5} \frac{G_{\epsilon_k}(y) G(y)}{1+\|y-x\|^2} 
\lesssim \frac{1}{(1+\|x\|^2)\sqrt{\epsilon_k}}.
\end{align}
Injecting \eqref{tau132.1} and \eqref{tau132.2} into \eqref{tau132} finally gives 
$$\bP[\tau_1\le \tau_2\le \tau_3<\infty] \lesssim \frac{1}{\epsilon_k\sqrt k}.$$
The other terms in \eqref{phi123.tauij} are entirely similar, so this concludes the proof of the lemma.  
\end{proof}

For the proof of Lemma \ref{lem.ijl}, one needs some additional estimates that we state as two separate lemmas. 
\begin{lemma}\label{lem.prep.ijl}
There exists a constant $C>0$, such that for any $x,y\in \Z^5$, 
\begin{align*}
 \sum_{i=\epsilon_k}^{k-\epsilon_k}\sum_{z\in \Z^5} & \frac{p_i(z) G(z-y)}{(\|z-x\| + \sqrt{k-i})^5}\left(\frac{1}{1+\|z-x\|} + \frac 1{\sqrt{k-i}}\right) \\
& \le C\cdot \left\{
\begin{array}{ll}
\frac{1}{k^{5/2}}\left( \frac 1{1+\|x\|^2}  + \frac{1}{\varepsilon_k}\right) + \frac{1}{k^{3/2}\epsilon_k^{3/2}(1+\|y-x\|)}&\quad  \text{if }\|x\|\le \sqrt k\\
\frac{1}{\|x\|^5\epsilon_k}\left(1+\frac{k}{\sqrt{\epsilon_k}(1+\|y-x\|)} \right) &\quad \text{if }\|x\|>\sqrt k.
\end{array}
\right. 
\end{align*} 
\end{lemma}
\begin{proof}
We proceed similarly as for the proof of Lemma \ref{lem.prep.123}. Assume first that $\|x\|\le \sqrt k$. On one hand, using Lemma \ref{lem.upconvolG}, we get
$$\sum_{i=\epsilon_k}^{k/2} \frac{1}{\sqrt{k-i}} \sum_{z\in \Z^5} \frac{p_i(z) G(z-y)}{(\|z-x\| + \sqrt{k-i})^5}\lesssim \frac{1}{k^3} \sum_{z\in \Z^5} G_{\epsilon_k}(z) G(z-y)
\lesssim \frac{1}{k^{5/2}\sqrt{k\epsilon_k}},$$
and,  
\begin{align*}
&\sum_{i=\epsilon_k}^{k/2} \sum_{z\in \Z^5} \frac{p_i(z) G(z-y)}{(\|z-x\| + \sqrt{k-i})^5(1+\|z-x\|)} \lesssim \frac{1}{k^{5/2}}\sum_{z\in \Z^5} \frac{G_{\epsilon_k}(z)G(z-y)}{1+\|z-x\|}\\
&\lesssim \frac{1}{k^{5/2}} \left(\sum_{\|z-x\|\ge \frac{\|x\|}{2}} \frac{G_{\epsilon_k}(z)G(z-y)}{1+\|z-x\|} +  \sum_{\|z-x\|\le \frac{\|x\|}{2}} \frac{G_{\epsilon_k}(z)G(z-y)}{1+\|z-x\|}\right) \\
& \lesssim  \frac{1}{k^{5/2}} \left(\frac{1}{(1+\|x\|)\sqrt{\epsilon_k}} + \frac{1}{1+\|x\|^2}\right) \lesssim  \frac{1}{k^{5/2}} \left(\frac{1}{1+\|x\|^2}+ \frac 1{\epsilon_k}\right).
\end{align*}
On the other hand, by \eqref{pn.largex}
\begin{align*}
 \sum_{i=k/2}^{k-\epsilon_k} \sum_{\|z\|>2\sqrt k} \frac{p_i(z) G(z-y)}{(\|z-x\| + \sqrt{k-i})^5}\left(\frac{1}{1+\|z-x\|}+ \frac{1}{\sqrt{k-i}} \right)
 \lesssim \frac{1}{k^2}\sum_{\|z\|>2\sqrt k} \frac{G(z-y)}{\|z\|^5} \lesssim k^{-\frac 72}.
\end{align*}
Furthermore, 
\begin{align*}
 \sum_{i=k/2}^{k-\epsilon_k} \frac 1{\sqrt{k-i}} \sum_{\|z\|\le 2\sqrt k} \frac{p_i(z) G(z-y)}{(\|z-x\| + \sqrt{k-i})^5}
 \lesssim \frac{1}{k^2\epsilon_k}\sum_{\|z\|\le 2\sqrt k} \frac{G(z-y)}{1+\|z-x\|^3} \lesssim \frac{(k^2\epsilon_k)^{-1}}{1+\|y-x\|},
\end{align*}
and 
\begin{align*}
 \sum_{i=\frac k2}^{k-\epsilon_k}  \sum_{\|z\|\le 2\sqrt k} \frac{p_i(z) G(z-y)}{(\|z-x\| + \sqrt{k-i})^5}\frac{1}{1+\|z-x\|} & \lesssim \frac{1}{k^{3/2}\epsilon_k^{3/2}}\sum_{\|z\|\le 2\sqrt k} \frac{G(z-y)}{1+\|z-x\|^3} \\
 & \lesssim \frac{1}{k^{3/2}\epsilon_k^{3/2}} \frac{1}{1+\|y-x\|}.
\end{align*}
Assume now that $\|x\|>\sqrt k$. One has on one hand, using Lemma \ref{lem.upconvolG}, 
\begin{align*}
\sum_{i=\epsilon_k}^{k-\epsilon_k} \sum_{\|z-x\|\ge \frac{\|x\|}2} \frac{p_i(z) G(z-y)}{(\|z-x\| + \sqrt{k-i})^5}\left(\frac{1}{1+\|z-x\|}+ \frac{1}{\sqrt{k-i}} \right)
\lesssim \frac{1}{\|x\|^5 \epsilon_k}.
\end{align*}
On the other hand, 
\begin{align*}
 \sum_{i=\epsilon_k}^{k-\epsilon_k} \sum_{\|z-x\|\le \frac{\|x\|}2} \frac{p_i(z) G(z-y)}{(\|z-x\| + \sqrt{k-i})^5} \frac{1}{1+\|z-x\|} & \lesssim \frac{k}{\|x\|^5 \epsilon_k^{3/2}} \sum_{z\in \Z^5} \frac{G(z-y)}{1+\|z-x\|^3} \\
 &  \lesssim \frac{k}{\|x\|^5 \epsilon_k^{3/2}(1+\|y-x\|)}, 
\end{align*}
and 
\begin{align*}
 \sum_{i=\epsilon_k}^{k-\epsilon_k}\frac 1{\sqrt {k-i}} \sum_{\|z-x\|\le \frac{\|x\|}2} \frac{p_i(z) G(z-y)}{(\|z-x\| + \sqrt{k-i})^5}  & \lesssim \frac{\sqrt k}{\|x\|^5 \epsilon_k} \sum_{z\in \Z^5} \frac{G(z-y)}{1+\|y-x\|^3}  \\
& \lesssim \frac{\sqrt k}{\|x\|^5 \epsilon_k(1+\|y-x\|)}, 
\end{align*}
concluding the proof of the lemma. 
\end{proof}

\begin{lemma}\label{lem.prep.ijl2}
There exists a constant $C>0$, such that for any $x,y\in \Z^5$, 
\begin{align*}
&  \sum_{v\in \Z^5} \frac{1}{(\|v\|+\sqrt k)^5} \left(\frac{1}{1+\|x-v\|} + \frac{1}{ 1+\|x\|}\right)\frac 1{(\|x-v\|+ \sqrt{\epsilon_k} )^5} \left(\frac{1}{1+\|y-x\|}+\frac 1{1+\|y-v\|}\right)\\
& \le C\cdot \left\{
\begin{array}{ll}
\frac{1}{k^2\epsilon_k} \left( \frac 1{\sqrt {\epsilon_k}} + \frac{1}{1+\|x\|} +\frac 1{1+\|y-x\|}+\frac{\sqrt{\epsilon_k}}{(1+\|x\|) (1+\|y-x\|)}\right) & \quad \text{if }\|x\|\le \sqrt k\\
\frac{\log (\frac{\|x\|}{\sqrt{\epsilon_k}})}{\|x\|^5\sqrt{\epsilon_k}} \left(\frac{1}{1+\|y-x\|} +  \frac 1{\sqrt k} \right) & \quad \text{if }\|x\|>\sqrt k.
\end{array}
\right.
\end{align*}
\end{lemma}
\begin{proof}
Assume first that $\|x\| \le \sqrt k$. In this case it suffices to notice that on one hand, for any $\alpha\in \{3,4\}$, one has 
$$\sum_{\|v\|\le 2 \sqrt k} \frac{1}{(1+\|x-v\|^\alpha)(1+\|y-v\|^{4-\alpha})} = \OO(\sqrt k),$$
and on the other hand, for any $\alpha, \beta \in \{0,1\}$,
$$\sum_{\|v\|>2\sqrt k} \frac{1}{\|v\|^{10+\alpha} (1+\|y-v\|)^\beta}  = \OO(k^{-5/2 - \alpha - \beta}). $$ 
Assume next that $\|x\|>\sqrt k$. In this case it is enough to observe that  
$$\sum_{\|v\|\le \frac{\sqrt k}{2}} \left(\frac{1}{1+\|x-v\|} + \frac{1}{\|x\|}\right) \left(\frac{1}{1+\|y-x\|}+\frac 1{1+\|y-v\|}\right) \lesssim \frac {k^2}{(1+\|y-x\|)},$$ 
$$\sum_{\|v\|\ge \frac{\sqrt k}{2}} \frac{1}{\|v\|^5 (\sqrt{\epsilon_k}+\|x-v\|)^5} \lesssim \frac{\log (\frac{\|x\|}{\sqrt{\epsilon_k}})}{\|x\|^5} ,$$
$$\sum_{\|v\|\ge \frac{\sqrt k}{2}} \frac{1}{\|v\|^5 (\sqrt{\epsilon_k}+\|x-v\|)^5(1+\|y-v\|)} \lesssim \frac{\log (\frac{\|x\|}{\sqrt{\epsilon_k}})}{\|x\|^5}\left(\frac 1{\sqrt k} + \frac 1{1+\|y-x\|}\right). $$
\end{proof}

\begin{proof}[Proof of Lemma \ref{lem.ijl}]
First note that for any $\ell$, one has $\E[\psi_\ell] = \OO(\epsilon_k^{-1/2})$, by \eqref{lem.hit.3}. Using also similar arguments as in the proof of Lemma \ref{lem.123}, that we will not reproduce here, one can see that $\E[\phi_{i,j}] = \OO(\epsilon_k^{-1})$, for any $i\neq j$. Thus only the terms of the form $\E[\phi_{i,j}\psi_\ell]$ are at stake.

Let $(S_n)_{n\in \Z}$, $(\tilde S_n)_{n\ge 0}$ and $(\hat S_n)_{n\ge 0}$ be three independent walks starting from the origin. 
Recall the definition of $\tau_1$, $\tau_2$ and $\tau_3$ from the proof of Lemma \ref{lem.123}, and define analogously  
$$\hat \tau_1:=\inf\{n\ge 1 : S_k+\hat S_n \in \RR(-\infty,-1]\}, \, \hat \tau_2:=\inf\{n\ge 1 : S_k+\hat S_n \in \RR[0,k-\epsilon_k-1]\},$$
and 
$$\hat \tau_3:=\inf\{n\ge 1 : S_k+\hat S_n \in \RR[k+\epsilon_k+1,\infty)\}.$$
When $\ell\neq i,j$, one can take advantage of the independence between the different parts of the range of $S$, at least once we condition on the value of $S_k$. This allows for instance to write 
$$\E[\phi_{1,2}\psi_3] \le \bP[\tau_1<\infty,\, \tau_2<\infty,\, \hat \tau_3<\infty] =\bP[\tau_1<\infty,\, \tau_2<\infty]  \bP[\hat \tau_3<\infty] \lesssim \epsilon_k^{-3/2},$$
using independence for the second equality and our previous estimates for the last one. Similarly, 
\begin{align*}
& \E[\phi_{1,3} \psi_2]  \le \sum_{x\in \Z} \bP[\tau_1<\infty, \, \tau_3<\infty \mid S_k=x] \times \bP[\hat \tau_2<\infty,\, S_k=x] \\
&\lesssim \sum_{x\in \Z^5} \frac{1}{(1+\|x\|)\sqrt \epsilon_k}\cdot \frac 1{(1+\|x\| + \sqrt k)^5}\left(\frac{1}{1+\|x\|}+\frac 1{\sqrt{\epsilon_k}}\right)   \lesssim  \frac{1}{\epsilon_k \sqrt k},
\end{align*}
using \eqref{tau132.1} and Lemma \ref{lem.prep.123} for the second inequality. The term $\E[\phi_{2,3} \psi_1]$ is handled similarly. 
We consider now the other cases. One has  
\begin{equation}\label{phi233}
\E[\phi_{2,3} \psi_3] \le \bP[\tau_2\le \tau_3<\infty,\, \hat \tau_3<\infty] + \bP[\tau_3\le \tau_2<\infty,\, \hat \tau_3<\infty].
\end{equation}
By using the Markov property at time $\tau_2$, one can write 
\begin{align*}
& \bP[\tau_2\le \tau_3<\infty,\, \hat \tau_3<\infty] \\
& \le \sum_{x,y\in \Z^5} \E\left[\left(\sum_{i=0}^\infty G(S_i-y+x)\right)\left(\sum_{j=\epsilon_k}^\infty G(S_j) \right) \right]  \bP[\tau_2<\infty, \tilde S_{\tau_2}=y, S_k=x].
\end{align*}
Then applying Lemmas \ref{lem.upconvolG} and \ref{lem.prep.123}, we get 
\begin{align}\label{tau33}
\nonumber &   \E\left[ \left(\sum_{i=0}^{\epsilon_k}  G(S_i-y+x)\right)  \left(\sum_{j=\epsilon_k}^\infty G(S_j) \right)\right]\\ 
& \nonumber  = \sum_{v\in \Z^5} \E\left[\left(\sum_{i=0}^{\epsilon_k} G(S_i-y+x) \right)\1\{S_{\epsilon_k} = v\} \right]  \E\left[\left(\sum_{j=0}^\infty G(S_j+v) \right)\right]\\
\nonumber & \lesssim \sum_{v\in \Z^5} \frac{1}{1+\|v\|}\cdot \left(\sum_{i=0}^{\epsilon_k} p_i(z) G(z-y+x) p_{\epsilon_k-i}(v-z) \right)\\ 
 & \lesssim \sum_{v\in \Z^5}  \frac{1}{1+\|v\|} \frac{1}{(\|v\|+ \sqrt{\epsilon_k})^5} \left(\frac{1}{1+\|v-y+x\|} + \frac{1}{1+\|y-x\|}\right) \lesssim \frac{\epsilon_k^{-1/2}}{ 1+\|y-x\|}. 
\end{align}
Likewise,  
\begin{align}\label{tau33bis}
 \nonumber \E\left[ \left(\sum_{i=\epsilon_k}^\infty G(S_i-y+x)\right)\left(\sum_{j=\epsilon_k}^\infty G(S_j) \right)\right] &\le \sum_{z\in \Z^5} G_{\epsilon_k}(z) \left(\frac{G(z-y+x)}{1+\|z\|} + \frac{G(z)}{1+\|z-y+x\|}\right)\\
&\lesssim \frac{1}{\sqrt{\epsilon_k} (1+\|y-x\|)}. 
\end{align}
Recall now that by \eqref{lem.hit.3}, one has $\bP[\tau_2<\infty] \lesssim \epsilon_k^{-1/2}$. Moreover, from the proof of Lemma \ref{lem.123}, one can deduce that 
$$\E\left[\frac{\1\{\tau_2<\infty\}}{\|\tilde S_{\tau_2}-S_k\|}\right] \lesssim  \frac{1}{\sqrt {k \epsilon_k}}.$$
Combining all these estimates we conclude that 
$$\bP[\tau_2\le \tau_3<\infty,\, \hat \tau_3<\infty] \lesssim \frac 1{\epsilon_k\sqrt k}.$$
We deal next with the second term in the right-hand side of \eqref{phi233}.  Applying the Markov property at time $\tau_3$, and then Lemma \ref{lem.prep.123}, we obtain
\begin{align}\label{tau323.start}
\nonumber & \bP[\tau_3\le \tau_2<\infty,\, \hat \tau_3<\infty]  \\
&\nonumber  \le  \sum_{x,y\in \Z^5}\left(\sum_{i=\epsilon_k}^k \E[G(S_i-y)\1\{S_k=x\}]\right) \bP[\tau_3<\infty, \hat \tau_3<\infty, \tilde S_{\tau_3} = y\mid S_k=x]\\
\nonumber & \lesssim \sum_{x,y\in \Z^5} \frac{1}{(\|x\|+\sqrt k)^5} \left(\frac{1}{1+\|y-x\|}+\frac{1}{\sqrt{\epsilon_k}}\right) \bP[\tau_3<\infty, \hat \tau_3<\infty, \tilde S_{\tau_3} = y\mid S_k=x]\\
\nonumber &\lesssim \sum_{x\in \Z^5} \frac{1}{(\|x\|+\sqrt k)^5}\left( \frac{\bP[\tau_3<\infty, \hat \tau_3<\infty\mid S_k=x]}{\sqrt{\epsilon_k}} + 
\E\left[\frac{\1\{\tau_3<\infty, \hat \tau_3<\infty\}}{1+\|\tilde S_{\tau_3} - x\|}\ \Big| \ S_k=x\right]\right)\\
&  \lesssim \sum_{x\in \Z^5} \frac{1}{(\|x\|+\sqrt k)^5}\left( \frac{1}{\epsilon_k(1+\|x\|)} + \E\left[\frac{\1\{\tau_3<\infty, \hat \tau_3<\infty\}}{1+\|\tilde S_{\tau_3} - x\|}\ \Big| \ S_k=x\right]\right),
\end{align}
using also \eqref{tau33} and \eqref{tau33bis} (with $y=0$) for the last inequality. 
We use now \eqref{hit.ball} and Lemma \ref{hit.ball.overshoot} to remove the denominator in the last expectation above. 
Define for $r\ge 0$, and $x\in \Z^5$, 
$$\eta_r(x):=\inf\{n\ge 0\ :\ \|\tilde S_n -x\|\le r\}.$$
On the event when $r/2\le \|\tilde S_{\eta_r(x)} -x\| \le r$, one applies the Markov property at time $\eta_r(x)$, and we deduce from \eqref{hit.ball} and Lemma 
\ref{hit.ball.overshoot} that 
\begin{align*}
&\E\left[ \frac{\1\{\tau_3<\infty, \, \hat \tau_3<\infty\}}{1+\|\tilde S_{\tau_3} - x\|}\ \Big|\ S_k=x\right] \le \frac{\bP[\tau_3<\infty, \, \hat \tau_3<\infty\mid S_k=x]}{1+\|x\| } \\
&\qquad  + \sum_{i=0}^{\log_2\|x\|} 
\frac{\bP\left[\tau_3<\infty, \, \hat \tau_3<\infty,\, 2^i \le \|\tilde S_{\tau_3} -x\| \le 2^{i+1}\mid S_k=x\right]}{2^i} \\
& \lesssim  \frac{1}{\sqrt{\epsilon_k}(1+\|x\|^2)} + \sum_{i=0}^{\log_2\|x\|} 
\frac{\bP\left[\eta_{2^{i+1}}(x)\le \tau_3<\infty, \, \hat \tau_3<\infty\mid S_k=x\right]}{2^i} \\
& \lesssim \frac{\epsilon_k^{-1/2}}{1+\|x\|^2} +  \frac{\bP[\hat \tau_3<\infty]}{1+\|x\|^3} + \sum_{i=0}^{\log_2\|x\|} \frac{2^{2i}}{1+\|x\|^3} \max_{\|z\|\ge 2^i} \bP_{0,0,z}\left[H_{\RR[\epsilon_k,\infty)}<\infty,  \tilde H_{\RR_\infty}<\infty \right],
\end{align*}
where in the last probability, $H$ and $\tilde H$ refer to hitting times by  
two independent walks, independent of $S$, starting respectively from the origin and from $z$. 
Then it follows from \eqref{tau33} and \eqref{tau33bis} that 
\begin{equation}\label{remove.denominator}
\E\left[ \frac{\1\{\tau_3<\infty,  \hat \tau_3<\infty\}}{1+\|\tilde S_{\tau_3} - x\|}\ \Big|\ S_k=x\right] \lesssim \frac{1}{\sqrt{\epsilon_k}(1+\|x\|^2)}.
\end{equation}
Combining this with \eqref{tau323.start}, it yields that 
\begin{align*} 
\bP[\tau_2\le \tau_3<\infty,  \hat \tau_3 <\infty]  \lesssim  \frac 1{\epsilon_k\sqrt k}.  
\end{align*}
The terms $\E[\phi_{1,3} \psi_3]$ and $\E[\phi_{1,3}\psi_1]$ are entirely similar, and we omit repeating the proof. Thus it only remains to consider the terms $\E[\phi_{2,3} \psi_2]$ and $\E[\phi_{1,2} \psi_2]$. Since they are also similar we only give the details for the former. 
We start again by writing
\begin{equation}\label{tau232}
\E[\phi_{2,3} \psi_2]\le \bP[\tau_2\le \tau_3<\infty, \, \hat \tau_2 <\infty] + \bP[\tau_3\le \tau_2<\infty, \, \hat \tau_2<\infty].
\end{equation}
Then one has 
\begin{align}\label{Sigmai} 
&\bP[\tau_3\le \tau_2<\infty, \, \hat \tau_2 <\infty] \\
\nonumber  \le & \sum_{x,y\in \Z^5}\E\left[\left(\sum_{i=\epsilon_k}^k G(S_i-y)\right) \left( \sum_{j=0}^{k-\epsilon_k} G(S_j-x)\right)  \1\{S_k=x\} \right]
  \bP[\tau_3<\infty, \tilde S_{\tau_3}=y\mid  S_k=x]\\
\nonumber \le & \sum_{x,y\in \Z^5} \left(\sum_{i=\epsilon_k}^k \sum_{j=0}^{k-\epsilon_k}\sum_{z,w\in \Z^5} \bP[S_i= z, S_j=w, S_k=x] G(z-y)G(w-x)\right)  \\
\nonumber & \qquad \times \bP[\tau_3<\infty, \tilde S_{\tau_3}=y\mid  S_k=x]. 
\end{align}
Now for any $x,y\in \Z^5$, 
\begin{align*}
& \Sigma_1(x,y):= \sum_{i=\epsilon_k}^{k-\epsilon_k} \sum_{j=\epsilon_k}^{k-\epsilon_k} \sum_{z,w\in \Z^5} \bP[S_i= z,\, S_j=w,\, S_k=x] G(z-y)G(w-x)\\ 
 & \le   2\sum_{i=\epsilon_k}^{k-\epsilon_k}  \sum_{z\in \Z^5}p_i(z) G(z-y) \left(\sum_{j=i}^{k-\epsilon_k} \sum_{w\in \Z^5} p_{j-i}(w-z)G(w-x) p_{k-j}(x-w) \right)\\
& =  2\sum_{i=\epsilon_k}^{k-\epsilon_k}  \sum_{z\in \Z^5}p_i(z) G(z-y) \left(\sum_{j=\epsilon_k}^k \sum_{w\in \Z^5} p_j(w) G(w) p_{k-i-j}(w+x-z) \right)\\
& \stackrel{\text{Lemma }  \ref{lem.prep.123}}{\lesssim}  \sum_{i=\epsilon_k}^{k-\epsilon_k}  \sum_{z\in \Z^5} \frac{p_i(z)G(z-y)}{(\|z-x\| + \sqrt{k-i})^5} \left(\frac 1{1+\|z-x\|} + \frac{1}{\sqrt{k-i}}\right)\\
 & \stackrel{\text{Lemma }  \ref{lem.prep.ijl}}{\lesssim}    \left\{
\begin{array}{ll}
\frac{1}{k^{5/2}}\left( \frac 1{1+\|x\|^2}  + \frac{1}{\varepsilon_k}\right) + \frac{1}{k^{3/2}\epsilon_k^{3/2}(1+\|y-x\|)}  & \text{if }\|x\|\le \sqrt k\\
\frac{1}{\|x\|^5\epsilon_k}\left(1+ \frac{k}{\sqrt{\epsilon_k}(1+\|y-x\|)}\right) &\text{if }\|x\|>\sqrt k.
\end{array}
\right. 
\end{align*}
We also have 
\begin{align*}
& \Sigma_2(x,y) :=  \sum_{i=k-\epsilon_k}^k \sum_{j=0}^{k-\epsilon_k} \sum_{z,w\in \Z^5} \bP[S_i= z, S_j=w, S_k=x] G(z-y)G(w-x)\\ 
 =& \sum_{i=k-\epsilon_k}^k \sum_{j=0}^{k-\epsilon_k} \sum_{z,v,w\in \Z^5} \bP[S_j= w, S_{k-\epsilon_k} = v, S_i=z, S_k=x] G(z-y)G(w-x)\\
 =& \sum_{v\in \Z^5} \left(\sum_{j=0}^{k-\epsilon_k} \sum_{w\in \Z^5} p_j(w) p_{k-\epsilon_k-j}(v-w)  G(w-x)\right) \left(\sum_{i=0}^{\epsilon_k} \sum_{z\in \Z^5} p_i(z-v)
p_{\epsilon_k-i}(x-z)G(z-y)\right),
\end{align*}
and applying then Lemmas \ref{lem.prep.123} and \ref{lem.prep.ijl2}, gives 
\begin{align*}
& \Sigma_2(x,y) \\
\lesssim &  \sum_{v\in \Z^5} \frac{1}{(\|v\|+\sqrt k)^5}\left(\frac{1}{1+\|x-v\|} + \frac{1}{ 1+\|x\|}\right)\frac 1{(\|x-v\|+ \sqrt{\epsilon_k} )^5} \left(\frac{1}{1+\|y-x\|}+\frac 1{1+\|y-v\|}\right)\\
 \lesssim &  \left\{
\begin{array}{ll}
\frac{1}{k^2\epsilon_k} \left( \frac 1{\sqrt {\epsilon_k}} + \frac{1}{1+\|x\|} +\frac 1{1+\|y-x\|}+\frac{\sqrt{\epsilon_k}}{(1+\|x\|) (1+\|y-x\|)}\right) & \quad \text{if }\|x\|\le \sqrt k\\
\frac{\log (\frac{\|x\|}{\sqrt{\epsilon_k}})}{\|x\|^5\sqrt{\epsilon_k}} \left(\frac{1}{1+\|y-x\|} +  \frac 1{\sqrt k} \right)  &\quad \text{if }\|x\|>\sqrt k.
\end{array}
\right.
\end{align*}
Likewise, by reversing time, one has  
\begin{align*}
& \Sigma_3(x,y):= \sum_{i=\epsilon_k}^k \sum_{j=0}^{\epsilon_k} \sum_{z,w\in \Z^5} \bP[S_i= z, S_j=w, S_k=x] G(z-y)G(w-x)\\ 
 =& \sum_{i=0}^{k-\epsilon_k} \sum_{j=k-\epsilon_k}^k \sum_{z,v,w\in \Z^5} \bP[S_i = z-x, S_{k-\epsilon_k} = v-x, S_j= w-x,  S_k=-x] G(z-y)G(w-x)\\
 =& \sum_{v\in \Z^5} \left(\sum_{i=0}^{k-\epsilon_k} \sum_{z\in \Z^5} p_i(z-x) p_{k-\epsilon_k-i}(v-z)  G(z-y)\right) \left(\sum_{j=0}^{\epsilon_k} \sum_{w\in \Z^5} p_j(w-v)
p_{\epsilon_k-j}(w)G(w-x)\right)  \\
\lesssim &  \sum_{v\in \Z^5} \frac{1}{(\|v-x\|+\sqrt k)^5}\left(\frac{1}{1+\|y-v\|} + \frac{1}{ 1+\|y-x\|}\right)\frac 1{(\|v\|+ \sqrt{\epsilon_k} )^5} \left(\frac{1}{1+\|x\|}+\frac 1{1+\|x-v\|}\right),
\end{align*}
and then a similar argument as in the proof of Lemma \ref{lem.prep.ijl2} gives the same bound for $\Sigma_3(x,y)$ as for $\Sigma_2(x,y)$.  
Now recall that \eqref{Sigmai} yields   
$$\bP[\tau_3\le \tau_2<\infty,  \hat \tau_2 <\infty] \le \sum_{x,y\in \Z^5} \left(\Sigma_1(x,y) + \Sigma_2(x,y) + \Sigma_3(x,y)\right)  \bP[\tau_3<\infty, \tilde S_{\tau_3}=y\mid  S_k=x]. $$
Recall also that by \eqref{lem.hit.2}, 
$$\bP[\tau_3<\infty\mid S_k=x] \lesssim  \frac{1}{1+\|x\|},$$
and
\begin{align*}
\E\left[\frac{\1\{\tau_3<\infty\}}{1+\|\tilde S_{\tau_3} -x \| }\, \Big|\, S_k=x\right] 
 \le \sum_{y\in \Z^5} \frac{G(y) G(y-x)}{1+\|y-x\|} 
\lesssim \frac{1}{1+\|x\|^2}.  
\end{align*}
 Furthermore, for any $\alpha\in\{1,2,3\}$, and any $\beta\ge 6$, 
 $$\sum_{\|x\|\le \sqrt k} \frac{1}{1+\|x\|^\alpha} \lesssim k^{\frac{5-\alpha}{2}},  \quad    \sum_{\|x\|\ge \sqrt k} \frac {\log (\frac{\|x\|}{\sqrt{\epsilon_k}})} {\|x\|^{\beta}} \le \sum_{\|x\|\ge \sqrt{\epsilon_k}} \frac {\log (\frac{\|x\|}{\sqrt{\epsilon_k}})} {\|x\|^{\beta}}\lesssim  \epsilon_k^{\frac{5-\beta}{2}}.$$
 Putting all these pieces together we conclude that 
 $$\bP[\tau_3\le \tau_2<\infty, \, \hat \tau_2 <\infty]  \lesssim  \epsilon_k^{-3/2}. $$
We deal now with the other term in \eqref{tau232}. As previously, we first write using the Markov property, and then using \eqref{lem.hit.1} and Lemma \ref{lem.upconvolG},  
\begin{align*} 
\bP[\tau_2\le \tau_3<\infty, \, \hat \tau_2 <\infty]  \le \E\left[ \frac{\1\{\tau_2<\infty, \, \hat \tau_2<\infty\}}{1+\|\tilde S_{\tau_2} - S_k\|}\right]. 
\end{align*}
Then using \eqref{hit.ball} and Lemma \ref{hit.ball.overshoot} one can handle the denominator in the last expectation, the same way as for \eqref{remove.denominator}, 
and we conclude similarly that 
\begin{align*} 
\bP[\tau_2\le \tau_3<\infty, \, \hat \tau_2 <\infty]  \lesssim  \epsilon_k^{-3/2}.  
\end{align*}
This finishes the proof of Lemma \ref{lem.ijl}.
\end{proof}

%%%%%%%%%%%%%%%%%%%%%%%%%%%%%%%%%%%%%%%%%%%%%%%%%%%%%%%%%%%%%%%%%%%%%%%%%%%%%%%%%

\section{Proof of Propositions \ref{prop.ij0} and \ref{prop.phipsi.1}}
For the proof of these propositions we shall need the following estimate.  
\begin{lemma}\label{lem.prep.0ij}
One has for all $x,y\in \Z^5$, 
\begin{align*}
\sum_{i=k-\epsilon_k}^k & \E\left[G(S_i-y) \1\{S_k=x\}\right] \\
& \lesssim  \epsilon_k \left(\frac{\log (2+\frac{\|y-x\|}{\sqrt{\epsilon_k}})}{(\|x\|+\sqrt k)^5(\|y-x\| + \sqrt{\epsilon_k})^3} 
+ \frac{\log (2+\frac{\|y\|}{\sqrt{k}})}{(\|x\|+\sqrt{\epsilon_k})^5(\|y\| + \sqrt{k})^3}\right). 
\end{align*}
\end{lemma}
\begin{proof}
One has using \eqref{pn.largex} and \eqref{Green}, 
\begin{align*}
&\sum_{i=k-\epsilon_k}^k \E\left[G(S_i-y)\1\{S_k=x\}\right]  = \sum_{i=k-\epsilon_k}^k \sum_{z\in \Z^5} p_i(z) G(z-y) p_{k-i}(x-z)\\
& \lesssim \sum_{z\in \Z^5} \frac{\epsilon_k}{(\|z\| + \sqrt k)^5(1+\|z-y\|^3) (\|x-z\| +\sqrt{\epsilon_k})^5} \\
& \lesssim \frac{1}{\epsilon_k^{3/2}(\|x\|+\sqrt k)^5} \sum_{\|z-x\|\le \sqrt{\epsilon_k}}\frac 1{1+\|z-y\|^3} \\
 & \quad + \frac{\epsilon_k}{(\|x\|+\sqrt k)^5} \sum_{\sqrt{\epsilon_k}\le \|z-x\|\le \frac{\|x\|}{2}} \frac{1}{(1+\|z-y\|^3)(1+\|z-x\|^5)} \\
 &\quad + \frac{\epsilon_k}{(\|x\|+\sqrt {\epsilon_k})^5} \sum_{\|z-x\|\ge \frac{\|x\|}{2}}\frac{1}{(\|z\|+\sqrt k)^5(1+\|z-y\|^3)}. 
\end{align*}
Then it suffices to observe that 
$$ \sum_{\|z-x\|\le \sqrt{\epsilon_k}}\frac 1{1+\|z-y\|^3} \lesssim  \frac{\epsilon_k^{5/2}}{(\|y-x\| + \sqrt{\epsilon_k})^3},$$
$$ \sum_{\sqrt{\epsilon_k}\le \|z-x\|\le \frac{\|x\|}{2}} \frac{1}{(1+\|z-y\|^3)(1+\|z-x\|^5)} \lesssim 
 \frac{\log(2+\frac{\|y-x\|}{\sqrt{\epsilon_k}})}{(\|y-x\| + \sqrt{\epsilon_k})^3}, $$
$$ \sum_{z\in \Z^5}\frac{1}{(\|z\|+\sqrt k)^5(1+\|z-y\|^3)} \lesssim  \frac{\log (2+\frac{\|y\|}{\sqrt{k}})}{(\|y\| + \sqrt{k})^3}. $$
\end{proof}

\begin{proof}[Proof of Proposition \ref{prop.ij0} (i)]
This part is the easiest: it suffices to observe that $\phi_{1,2}$ 
is a sum of one term which is independent of $Z_k\psi_0$ and another one, whose expectation is negligible. To be more precise, define 
$$\phi_{1,2}^0 := \bP\left[H_{\RR[-\epsilon_k,\epsilon_k]}^+ = \infty,\, H^+_{\RR(-\infty,-\epsilon_k-1]}<\infty, \, H^+_{\RR[\epsilon_k+1,k-\epsilon_k-1]}<\infty \mid S\right],$$
and note that $Z_0\phi_{1,2}^0$ is independent of $Z_k\psi_0$. It follows that  
$$|\cov(Z_0\phi_{1,2},Z_k\psi_0)| =|\cov(Z_0(\phi_{1,2}-\phi_{1,2}^0),Z_k\psi_0)|\le \bP\left[\tau_1<\infty, \, \tau_*<\infty\right],$$
with $\tau_1$ and $\tau_*$  the hitting times respectively of $\RR(-\infty,-\epsilon_k]$ and $\RR[k-\epsilon_k,k]$ by another walk $\tilde S$ starting from the origin, independent of $S$.  
Now, using \eqref{pn.largex}, we get 
\begin{align*}
\bP[\tau_1\le \tau_*<\infty] & \le \E\left[\1\{\tau_1<\infty \} \left(\sum_{i=k-\epsilon_k}^k G(S_i-\tilde S_{\tau_1})\right)\right] \\
& \le \sum_{y\in \Z^5} \left(\sum_{z\in \Z^5} \sum_{i=k-\epsilon_k}^k p_i(z) G(z-y) \right)\bP[\tau_1<\infty, \, \tilde S_{\tau_1} = y]\\
& \lesssim  \frac{\epsilon_k}{k^{3/2}} \, \bP[\tau_1<\infty] \stackrel{\eqref{lem.hit.3}}{\lesssim}  \frac{\sqrt{\epsilon_k}}{k^{3/2}}. 
\end{align*}
Likewise, using now Lemma \ref{lem.upconvolG}, 
\begin{align*}
\bP[\tau_*\le \tau_1<\infty] & \le \E\left[\1\{\tau_*<\infty \} \left(\sum_{i=\epsilon_k}^\infty G(S_{-i}-\tilde S_{\tau_*})\right)\right] \\
& \le \sum_{y\in \Z^5} \left(\sum_{z\in \Z^5} G_{\epsilon_k}(z) G(z-y) \right)\bP[\tau_*<\infty, \, \tilde S_{\tau_*} = y]\\
& \lesssim \frac{1}{\sqrt{\epsilon_k}}\, \bP[\tau_*<\infty] \lesssim  \frac{\sqrt{\epsilon_k}}{k^{3/2}},
\end{align*}
and the first part of (i) follows. But since $Z_0$ and $Z_k$ have played no role here, the same computation gives the result for the covariance between $Z_0\phi_0$ and $Z_k \psi_{2,3}$ as well. 
\end{proof}

\begin{proof}[Proof of Proposition \ref{prop.ij0} (ii)-(iii)]
These parts are more involved. Since they are entirely similar, we only prove (iii), and as for (i) we only give the details for the covariance between $Z_0 \phi_{2,3}$ and $Z_k\psi_0$, since $Z_0$ and $Z_k$ will not play any role here. We define similarly as in the proof of (i), 
$$\phi_{2,3}^0 := \bP\left[H_{\RR[-\epsilon_k,\epsilon_k]}^+ = \infty,\, H^+_{\RR[\epsilon_k,k-\epsilon_k]}<\infty, \, H^+_{\RR[k+\epsilon_k,\infty)}<\infty \mid S\right],$$
but observe that this time, the term $\phi_{2,3}^0$ is no more independent of $\psi_0$. This entails some additional difficulty, on which we shall come back later, 
but first we show that one can indeed replace $\phi_{2,3}$ by $\phi_{2,3}^0$ in the computation of the covariance. For this, denote 
respectively by $\tau_2$, $\tau_3$, $\tau_*$ and $\tau_{**}$ the hitting times of $\RR[\epsilon_k,k]$, $\RR[k,\infty)$, $\RR[k-\epsilon_k,k]$, and $\RR[k,k+\epsilon_k]$ by $\tilde S$. One has 
\begin{align*}
\E[|\phi_{2,3} - \phi_{2,3}^0|]& \le   \bP[\tau_2<\infty, \, \tau_{**}<\infty] +\bP[\tau_3<\infty, \, \tau_*<\infty].
\end{align*}
Using  \eqref{pn.largex}, \eqref{lem.hit.1} and Lemma \ref{lem.upconvolG}, we get
\begin{align*}
 \bP[\tau_*\le  \tau_3<\infty] & \le \E\left[ \frac{\1\{\tau_*<\infty\}}{1+\|\tilde S_{\tau_*} - S_k\|}\right] \le \sum_{i=k-\epsilon_k}^k \E\left[ \frac{G(S_i)}{1+\|S_i-S_k\|}\right] \\
&\lesssim \sum_{i=k-\epsilon_k}^k \E\left[ \frac{G(S_i)}{1+\sqrt{k-i} }\right]  \lesssim  \sum_{z\in \Z^5} \sum_{i=k-\epsilon_k}^k \frac{p_i(z) G(z)}{1+\sqrt{k-i}}\\
 & \lesssim \sqrt{\epsilon_k} \sum_{z\in \Z^5}\frac 1{(\|z\| + \sqrt k)^5} G(z) \lesssim \frac{\sqrt{\epsilon_k}}{k^{3/2}}.   
\end{align*}
Next, applying Lemma \ref{lem.prep.0ij}, we get 
\begin{align*}
& \bP[\tau_3\le  \tau_*<\infty] \\
& \le \sum_{x,y\in \Z^5} \E\left[\left(\sum_{i=k-\epsilon_k}^k G(S_i-y)\right)\1\{S_k=x\}\right]  \bP[\tau_3<\infty,  \tilde S_{\tau_3} = y\mid S_k=x]\\
& \lesssim \epsilon_k \sum_{x\in \Z^5} \left(  \E\left[\frac{\1\{\tau_3<\infty\}\log (2+\frac{\|\tilde S_{\tau_3}-x\|}{\sqrt{\epsilon_k}})}{(\|x\|+\sqrt k)^5(\sqrt{\epsilon_k}+\|\tilde S_{\tau_3} - x\|)^3}\, \Big|\, S_k=x \right]  \right. \\
& \qquad \left. + \E\left[\frac{\1\{\tau_3<\infty\}\log(2+\frac{\|\tilde S_{\tau_3}\|}{\sqrt k})}{(\|x\|+\sqrt{\epsilon_k})^5(\sqrt{k}+\|\tilde S_{\tau_3}\|)^3}\, \Big|\, S_k=x \right] \right).
\end{align*}
Moreover, 
\begin{align*}
\E\left[\frac{\1\{\tau_3<\infty\}\log (2+\frac{\|\tilde S_{\tau_3}-x\|}{\sqrt{\epsilon_k}})}{(\sqrt{\epsilon_k}+\|\tilde S_{\tau_3} - x\|)^3}\, \Big|\, S_k=x \right] 
& \stackrel{\eqref{lem.hit.1}}{\le} \sum_{y\in \Z^5} \frac{G(y) G(y-x)
\log (2+\frac{\|y-x\|}{\sqrt{\epsilon_k}})}{ (\sqrt{\epsilon_k}+\|y - x\|)^3} \\
& \lesssim \frac{1}{\sqrt{\epsilon_k}(1+\|x\|)^3}, 
\end{align*}
and 
\begin{align*}
\E\left[\frac{\1\{\tau_3<\infty\}\log(2+\frac{\|\tilde S_{\tau_3}\|}{\sqrt k})}{(\sqrt{k}+\|\tilde S_{\tau_3}\|)^3}\, \Big|\, S_k=x \right] & \stackrel{\eqref{lem.hit.1}}{\le} \sum_{y\in \Z^5} \frac{G(y) G(y-x)\log(2+\frac{\|y\|}{\sqrt k})}{ (\sqrt{k}+\|y \|)^3} \\
& \lesssim \frac{1}{\sqrt{k}(1+\|x\|)(\sqrt k + \|x\|)^2}. 
\end{align*}
Furthermore, it holds  
$$\sum_{x\in \Z^5}   \frac{1}{(\|x\|+\sqrt k)^5(1+\|x\|)^3} \lesssim \frac{1}{k^{3/2}},$$ 
$$\sum_{x\in \Z^5}   \frac{1}{(\|x\|+\sqrt{\epsilon_k})^5(1+\|x\|)(\sqrt k + \|x\|)^2} \lesssim \frac{1}{\sqrt{k\epsilon_k}},$$
which altogether proves that 
$$ \bP[\tau_3\le  \tau_*<\infty] \lesssim \frac{\sqrt{\epsilon_k}}{k^{3/2}}.$$
Likewise, 
\begin{align*}
 \bP[\tau_2\le  \tau_{**}<\infty]  \le \sum_{x,y\in \Z^5} \E\left[\sum_{i=0}^{\epsilon_k} G(S_i-y+x) \right]  \bP[\tau_2<\infty,  \tilde S_{\tau_2} = y, S_k=x],
\end{align*}
and using \eqref{Green}, we get 
\begin{align*}
& \E\left[\sum_{i=0}^{\epsilon_k} G(S_i-y+x) \right] = \sum_{i=0}^{\epsilon_k} \sum_{z\in \Z^5} p_i(z) G(z-y+x)\\
& \lesssim \sum_{\|z\|\le \sqrt{\epsilon_k}} G(z) G(z-y+x) + \epsilon_k\, \sum_{\|z\|\ge \sqrt{\epsilon_k}} \frac{G(z-y+x)}{\|z\|^5}  \\
&\lesssim   \frac{\epsilon_k}{(\|y-x\| +\sqrt{\epsilon_k})^2(1+\|y-x\|)}  + \epsilon_k \, \frac{\log\left(2+\frac{\|y-x\|}{\sqrt{\epsilon_k}}\right) }{(\|y-x\| + \sqrt{\epsilon_k})^3}\\
& \lesssim \epsilon_k \, \frac{\log\left(2+\frac{\|y-x\|}{\sqrt{\epsilon_k}}\right) }{(\|y-x\| + \sqrt{\epsilon_k})^2(1+\|y-x\|)}. 
\end{align*}
Therefore, using the Markov property, 
\begin{align*}
&\bP[\tau_2\le  \tau_{**}<\infty]  \lesssim  \epsilon_k \cdot \E\left[ \frac{\log\left(2+\frac{\|\tilde S_{\tau_2}-S_k\|}{\sqrt{\epsilon_k}}\right) \cdot \1\{\tau_2<\infty\}}{(\|\tilde S_{\tau_2}-S_k\| + \sqrt{\epsilon_k})^2(1+\|\tilde S_{\tau_2} - S_k\|)} \right]\\
&\lesssim \epsilon_k \sum_{i=\epsilon_k}^k \E[G(S_i)] \cdot \E\left[\frac{\log\left(2+\frac{\|S_{k-i}\|}{\sqrt{\epsilon_k}}\right) }{(\|S_{k-i}\| + \sqrt{\epsilon_k})^2(1+\| S_{k-i}\|)} \right].
\end{align*}
Furthermore, using \eqref{pn.largex} we obtain after straightforward computations, 
$$\E\left[\frac{\log\left(2+\frac{\|S_{k-i}\|}{\sqrt{\epsilon_k}}\right) }{(\|S_{k-i}\| + \sqrt{\epsilon_k})^2(1+\| S_{k-i}\|)} \right] \lesssim \frac{\log\left(2+\frac{k-i}{\epsilon_k}\right) }{\sqrt{k-i}(\epsilon_k + k-i)},$$ 
and using in addition \eqref{exp.Green}, we conclude that 
$$\bP[\tau_2\le  \tau_{**}<\infty]   \lesssim  \frac{\sqrt{\epsilon_k}}{k^{3/2}}\cdot \log(\frac k{\epsilon_k}).$$ 
Similarly, using Lemma \ref{lem.prep.123} we get 
\begin{align*}
&\bP[\tau_{**} \le \tau_2<\infty] \\
 & = \sum_{x,y\in \Z^5} \bP[\tau_{**}<\infty,\, \tilde S_{\tau_{**}} = y \mid S_k=x] \cdot \E\left[\sum_{i=\epsilon_k}^k G(S_i-y) \1\{S_k=x\}\right] \\
& \lesssim \sum_{x\in\Z^5} \frac{1}{(\|x\| + \sqrt{k})^5} \left(\E\left[\frac{1\{\tau_{**}<\infty\}}{ 1+\|\tilde S_{\tau_{**}}-x\|}\, \Big|\,  S_k=x\right] + \frac{\bP[\tau_{**}<\infty\mid S_k=x]}{\sqrt{\epsilon_k}}\right).
\end{align*}
Moreover, one has 
\begin{align*}
\bP[\tau_{**}<\infty\mid S_k=x] &\le \sum_{i=0}^{\epsilon_k} \E[G(S_i+x)]\lesssim \sum_{i=0}^{\epsilon_k} \sum_{z\in \Z^5} \frac{1}{(1+\|z\| + \sqrt i)^5(1+\|z+x\|^3)} \\ 
&\lesssim \sum_{\|z\| \le \sqrt{\epsilon_k}} \frac{1}{(1+\|z\|^3)(1+\|z+x\|^3)} + \sum_{\|z\|\ge \sqrt{\epsilon_k}} \frac{\epsilon_k}{\|z\|^5(1+\|z+x\|^3)}\\
&\lesssim \frac{\epsilon_k \log(2+\frac{\|x\|}{\sqrt{\epsilon_k}})}{(\sqrt{\epsilon_k}+\|x\|)^2(1+\|x\|)}, 
\end{align*}
and likewise 
\begin{align*}
\E\left[\frac{1\{\tau_{**}<\infty\}}{ 1+\|\tilde S_{\tau_{**}}-x\|}\, \Big|\,  S_k=x\right]  &\le  \sum_{i=0}^{\epsilon_k} \sum_{z\in \Z^5} \frac{1}{(1+\|z\| + \sqrt i)^5(1+\|z-x\|^3)(1+\|z\|)} \\ 
&\lesssim \sum_{\|z\| \le \sqrt{\epsilon_k}} \frac{1}{(1+\|z\|^4)(1+\|z-x\|^3)} + \sum_{\|z\|\ge \sqrt{\epsilon_k}} \frac{\epsilon_k}{\|z\|^6(1+\|z-x\|^3)}\\
&\lesssim \frac{\sqrt{\epsilon_k} }{(\|x\|+\sqrt{\epsilon_k})(1+\|x\|^2)}. 
\end{align*}
Then it follows as above that 
$$\bP[\tau_{**} \le \tau_2<\infty] \lesssim \frac{\sqrt{\epsilon_k}}{k^{3/2}}\cdot \log(\frac k{\epsilon_k}).$$ 
In other words we have proved that 
$$\E[|\phi_{2,3} - \phi_{2,3}^0|]\lesssim  \frac{\sqrt{\epsilon_k}}{k^{3/2}}\cdot \log(\frac k{\epsilon_k}).$$
We then have to deal with the fact that $Z_0\phi_{2,3}^0$ is not really independent of $Z_k\psi_0$. Therefore, we introduce the new random variables 
$$\tilde Z_k := \1\{S_i \neq S_k \ \forall i=k+1,\dots,\epsilon'_k\}, \ \tilde \psi_0:=\bP_{S_k}\left[H^+_{\RR[k-\epsilon'_k,k+\epsilon'_k]}=\infty\mid S\right],$$
where $(\epsilon'_k)_{k\ge 0}$ is another sequence of integers, whose value will be fixed later. For the moment we only assume that it satisfies $\epsilon'_k\le \epsilon_k/4$, for all $k$. 
One has by \eqref{Green} and \eqref{lem.hit.3}, 
\begin{equation}\label{tildepsi0}
\E[|Z_k\psi_0 - \tilde Z_k\tilde \psi_0|] \lesssim \frac{1}{\sqrt{\epsilon'_k}}.
\end{equation}
Furthermore, for any $y\in \Z^5$,
\begin{align}\label{cov.230}
&  \E\left[\phi_{2,3}^0 \mid S_{k+\epsilon_k} -S_{k-\epsilon_k}= y\right]  = \sum_{x\in \Z^5} \E\left[\phi_{2,3}^0 \1\{S_{k-\epsilon_k}=x\}\mid S_{k+\epsilon_k} -S_{k-\epsilon_k}= y\right]\\
 \nonumber & \le \sum_{x\in \Z^5}\bP\left[\tilde \RR_\infty \cap \RR[\epsilon_k,k-\epsilon_k]\neq \emptyset,\, \tilde \RR_\infty \cap (x+y+\hat \RR_\infty)\neq \emptyset,\, S_{k-\epsilon_k}=x\right],
\end{align}
where in the last probability, $\tilde \RR_\infty$ and $\hat \RR_\infty$ are the ranges of two independent walks, independent of $S$, starting from the origin. 
Now $x$ and $y$ being fixed, define 
$$\tau_1:=\inf\{n\ge 0 : \tilde S_n\in \RR[\epsilon_k,k-\epsilon_k]\}, \ \tau_2:= \inf\{n\ge 0 :  \tilde S_n\in (x+y + \hat \RR_\infty)\}.$$
Applying  \eqref{lem.hit.1} and the Markov property we get 
\begin{align*}
&\bP[\tau_1\le \tau_2<\infty,\, S_{k-\epsilon_k} = x] \le \E\left[\frac{\1\{\tau_1<\infty,\, S_{k-\epsilon_k}=x\}}{1+\|\tilde S_{\tau_1} - (x+y)\|}\right] \\
&  \le \sum_{i=\epsilon_k}^{k-\epsilon_k} \sum_{z\in \Z^5}\frac{p_i(z) G(z) p_{k-\epsilon_k-i}(x-z)}{1+\|z-(x+y)\|} \\
&\lesssim \frac{1}{(\|x\|+\sqrt k)^5}\left(\frac{1}{\sqrt{\epsilon_k} (1+\|x+y\|)} + \frac{1}{1+\|x\|^2}\right), 
\end{align*}
using also similar computations as in the proof of Lemma \ref{lem.prep.123} for the last inequality. 
It follows that for some constant $C>0$, independent of $y$, 
$$\sum_{x\in \Z^5} \bP[\tau_1\le \tau_2<\infty,\, S_{k-\epsilon_k} = x] \lesssim \frac{1}{\sqrt{k\epsilon_k}}.$$
On the other hand, by Lemmas \ref{lem.prep.123} and \ref{lem.simplehit}, 
\begin{align*}
\bP[\tau_2\le \tau_1<\infty,\, S_{k-\epsilon_k} = x] & \lesssim \frac{1}{( \| x \| +\sqrt k)^5}  \left(\E\left[\frac{\1\{\tau_2<\infty\} }{1+\|\tilde S_{\tau_2} - x\|}\right]  + \frac{\bP[\tau_2<\infty]}{\sqrt{\epsilon_k} }\right) \\
& \lesssim \frac{1}{(\|x\|+\sqrt k)^5}\left(\frac{1}{\sqrt{\epsilon_k} (1+\|x+y\|)} + \frac{1}{1+\|x\|^2}\right),
\end{align*}
and it follows as well that 
$$\sum_{x\in \Z^5} \bP[\tau_2\le \tau_1<\infty, S_{k-\epsilon_k} = x] \lesssim \frac{1}{\sqrt{k\epsilon_k}}.$$
Coming back to \eqref{cov.230}, we deduce that 
\begin{equation}\label{phi230.cond}
  \E\left[\phi_{2,3}^0 \mid S_{k+\epsilon_k} -S_{k-\epsilon_k}= y\right] \lesssim  \frac{1}{\sqrt{k\epsilon_k}},  
  \end{equation}
 with an implicit constant independent of $y$. 
Together with \eqref{tildepsi0}, this gives 
\begin{align*} 
& \E\left[\phi_{2,3}^0|Z_k\psi_0-\tilde Z_k\tilde \psi_0|\right] \\
& = \sum_{y\in \Z^5}  \E\left[\phi_{2,3}^0 \mid S_{k+\epsilon_k} -S_{k-\epsilon_k}= y\right]\cdot \E\left[|Z_k\psi_0-\tilde Z_k\tilde \psi_0|\1\{S_{k+\epsilon_k}-S_{k-\epsilon_k} = y\}\right] \\
&\lesssim \frac{1}{\sqrt{k\epsilon_k\epsilon'_k}}.
\end{align*}
Thus at this point we have shown that 
$$ \cov(Z_0\phi_{2,3},Z_k\psi_0) = \cov(Z_0\phi_{2,3}^0,\tilde Z_k\tilde \psi_0) +\OO\left(\frac{\sqrt{\epsilon_k}}{k^{3/2}}\cdot \log(\frac k{\epsilon_k}) + \frac{1}{\sqrt{k\epsilon_k\epsilon'_k}}\right).$$
Note next that 
\begin{align*}
 \cov(Z_0\phi_{2,3}^0,\tilde Z_k\tilde \psi_0) & = \sum_{y,z\in \Z^5} \E\left[Z_0\phi_{2,3}^0 \mid S_{k+\epsilon_k} - S_{k-\epsilon'_k} =y\right] \\
 & \times \E\left[\tilde Z_k \tilde \psi_0\1\{S_{k+\epsilon'_k}-S_{k-\epsilon'_k}=z\}\right] \left(p_{\epsilon_k-\epsilon'_k}(y-z) - p_{\epsilon_k+\epsilon'_k}(y)\right). 
\end{align*}
Moreover, one can show exactly as \eqref{phi230.cond} that uniformly in $y$, 
$$\E\left[\phi_{2,3}^0 \mid S_{k+\epsilon_k} - S_{k-\epsilon'_k} =y\right] \lesssim \frac{1}{\sqrt{k\epsilon_k}}.$$
Therefore by using also \eqref{Sn.large} and Theorem \ref{LCLT}, we see that 
\begin{align*}
&  |\cov(Z_0\phi_{2,3}^0,\tilde Z_k\tilde \psi_0) | \\
&  \lesssim \frac{1}{\sqrt{k\epsilon_k}}  \sum_{\|y\| \le \epsilon_k^{\frac 6{10}}}  \sum_{\|z\|\le \epsilon_k^{\frac 1{10}}\cdot \sqrt{\epsilon'_k} } p_{2\epsilon'_k}(z)\, |\overline p_{\epsilon_k-\epsilon'_k}(y-z) - \overline p_{\epsilon_k+\epsilon'_k}(y)| + \frac 1{\epsilon_k \sqrt{k} }. 
\end{align*}
Now straightforward computations show that for $y$ and $z$ as in the two sums above, one has for some constant $c>0$, 
$$|\overline p_{\epsilon_k-\epsilon'_k}(y-z) - \overline p_{\epsilon_k+\epsilon'_k}(y)| \lesssim \left(\frac{\|z\|}{\sqrt{\epsilon_k}} + \frac{\epsilon'_k}{\epsilon_k}\right)\overline p_{\epsilon_k-\epsilon'_k}(cy),$$
at least when $\epsilon'_k\le \sqrt{\epsilon_k}$, as will be assumed in a moment. 
Using also that $\sum_z \|z\|p_{2\epsilon'_k}(z) \lesssim \sqrt{\epsilon'_k}$, we deduce that 
$$ |\cov(Z_0\phi_{2,3}^0,\tilde Z_k\tilde \psi_0) | = \OO\left(\frac{\sqrt{\epsilon'_k}}{\epsilon_k \sqrt k}\right).$$
This concludes the proof as we choose $\epsilon'_k = \lfloor \sqrt{\epsilon_k} \rfloor$. 
\end{proof}

We can now quickly give the proof of Proposition \ref{prop.phipsi.1}. 

\begin{proof}[Proof of Proposition \ref{prop.phipsi.1}]
\underline{Case $1\le i<j\le 3$.} First note that $Z_0\phi_1$ and $Z_k\psi_3$ are independent, so only the cases $i=1$ and $j=2$, or $i=2$ and $j=3$ are at stake. Let us only consider the case $i=2$ and $j=3$, since the other one is entirely similar. Define, in the same fashion as in the proof of Proposition \ref{prop.ij0}, 
$$\phi_2^0:= \bP\left[H^+_{\RR[-\epsilon_k,\epsilon_k]}=\infty,\, H^+_{\RR[\epsilon_k+1,k-\epsilon_k]}<\infty\mid S\right]. $$
One has by using independence and translation invariance, 
$$\E[|\phi_2-\phi_2^0| \psi_3] \le \bP[H_{\RR[k-\epsilon_k,k]}<\infty]\cdot \bP[H_{\RR[\epsilon_k,\infty)}<\infty] \lesssim \frac{\sqrt{\epsilon_k}}{k^{3/2}},$$
which entails 
$$\cov(Z_0\phi_2,Z_k\psi_3) = \cov(Z_0\phi_2^0,Z_k\psi_3) + \OO\left(\frac{\sqrt{\epsilon_k}}{k^{3/2}}\right) \lesssim  \frac{\sqrt{\epsilon_k}}{k^{3/2}},$$
since $Z_0\phi_2^0$ and $Z_k\psi_3$ are independent.

\underline{Case $1\le j\le i\le 3$}. Here one can use entirely similar arguments as those from the proof of Lemma \ref{lem.ijl}, and we therefore omit the details.  
\end{proof}

%%%%%%%%%%%%%%%%%%%%%%%%%%%%%%%%%%%%%%%%%%%%%%%%%%%%%%%%%%%%%%%%%%%%%%%%%%%%%%%%%%%%%%%%%%%%

\section{Proof of Proposition \ref{prop.phi0}}
We need to estimate here the covariances $\cov(Z_0\phi_i, Z_k\psi_0)$ and $\cov(Z_0\phi_0, Z_k\psi_{4-i})$, for all $1\le i \le 3$.  

\vspace{0.2cm}
\underline{Case $i=1$.} It suffices to observe that 
$Z_0\phi_1$ and $Z_k\psi_0$ are independent, as are $Z_0\phi_0$ and $Z_k\psi_3$. Thus their covariances are equal to zero.

\underline{Case $i=2$.} We first consider the covariance between $Z_0\phi_2$ and $Z_k\psi_0$, which is easier to handle. Define 
$$\tilde \phi_2:=\bP\left[H^+_{\RR[-\epsilon_k,k-\epsilon_k-1]}=\infty,\, H^+_{\RR[k-\epsilon_k,k]}<\infty\mid S\right],$$ 
and note that $Z_0(\phi_2 - \tilde \phi_2)$ is independent of $Z_k\psi_0$. Therefore 
$$\cov(Z_0\phi_2,Z_k\psi_0) = \cov(Z_0\tilde \phi_2, Z_k\psi_0).$$
Then we decompose $\psi_0$ as $\psi_0=\psi_0^1-\psi_0^2$, where
$$\psi_0^1:=\bP_{S_k}[H^+_{\RR[k,k+\epsilon_k]}=\infty\mid S],\   \psi_0^2:=\bP_{S_k}[H^+_{\RR[k,k+\epsilon_k]}=\infty, H^+_{\RR[k-\epsilon_k,k-1]}<\infty \mid S].$$
Using now that $Z_k\psi_0^1$ is independent of $Z_0\tilde \phi_2$ we get 
$$\cov(Z_0\phi_2,Z_k\psi_0) =- \cov(Z_0\tilde \phi_2, Z_k\psi_0^2).$$
Let $(\tilde S_n)_{n\ge 0}$ and $(\hat S_n)_{n\ge 0}$ be two independent walks starting from the origin, and define 
$$\tau_1:=\inf \{n\ge 0 : S_{k-n}\in \tilde \RR[1,\infty)\},\  \tau_2:=\inf\{n\ge 0 : S_{k-n}\in (S_k + \hat \RR[1,\infty))\}.$$
We decompose  
\begin{align*}
& \cov(Z_0\tilde \phi_2, Z_k\psi_0^2)\\
&  = \E\left[Z_0\tilde \phi_2Z_k\psi_0^2\1\{\tau_1\le \tau_2\}\right] + \E\left[Z_0\tilde \phi_2Z_k\psi_0^2\1\{\tau_1> \tau_2\}\right] - \E[Z_0\tilde \phi_2] \E[Z_k\psi_0^2].
\end{align*} 
We bound the first term on the right-hand side simply by the probability of the event $\{\tau_1\le \tau_2\le \epsilon_k\}$, which we treat later, and for the difference between the last two terms, we use that 
$$\left| \1\{\tau_2<\tau_1\le \epsilon_k\} - \sum_{i=0}^{\epsilon_k} \1\left\{\tau_2=i,\, H^+_{\RR[k-\epsilon_k,k-i-1]}<\infty\right\}\right| \le \1\{\tau_1\le \tau_2\le \epsilon_k\}.$$
Using also that the event $\{\tau_2=i\}$ is independent of $(S_n)_{n\le k-i}$, we deduce that
\begin{align*}
&  |\cov(Z_0\tilde \phi_2, Z_k\psi_0^2)| \\ 
&\le 2\bP[\tau_1\le \tau_2\le \epsilon_k] + \sum_{i=0}^{\epsilon_k} \bP[\tau_2=i] \left|\bP\left[H^+_{\RR[k-\epsilon_k,k-i]}<\infty \right] - \bP\left[H^+_{\RR[k-\epsilon_k,k]}<\infty \right] \right|\\
 & \le 2\bP[\tau_1\le \tau_2\le \epsilon_k] +  \sum_{i=0}^{\epsilon_k} \bP[\tau_2=i] \cdot \bP\left[H^+_{\RR[k-i,k]}<\infty \right] \\
 & \stackrel{\eqref{lem.hit.3}}{\le}  2\bP[\tau_1\le \tau_2\le \epsilon_k]  + \frac{C}{k^{3/2}}\sum_{i=0}^{\epsilon_k} i \bP[\tau_2=i] \\
 & \le  2\bP[\tau_1\le \tau_2\le \epsilon_k]  + \frac{C}{k^{3/2}}\sum_{i=0}^{\epsilon_k} \bP[\tau_2\ge i]\\
 & \stackrel{\eqref{lem.hit.3}}{\le}  2\bP[\tau_1\le \tau_2\le \epsilon_k]  + \frac{C}{k^{3/2}}\sum_{i=0}^{\epsilon_k} \frac{1}{\sqrt i} \le  2\bP[\tau_1\le \tau_2\le \epsilon_k]  + \frac{C\sqrt{\epsilon_k}}{k^{3/2}}.  
\end{align*}  
Then it amounts to bound the probability of $\tau_1$ being smaller than $\tau_2$: 
\begin{align*}
&\bP[\tau_1\le \tau_2\le \epsilon_k] =\sum_{x,y\in \Z^5} \sum_{i=0}^{\epsilon_k} \bP\left[\tau_1=i, i\le \tau_2\le \epsilon_k, S_k=x,\, S_{k-i} = x+y\right]\\
 \le & \sum_{x,y\in \Z^5} \sum_{i=0}^{\epsilon_k} \bP\left[\tau_1=i, S_{k-i}=x+y,  (x+\hat \RR_\infty) \cap \RR[k-\epsilon_k,k-i]\neq \emptyset, S_k=x\right]\\
\le &  \sum_{x,y\in \Z^5} \sum_{i=0}^{\epsilon_k} \bP\left[\tilde \RR_\infty \cap (x+\RR[0,i-1])=\emptyset,  S_i=y,\, x+y\in \tilde \RR_\infty\right]\\
& \qquad \times  \bP\left[ \hat \RR_\infty \cap (y+ \RR[0,\epsilon_k-i])\neq \emptyset, S_{k-i}=-x-y\right],
\end{align*}
using invariance by time reversal of $S$, and where we stress the fact that in the first probability in the last line, $\RR$ and $\tilde \RR$ are two independent ranges starting from the origin. 
Now the last probability can be bounded using \eqref{Green.hit}  and Lemma \ref{lem.prep.123}, which give
\begin{align*}
 &\bP\left[ \hat \RR_\infty \cap (y+ \RR[0,\epsilon_k-i])\neq \emptyset, S_{k-i}=-x-y\right]  \le \sum_{j=0}^{\epsilon_k-i} \E\left[G(S_j+y)\1\{S_{k-i} = -x-y\}\right]\\
=& \sum_{j=0}^{\epsilon_k-i} \sum_{z\in \Z^5} p_j(z) G(z+y)p_{k-i-j}(z+x+y) =\sum_{j=k-\epsilon_k}^{k-i} \sum_{z\in \Z^5} p_j(z) G(z-x)p_{k- i-j}(z-x-y) \\
\lesssim & \frac{1}{(\|x+y\|+\sqrt k)^5}\left(\frac 1{1+\|y\|} + \frac 1{\sqrt k + \|x\|}\right).
\end{align*}
It follows that  
\begin{align*}
\bP[\tau_1\le \tau_2\le \epsilon_k] \lesssim \sum_{x,y\in \Z^5} \sum_{i=0}^{\epsilon_k} \frac{G(x+y) p_i(y)}{(\|x+y\| + \sqrt k)^5}\left(\frac 1{1+\|y\|} + \frac 1{\sqrt k+\|x\|}\right),
\end{align*}
and then standard computations show that 
\begin{equation}\label{tau12eps}
\bP[\tau_1\le \tau_2\le \epsilon_k] \lesssim \frac{\sqrt{\epsilon_k}}{k^{3/2}}.
\end{equation}
Taking all these estimates together proves that 
$$  \cov(Z_0\phi_2,Z_k\psi_0) \lesssim  \frac{\sqrt{\epsilon_k}}{k^{3/2}}.$$
We consider now the covariance between $Z_0\phi_0$ and $Z_k\psi_2$. Here a new problem arises due to the random variable $Z_0$, which does not play the same role as $Z_k$, 
but one can use similar arguments. In particular the previous proof gives    
$$\cov(Z_0\phi_0,Z_k\psi_2) = -\cov((1-Z_0)\phi_0,Z_k\psi_2) +  \OO\left(\frac{\sqrt{\epsilon_k}}{k^{3/2}}\right).$$
Then we decompose as well $\phi_0=\phi_0^1 - \phi_0^2$, with 
$$\phi_0^1:=\bP[H^+_{\RR[k-\epsilon_k,k]}=\infty\mid S],\   \phi_0^2:=\bP[H^+_{\RR[k-\epsilon_k,k]}=\infty, H^+_{\RR[k+1,k+\epsilon_k]}<\infty \mid S].$$
Using independence we get 
$$\cov((1-Z_0)\phi_0^1,Z_k\psi_2) = \E[\phi_0^1]\cdot \cov((1-Z_0),Z_k\psi_2).$$
Then we define in the same fashion as above, 
$$\tilde \tau_0:= \inf\{n\ge 1 : S_n=0\}, \  \tilde \tau_2:= \inf\{n\ge 0 : S_n\in (S_k+\hat \RR[1,\infty))\},$$
with $\hat \RR$ the range of an independent walk starting from the origin. Recall that by definition $1-Z_0 = \1\{\tilde \tau_0 \le \epsilon_k\}$. 
Thus one can write 
$$\cov((1-Z_0),Z_k\psi_2) = \E[Z_k\psi_2 \1\{\tilde \tau_2 \le \tilde \tau_0\le \epsilon_k\}] + \E[Z_k\psi_2 \1\{\tilde \tau_0 < \tilde \tau_2\}] - 
\bP[\tilde \tau_0\le \epsilon_k] \E[Z_k\psi_2].$$
On one hand, using \eqref{Green.hit}, the Markov property, and \eqref{exp.Green},  
\begin{align*}
& \E[Z_k\psi_2  \1\{\tilde \tau_2 \le \tilde \tau_0\le \epsilon_k\}] \le \bP[\tilde \tau_2 \le \tilde \tau_0 \le \epsilon_k]\le \sum_{y\in \Z^5} \bP[\tilde \tau_2\le \epsilon_k,\, S_{\tilde \tau_2} =y] \cdot G(y)\\
& \le \sum_{i=0}^{\epsilon_k} \E\left[G(S_i-S_k)G(S_i)\right] \le \sum_{i=0}^{\epsilon_k} \E[G(S_{k-i})]\cdot \E[G(S_i)] \lesssim \frac{1}{k^{3/2}} \sum_{i=0}^{\epsilon_k} \frac{1}{1+i^{3/2}} \lesssim \frac{1}{k^{3/2}}.
\end{align*}  
On the other hand, similarly as above, 
\begin{align}\label{tau20eps}
\nonumber & \E[Z_k\psi_2 \1\{\tilde \tau_0 < \tilde \tau_2\}] - \bP[\tilde \tau_0\le \epsilon_k]\cdot \E[Z_k\psi_2] \\
\nonumber  & \le  \bP[\tilde \tau_2 \le \tilde \tau_0 \le \epsilon_k] + \sum_{i=1}^{\epsilon_k}\bP[\tilde \tau_0=i]  
\left(\bP\left[(S_k+\hat \RR[1,\infty))\cap \RR[i+1,\epsilon_k]\neq \emptyset\right] - \bP[\tilde \tau_2\le \epsilon_k]\right)\\
\nonumber &\lesssim  \frac{1}{k^{3/2}} + \sum_{i=1}^{\epsilon_k}\bP[\tilde \tau_0=i]   \bP[\tilde \tau_2\le i] \stackrel{\eqref{lem.hit.3}}{\lesssim} \frac{1}{k^{3/2}} + 
\frac{1}{k^{3/2}}\sum_{i=1}^{\epsilon_k} i \bP[\tilde \tau_0=i] \\
& \lesssim \frac{1}{k^{3/2}} + \frac{1}{k^{3/2}}\sum_{i=1}^{\epsilon_k}  \bP[\tilde \tau_0\ge i] \stackrel{\eqref{Green.hit}, \eqref{Green}}{\lesssim} \frac{1}{k^{3/2}}+ \frac{1}{k^{3/2}}\sum_{i=1}^{\epsilon_k}  \frac{1}{1+i^{3/2}}  \lesssim \frac{1}{k^{3/2}}. 
\end{align}
In other terms, we have already shown that 
$$|\cov((1-Z_0)\phi_0^1,Z_k\psi_2) | \lesssim \frac{1}{k^{3/2}}.$$
The case when $\phi_0^1$ is replaced by $\phi_0^2$ is entirely similar. Indeed, we define  
$$\tilde \tau_1:=\inf\{n\ge 0 : S_n\in \tilde \RR[1,\infty)\},$$
with $\tilde \RR$ the range of a random walk starting from the origin, independent of $S$ and $\hat \RR$. Then we set  $\tilde \tau_{0,1} := \max(\tilde \tau_0,\tilde \tau_1)$, and 
exactly as for \eqref{tau12eps} and \eqref{tau20eps}, one has 
\begin{equation*}
\bP[\tilde \tau_2\le \tilde \tau_{0,1} \le \epsilon_k] \lesssim  \frac{\sqrt{\epsilon_k}}{k^{3/2}},
\end{equation*}
and 
\begin{align*}
& \E \left[(1-Z_0) \phi_0^2 Z_k\psi_2   \1\{ \tilde \tau_{0,1}<\tilde \tau_2 \} \right] - \E[(1-Z_0) \phi_0^2]\cdot \E[ Z_k\psi_2] \\
& \le \bP[\tilde \tau_2\le \tilde \tau_{0,1} \le \epsilon_k] +\sum_{i=0}^{\epsilon_k} \bP[\tilde \tau_{0,1}= i] \cdot \bP[\tilde \tau_2\le i] 
\lesssim \frac{\sqrt{\epsilon_k}}{k^{3/2}}.
\end{align*}
Altogether, this gives  
$$|\cov(Z_0\phi_0,Z_k\psi_2)|\lesssim \frac{\sqrt{\epsilon_k}}{k^{3/2}}.$$

\underline{Case $i=3$.} We only need to treat the case of the covariance between $Z_0\phi_3$ and $Z_k\psi_0$, as the other one is entirely similar here. Define 
$$\tilde \phi_3:=\bP\left[H^+_{\RR[-\epsilon_k,\epsilon_k]\cup \RR[k+\epsilon_k+1,\infty)}=\infty,\, H^+_{\RR[k,k+\epsilon_k]}<\infty\mid S\right].$$
The proof of the case $i=2$, already shows that 
$$|\cov(Z_0\tilde \phi_3,Z_k\psi_0)|\lesssim \frac{\sqrt{\epsilon_k}}{k^{3/2}}.$$
Define next 
$$h_3:=\phi_3-\tilde \phi_3=\bP\left[H^+_{\RR[-\epsilon_k,\epsilon_k]}=\infty,\, H^+_{\RR[k+\epsilon_k+1,\infty)}<\infty\mid S\right].$$ 
Assume for a moment that $\epsilon_k\ge k^{\frac{9}{20}}$. We will see later another argument when this condition is not satisfied. 
Then define $\epsilon_k':= \lfloor \epsilon_k^{10/9}/k^{1/9}\rfloor$, and note that one has $\epsilon'_k\le \epsilon_k$. Write $\psi_0=\psi'_0+h_0$, with  
$$\psi_0':= \bP\left[H^+_{\RR[k-\epsilon'_k+1,k+\epsilon'_k-1]}=\infty \mid S\right], $$
and 
$$h_0:= \bP\left[H^+_{\RR[k-\epsilon'_k+1,k+\epsilon'_k-1]}=\infty, \, H^+_{\RR[k-\epsilon_k,k-\epsilon'_k]\cup\RR[k+\epsilon'_k,k+\epsilon_k] } <\infty \mid S\right]. $$
Define also  
$$Z'_k:=\1\{S_\ell\neq S_k,\text{ for all }\ell=k+1,\dots,k+\epsilon'_k-1\}.$$ 
One has 
$$\cov(Z_0h_3,Z_k\psi_0) = \cov(Z_0h_3,Z'_k\psi'_0) + \cov(Z_0h_3,Z'_kh_0) +\cov(Z_0h_3,(Z_k-Z'_k)\psi_0) .$$
For the last of the three terms, one can simply notice that, using the Markov property at the first return time to $S_k$ (for the walk $S$), and then \eqref{Green.hit}, \eqref{Green}, and \eqref{lem.hit.3}, we get 
\begin{align*}
\E[h_3(Z_k-Z'_k)] & \le \E[Z_k-Z'_k] \times \bP[\tilde \RR_\infty \cap \RR\left[k,\infty)\neq \emptyset\right] \\
& \lesssim \frac{1}{(\epsilon'_k)^{3/2}\sqrt k} \lesssim \frac 1{\epsilon_k^{5/3}k^{1/3} }\lesssim \frac{1}{k^{\frac{13}{12}}},
\end{align*}
using our hypothesis on $\epsilon_k$ for the last equality. 
As a consequence, it also holds 
$$|\cov(Z_0h_3,(Z_k-Z'_k)\psi_0)| \lesssim k^{-\frac{13}{12}}.  $$
Next we write 
\begin{equation}\label{cov.i3}
\cov(Z_0h_3,Z'_kh_0) = \sum_{x,y\in \Z^5} (p_{k-2\epsilon_k}(x-y) - p_k(x)) H_1(y)H_2(x),
\end{equation}
where 
$$H_1(y) := \E\left[Z'_kh_0 \1\{S_{k+\epsilon_k} - S_{k-\epsilon_k} = y\}\right], \   
H_2(x) := \E\left[Z_0h_3 \mid S_{k+\epsilon_k} - S_{\epsilon_k} = x\right].$$
Define $r_k:=(k/\epsilon'_k)^{1/8}$. By using symmetry and translation invariance, 
\begin{align*}
&\sum_{\|y\|\ge \sqrt{\epsilon_k} r_k} H_1(y) \le \bP\left[H_{\RR[-\epsilon_k,-\epsilon'_k]\cup\RR[\epsilon'_k,\epsilon_k]}<\infty, \, \|S_{\epsilon_k} - S_{-\epsilon_k}\|\ge  \sqrt{\epsilon_k} r_k\right]\\
&\le 2 \bP\left[H_{\RR[\epsilon'_k,\epsilon_k]}<\infty, \, \|S_{\epsilon_k} \|\ge  \sqrt{\epsilon_k} \frac{r_k}{2}\right] + 2\bP\left[H_{\RR[\epsilon'_k,\epsilon_k]}<\infty, \, \|S_{-\epsilon_k} \|\ge  \sqrt{\epsilon_k} \frac{r_k}{2} \right]  \\
& \stackrel{\eqref{lem.hit.3}, \, \eqref{Sn.large}}{\le} 2  \bP\left[H_{\RR[\epsilon'_k,\epsilon_k]}<\infty, \, \|S_{\epsilon_k} \|\ge  \sqrt{\epsilon_k} \frac{r_k}2\right]  + \frac{C}{\sqrt{\epsilon'_k} r_k^5}.
\end{align*}
Considering the first probability on the right-hand side, define $\tau$ as the first hitting time (for $S$), after time $\epsilon'_k$, of another independent walk $\tilde S$ (starting from the origin). One has  
\begin{align*}
& \bP\left[H_{\RR[\epsilon'_k,\epsilon_k]}<\infty, \, \|S_{\epsilon_k} \|\ge  \sqrt{\epsilon_k} \frac{r_k}2\right]  \\
& \le \bP[\|S_\tau\| \ge  \sqrt{\epsilon_k} \frac{r_k}4,\, \tau\le \epsilon_k] +  \bP[\|S_{\epsilon_k} - S_\tau\| \ge  \sqrt{\epsilon_k} \frac{r_k}4,\, \tau\le \epsilon_k] .
\end{align*}
Using then the Markov property at time $\tau$, we deduce with \eqref{lem.hit.3} and \eqref{Sn.large}, 
$$\bP[\|S_{\epsilon_k} - S_\tau\| \ge  \sqrt{\epsilon_k} \frac{r_k}4,\, \tau\le \epsilon_k] \lesssim  \frac{1}{\sqrt{\epsilon'_k} r_k^5}.$$
Likewise, using the Markov property at the first time when the walk exit the ball of radius $\sqrt{\epsilon_k} r_k/4$, and applying then \eqref{Sn.large} and \eqref{lem.hit.2}, we get as well 
$$  \bP[\|S_\tau\| \ge  \sqrt{\epsilon_k} \frac{r_k}4,\, \tau\le \epsilon_k] \lesssim  \frac{1}{\sqrt{\epsilon_k} r_k^6}.$$
Furthermore, for any $y$, one has 
$$\sum_{x\in \Z^5} p_{k-2\epsilon_k}(x-y) H_2(x) \stackrel{\eqref{pn.largex}, \eqref{lem.hit.2}}{\lesssim} \sum_{x\in \Z^5} \frac{1}{(1+\|x+y\|)(\|x\|+\sqrt{k})^5} 
\lesssim \frac{1}{\sqrt k},$$
with an implicit constant, which is uniform in $y$ (and the same holds with $p_k(x)$ instead of $p_{k-2\epsilon_k}(x-y)$). 
Similarly, define $r'_k:=(k/\epsilon'_k)^{\frac 1{10}}$. One has for any $y$, with $\|y\|\le \sqrt{\epsilon_k}r_k$, 
$$\sum_{\|x\|\ge \sqrt{k}r_k'} p_{k-2\epsilon_k}(x-y) H_2(x)  \stackrel{\eqref{Sn.large}, \eqref{lem.hit.2}}{\lesssim}   \frac{1}{\sqrt{k}(r'_k)^6}.$$
Therefore coming back to \eqref{cov.i3}, and using that by \eqref{lem.hit.2},  $\sum_y H_1(y)\lesssim 1/\sqrt{\epsilon'_k}$,  we get
\begin{align*}
& \cov(Z_0h_3,Z'_kh_0)  \\
& = \sum_{\|x\|\le  \sqrt{k}r'_k} \sum_{\|y\|\le \sqrt{\epsilon_k}r_k} (p_{k-2\epsilon_k}(x-y) - p_k(x)) H_1(y)H_2(x) + \OO\left(\frac{1}{\sqrt{k\epsilon'_k}(r'_k)^6} + \frac{1}{\sqrt{k\epsilon'_k} r_k^5}\right)\\
& = \sum_{\|x\|\le  \sqrt{k}r'_k} \sum_{\|y\|\le \sqrt{\epsilon_k}r_k} (p_{k-2\epsilon_k}(x-y) - p_k(x)) H_1(y)H_2(x) + \OO\left(\frac{(\epsilon'_k)^{\frac{1}{10} } }{k^{\frac{11}{10}}}\right). 
\end{align*}
Now we use the fact $H_1(y) = H_1(-y)$. Thus the last sum is equal to half of the following:  
\begin{align*}
&\sum_{\|x\|\le  \sqrt{k}r'_k} \sum_{\|y\|\le \sqrt{\epsilon_k}r_k} (p_{k-2\epsilon_k}(x-y) + p_{k-2\epsilon_k}(x+y) - 2p_k(x)) H_1(y)H_2(x) \\
&\stackrel{\text{Theorem }\ref{LCLT},\eqref{lem.hit.2}}{\le} \sum_{\|x\|\le  \sqrt{k} r'_k} \sum_{\|y\|\le \sqrt{\epsilon_k}r_k} (\overline p_{k-2\epsilon_k}(x-y) + \overline p_{k-2\epsilon_k}(x+y) - 2\overline p_k(x)) H_1(y)H_2(x) \\
& \qquad + \OO\left(\frac{(r'_k)^4}{k^{3/2}\sqrt{\epsilon'_k}} \right),
\end{align*}
(with an additional factor $2$ in front in case of a bipartite walk). Note that the error term above is $\OO(k^{-11/10})$, by definition of $r'_k$.  
Moreover, straightforward computations show that for any $x$ and $y$ as in the sum above, 
$$|\overline p_{k-2\epsilon_k}(x-y) + \overline p_{k-2\epsilon_k}(x+y) - 2\overline p_k(x)| \lesssim  \left(\frac{\|y\|^2+\epsilon_k}{k} \right) \overline p_k(cx). $$
In addition one has (with the notation as above for $\tau$),
\begin{align*}
&\sum_{y\in \Z^5} \|y\|^2\, H_1(y) \le 2 \E\left[\|S_{\epsilon_k} - S_{-\epsilon_k}\|^2 \1\{\tau\le \epsilon_k\}\right] \\
& \le 4 \E[\|S_{\epsilon_k}\|^2] \bP[\tau\le \epsilon_k]  + 4 \E\left[\|S_{\epsilon_k}\|^2 \1\{\tau\le \epsilon_k\}\right]   \\
& \stackrel{\eqref{Sn.large}, \eqref{lem.hit.3}}{\lesssim}  \frac{ \epsilon_k}{\sqrt{\epsilon'_k}} +  \E\left[\|S_{\tau}\|^2 \1\{\tau\le \epsilon_k\}\right]  + \E\left[\|S_{\epsilon_k}-S_{\tau}\|^2 \1\{\tau\le \epsilon_k\}\right] \\
& \stackrel{\eqref{Sn.large}, \eqref{lem.hit.3}}{\lesssim}  \frac{ \epsilon_k}{\sqrt{\epsilon'_k}} +  \sum_{r\ge \sqrt{\epsilon_k}} r \bP\left[\|S_{\tau}\|\ge r, \, \tau\le \epsilon_k\right]  \stackrel{\eqref{Sn.large}, \eqref{lem.hit.2}}{\lesssim}  \frac{ \epsilon_k}{\sqrt{\epsilon'_k}}, 
\end{align*}
using also the Markov property in the last two inequalities (at time $\tau$ for the first one, and at the exit time of the ball of radius $r$ for the second one). 
Altogether, this gives  
$$|\cov(Z_0h_3,Z'_kh_0)| \lesssim  \frac{\epsilon_k }{k^{3/2}\sqrt{\epsilon'_k}} +  \frac{(\epsilon'_k)^{\frac{1}{10} } }{k^{\frac{11}{10}}}
\lesssim    \frac{(\epsilon_k)^{\frac{1}{9} } }{k^{ \frac{10}{9}  }}. $$
In other words, for any sequence $(\epsilon_k)_{k\ge 1}$, such that $\epsilon_k\ge k^{9/20}$, one has 
$$\cov(Z_0h_3,Z_k\psi_0) = \cov(Z_0h_3,Z'_k\psi'_0)  + \OO\left(\frac{(\epsilon_k)^{\frac{1}{9} } }{k^{ \frac{10}{9}  }} + \frac{1}{k^{\frac{13}{12}} }\right).$$
One can then iterate the argument with the sequence $(\epsilon'_k)$ in place of $(\epsilon_k)$, and (after at most a logarithmic number of steps), we are left to consider a sequence $(\epsilon_k)$, satisfying $\epsilon_k\le k^{9/20}$. In this case, we use similar arguments as above. Define $\tilde H_1(y)$ as $H_1(y)$, but with $Z_k\psi_0$ instead of $Z'_kh_0$ in the expectation, and choose $r_k:= \sqrt{k/\epsilon_k}$, and $r'_k=k^{\frac 1{10}}$. Then we obtain exactly as above, 
\begin{align*}
&\cov(Z_0h_3,Z_k\psi_0) \\
& = \sum_{\|x\|\le \sqrt k r'_k} \sum_{\|y\|\le \sqrt{k}} ( p_{k-2\epsilon_k}(x-y) -  p_k(x)) \tilde H_1(y) H_2(x) + \OO\left(\frac{1}{r_k^5\sqrt k} + \frac 1{(r'_k)^6 \sqrt k}\right)\\
& = \sum_{\|x\|\le \sqrt k r'_k} \sum_{\|y\|\le \sqrt{k}} ( \overline p_{k-2\epsilon_k}(x-y) -  \overline p_k(x)) \tilde H_1(y) H_2(x) + \OO\left(\frac{1}{k^{\frac {11}{10}}} \right)\\
& \lesssim \frac{\epsilon_k}{k^{3/2}} + \frac{1}{k^{\frac {11}{10}}}  \lesssim  \frac{1}{k^{\frac {21}{20}}},  
\end{align*}
which concludes the proof of the proposition.

%%%%%%%%%%%%%%%%%%%%%%%%%%%%%%%%%%%%%%%%%%%%%%%%%%%%%%%%%%%%%%%%%%%%%%%%%

\section{Intersection of two random walks and proof of Theorem C} \label{sec.thmC}
In this section we prove a general result, which will be needed for proving Proposition \ref{prop.phipsi.2}, and which also gives Theorem C as a corollary.  
First we introduce some general condition for a function $F:\Z^d\to \R$, namely:
\begin{equation}\label{cond.F}
\begin{array}{c}
\text{there exists a constant $C_F>0$, such that }\\
|F(y) - F(x)|\le C_F\, \frac{\|y-x\|}{1+\|y\|} \cdot |F(x)|, \quad \text{for all }x,y\in \Z^d.
\end{array}
\end{equation}  
Note that any function satisfying \eqref{cond.F} is automatically bounded. 
Observe also that this condition is satisfied by functions which are equivalent to $c/\mathcal J(x)^\alpha$, for some constants $\alpha\in [0,1]$, and $c>0$. 
On the other hand, it is not satisfied by functions which are $o(1/\| x\|)$, as $\|x\|\to \infty$.  
However, this is fine, since the only two cases that will be of interest for us here are when either $F$ is constant, or when $F(x)$ is of order $1/\|x\|$. 
Now for a general function $F:\Z^d\to \R$, we define for $r>0$, 
$$\overline F(r) := \sup_{r\le \|x\|\le r+1} |F(x)|.$$
Then, set 
$$I_F(r):= \frac {\log (2+r)}{r^{d-2}}\int_0^r s\cdot \overline F(s)\, ds + \int_r^\infty \frac{\overline F(s)\log(2+s)}{s^{d-3} }\, ds,$$
and, with $\chi_d(r):= 1+(\log (2+r))\1_{\{d=5\}}$,  
$$ J_F(r):= \frac{\chi_d(r)}{r^{d-2}} \int_0^r \overline F(s)\, ds + \int_r^\infty \frac{\overline F(s)\chi_d(s)}{s^{d-2}}\, ds.$$

\begin{theorem}\label{thm.asymptotic}
Let $(S_n)_{n\ge 0}$ and $(\widetilde S_n)_{n\ge 0}$ be two independent random walks on $\Z^d$, $d\ge 5$, starting respectively from the origin and some $x\in \Z^d$. Let $\ell \in \N\cup \{\infty\}$, and define 
$$\tau:=\inf\{n\ge 0\, : \, \tilde S_n \in \RR[0,\ell] \}.$$
There exists $\nu\in (0,1)$, such that for any $F:\Z^d\to \R$, satisfying \eqref{cond.F},   
\begin{align}\label{thm.asymp.formula}
\E_{0,x}\left[F(\tilde S_\tau) \1\{\tau<\infty\}\right] = &\  \frac {\gamma_d}{\kappa}\cdot  \E\left[\sum_{i=0}^\ell G(S_i-x)F(S_i)\right] \\
\nonumber & +  \OO\left(\frac{I_F(\|x\|)}{(\ell\wedge \|x\|)^\nu} + (\ell\wedge \|x\|)^\nu J_F(\|x\|)\right),
\end{align}
where $\gamma_d$ is as in \eqref{LLN.cap}, and $\kappa$ is some positive constant given by 
$$\kappa:=\E\left[\Big(\sum_{n\in \Z} G(S_n)\Big)\cdot \bP\left[H^+_{\overline \RR_\infty}=+\infty \mid \overline \RR_\infty \right]\cdot \1\{S_n\neq 0,\, \forall n\ge 1\}\right],$$
with $(S_n)_{n\in\Z}$ a two-sided walk starting from the origin and $\overline \RR_\infty := \{S_n\}_{n\in \Z}$.  
\end{theorem}

\begin{remark}\emph{Note that when $F(x) \sim c/\mathcal J(x)^{\alpha}$, for some constants $\alpha \in [0,1]$ and $c>0$, then $I_F(r)$ and $J_F(r)$ are respectively of order $1/r^{d-4+ \alpha}$, and $1/r^{d-3+\alpha}$ (up to logarithmic factors), while  
one could show that 
$$\E\left[\sum_{i=0}^\ell G(S_i-x)F(S_i)\right] \sim \frac{c'}{\mathcal J(x)^{d-4+\alpha}}, \quad \text{as }\|x\|\to \infty\text{ and }\ell/\|x\|^2\to \infty,$$ 
for some other constant $c'>0$ (see below for a proof at least when $\ell = \infty$ and $\alpha=0$). Therefore in these cases Theorem \ref{thm.asymptotic} 
provides a true equivalent for the term on the left-hand side of \eqref{thm.asymp.formula}. 
}
\end{remark}

\begin{remark}
\emph{This theorem strengthens Theorem C in two aspects: on one hand it allows to consider functionals of the position of one of the two walks at its hitting time of the other path, and on the other hand it also allows to consider only a finite time horizon for one of the two walks (not mentioning the fact that it gives additionally some bound on the error term). Both these aspects will be needed later (the first one in the proof of Lemma \ref{lem.var.2}, and the second one in the proofs of Lemmas \ref{lem.var.3} and \ref{lem.var.4}). }
\end{remark}

Given this result one obtains Theorem C as a corollary. To see this, we first recall an asymptotic result on the Green's function: 
in any dimension $d\ge 5$, under our hypotheses on $\mu$, there exists a constant $c_d>0$, such that as $\|x\|\to \infty$, 
\begin{equation}\label{Green.asymp}
G(x)= \frac{c_d}{\mathcal J(x)^{d-2}} + \OO(\|x\|^{1-d}).
\end{equation}
This result is proved in \cite{Uchiyama98} under only the hypothesis that $X_1$ has a finite $(d-1)$-th moment (we refer also to Theorem 4.3.5 in \cite{LL}, for a proof under the stronger hypothesis that $X_1$ has a finite $(d+1)$-th moment). 
One also needs the following elementary fact: 
\begin{lemma}\label{Green.convolution}
There exists a positive constant $c$, such that as $\|x\|\to \infty$, 
$$\sum_{y\in \Z^d\setminus\{0,x\}} \frac{1}{\mathcal J(y)^{d-2}\cdot \mathcal J(y-x)^{d-2}} =  \frac{c}{\mathcal J(x)^{d-4}} + \OO\left(\frac 1{\|x\|^{d-3}}\right).$$
\end{lemma}
\begin{proof}
The proof follows by first an approximation by an integral, and then a change of variables. More precisely, letting $u:=x/\mathcal J(x)$, one has  
\begin{align*}
& \sum_{y\in \Z^d\setminus\{0,x\} } \frac{1}{\mathcal J(y)^{d-2} \mathcal J(y-x)^{d-2}}  = \ \int_{\R^d} \frac{1}{\mathcal J(y)^{d-2} \mathcal J(y-x)^{d-2}} \, dy + 
\OO(\|x\|^{3-d}) \\
 & =  \frac 1{\mathcal J(x)^{d-4}} \int_{\R^5}\frac{1}{\mathcal J(y)^{d-2} \mathcal J(y-u)^{d-2} }\, dy +   \OO(\|x\|^{3-d}), 
\end{align*}  
and it suffices to observe that by rotational invariance, the last integral is independent of $x$. 
\end{proof}

\begin{proof}[Proof of Theorem C] 
The result follows from Theorem \ref{thm.asymptotic}, by taking $F\equiv 1$ and $\ell = \infty$, and then by using \eqref{Green.asymp} together with Lemma \ref{Green.convolution}. 
\end{proof}

It amounts now to prove Theorem \ref{thm.asymptotic}. 
For this, we need some technical estimates that we gather in Lemma \ref{lem.thm.asymptotic} below. Since we believe this is not the most interesting part, we defer its proof to the end of this section.

\begin{lemma}\label{lem.thm.asymptotic}
Assume that $F$ satisfies \eqref{cond.F}. Then 
\begin{enumerate}
\item There exists a constant $C>0$, such that for any $x\in \Z^d$, 
\begin{equation}\label{lem.thm.asymp.1}
\sum_{i=0}^\infty \E\left[\left(\sum_{j=0}^\infty G(S_j-S_i)\frac{\|S_j-S_i\|}{1+\|S_j\|}\right) \cdot |F(S_i)| G(S_i-x)\right] \le C J_F(\|x\|). 
\end{equation}
\item There exists $C>0$, such that for any  $R>0$, and any $x\in \Z^d$,  
\begin{equation}\label{lem.thm.asymp.2}
\sum_{i=0}^\infty \E\left[\left(\sum_{|j-i|\ge R}G(S_j-S_i)\right) |F(S_i)| G(S_i-x)\right] \le \frac{C}{R^{\frac{d-4}2}}\cdot I_F(\|x\|), 
\end{equation}
\begin{equation}\label{lem.thm.asymp.2bis}
\sum_{i=0}^\infty \E\left[\left(\sum_{|j-i|\ge R}G(S_j-S_i) |F(S_j)|\right) G(S_i-x)\right] \le \frac{C}{R^{\frac{d-4}2} }\cdot I_F(\|x\|)  . 
\end{equation}
\end{enumerate}
\end{lemma}

One also need some standard results from (discrete) potential theory. If $\Lambda$ is a nonempty finite subset of $\Z^d$, containing the origin, we define 
$$\text{rad}(\Lambda):=1+\sup_{x\in \Lambda} \|x\|,$$
and also consider for $x\in \Lambda$,  
$$e_\Lambda(x):=\bP_x[H_\Lambda^+=\infty], \quad \text{and} \quad \overline e_\Lambda(x):=\frac{e_\Lambda(x)}{\cp(\Lambda)}.$$
The measure $\overline e_\Lambda$ is sometimes called the harmonic measure of $\Lambda$ from infinity, due to the next result.  
\begin{lemma}\label{lem.potential}
There exists a constant $C>0$, such that for any finite subset $\Lambda\subseteq \Z^d$ containing the origin, and any $y\in \Z^d$, with $\| y\|>2\text{rad}(\Lambda)$, 
\begin{eqnarray}\label{cap.hitting}
\bP_y[H_\Lambda<\infty] \le C\cdot \frac{\cp(\Lambda)}{1+\|y\|^{d-2}}.
\end{eqnarray}
Furthermore, for any $x\in \Lambda$, and any $y\in \Z^d$,  
\begin{eqnarray}\label{harm.hit}
\Big| \bP_y[S_{H_\Lambda}=x\mid  H_\Lambda<\infty] - \overline e_\Lambda(x)\Big| \le C\cdot \frac{\text{rad}(\Lambda)}{1+\|y\|}.
\end{eqnarray}
\end{lemma}
This lemma is proved in \cite{LL} for finite range random walks. The proof extends to our setting, but some little care is needed, so  
we shall give some details at the end of this section. 
Assuming this, one can now give the proof of our main result.

\begin{proof}[Proof of Theorem \ref{thm.asymptotic}]
The proof consists in computing the quantity  
\begin{equation}\label{eq.A}
A:= \E_{0,x}\left[\sum_{i=0}^\ell \sum_{j=0}^\infty \1\{S_i = \tilde S_j\}F(S_i)\right],
\end{equation}
in two different ways\footnote{This idea goes back to the seminal paper of Erd\'os and Taylor \cite{ET60}, even though it was not used properly there and was corrected only a few years later by Lawler, see \cite{Law91}.}. On one hand, by integrating with respect to the law of $\tilde S$ first, we obtain 
\begin{equation}\label{A.first}
A= \E\left[\sum_{i=0}^\ell G(S_i-x)F(S_i)\right].
\end{equation}
On the other hand, the double sum in \eqref{eq.A} is nonzero only when $\tau$ is finite. Therefore, using also the Markov property at time $\tau$, we get
\begin{align*}
A &= \E_{0,x}\left[\left(\sum_{i=0}^\ell \sum_{j=0}^\infty \1\{S_i = \tilde S_j\}F(S_i)\right) \1\{\tau<\infty\} \right]\\
& = \sum_{i=0}^\ell \E_{0,x}\left[ \left( \sum_{j=0}^\ell G(S_j-S_i) F(S_j)\right)Z_i^\ell \cdot \1\{\tau<\infty, \tilde S_\tau = S_i\} \right],
\end{align*}
where we recall that $Z_i^\ell = \1\{S_j \neq S_i,\, \forall j=i+1,\dots,\ell\}$. The computation of this last expression is divided in a few steps. 

\underline{Step 1.} Set 
$$B:=  \sum_{i=0}^\ell \E_{0,x}\left[ \left( \sum_{j=0}^\ell G(S_j-S_i) \right)F(S_i)Z_i^\ell \cdot \1\{\tau<\infty, \tilde S_\tau = S_i\} \right],$$
and note that,  
\begin{align*}
& |A-B| \stackrel{\eqref{cond.F}}{\le} C_F\, \sum_{i=0}^\ell \E_{0,x}\left[ \left( \sum_{j=0}^\ell G(S_j-S_i)\frac{\|S_j-S_i\|}{(1+\|S_j\|)} \right)|F(S_i)| \1\{S_i\in \tilde \RR_\infty\} \right]\\
&\stackrel{\eqref{Green.hit}}{\le} C_F\, \sum_{i=0}^\ell \E\left[ \left( \sum_{j=0}^\ell G(S_j-S_i)\frac{\|S_j-S_i\|}{(1+\|S_j\|)} \right)|F(S_i)| G(S_i-x)\right] \stackrel{\eqref{lem.thm.asymp.1}}{=} \OO\left(J_F(\|x\|)  \right).
\end{align*}

\underline{Step 2.} Consider now some positive integer $R$, and define 
$$D_R:=  \sum_{i=0}^{\ell} \E_{0,x}\left[ \mathcal G_{i,R,\ell} F(S_i)Z_i^\ell \cdot \1\{\tau<\infty, \tilde S_\tau = S_i\} \right],$$
with $\mathcal G_{i,R,\ell}:= \sum_{j=(i-R)\vee 0}^{(i+R)\wedge \ell} G(S_j-S_i)$. 
One has 
$$|B-D_R| \stackrel{\eqref{Green.hit}}{\le}  \sum_{i=0}^{\ell}  \E\left[ \left(\sum_{|j-i|>R} G(S_j-S_i)\right) |F(S_i)|G(S_i-x)\right] 
\stackrel{\eqref{lem.thm.asymp.2}}{\lesssim} \frac{ I_F(\|x\|)}{R^{\frac{d-4}2}}.$$

\underline{Step 3.} Let $R$ be an integer larger than $2$, and such that $\ell\wedge \|x\|^2 \ge R^6$. 
Let $M:=\lfloor \ell / R^5\rfloor -1$, and define for $0\le m\le M$, 
$$I_m:=\{mR^5+R^3,\dots, (m+1)R^5-R^3\},  \text{ and }  J_m:=\{mR^5,\dots, (m+1)R^5-1\}.$$
Define further  
$$E_R := \sum_{m=0}^M \sum_{i\in I_m} \E_{0,x}\left[ \mathcal G_{i,R} F(S_i)Z_i^\ell \cdot \1\{\tau<\infty, \tilde S_\tau = S_i\} \right],$$
with $\mathcal G_{i,R} := \sum_{j=i-R}^{i+R} G(S_j-S_i)$. 
One has, bounding $\mathcal G_{i,R}$ by $(2R+1)G(0)$, 
\begin{align*}
|D_R - E_R| &   \le   (2R+1) G(0)\\
 & \times \left\{ \sum_{m=0}^M \sum_{i\in J_m\setminus I_m} \E \left[|F(S_i)| G(S_i-x) \right]  + \sum_{i=(M+1) R^5}^\ell \E \left[|F(S_i)| G(S_i-x) \right] \right\},
\end{align*}
with the convention that the last sum is zero when $\ell$ is infinite. 
Using $\ell \ge R^6$, we get 
\begin{align*}
&\sum_{i=(M+1) R^5}^\ell \E \left[|F(S_i)| G(S_i-x) \right]  \le \sum_{z\in \Z^d} |F(z)| G(z-x) \sum_{i=(M+1)R^5}^{(M+2)R^5} p_i(z) \\
&\stackrel{\eqref{pn.largex},\, \eqref{Green}}{\lesssim} \frac{R^5}{\ell} \sum_{z\in \Z^d} \frac{|F(z)|}{(1+ \|z-x\|^{d-2})(1+\|z\|^{d-2})} \lesssim \frac{R^5}{ \ell} \cdot I_F(\|x\|). 
\end{align*}
Likewise, since $\|x\|^2\ge R^6$, 
\begin{align}\label{final.step3}
\nonumber & \sum_{m=0}^M \sum_{i\in J_m\setminus I_m} \E \left[|F(S_i)| G(S_i-x) \right] \le \sum_{z\in \Z^d} \frac{|F(z)|}{1+\|z-x\|^{d-2}} \sum_{m=0}^M \sum_{i\in J_m\setminus I_m} p_i(z)\\
\nonumber &\stackrel{\eqref{Green}}{\lesssim} \frac{1}{1+\|x\|^{d-2}} \sum_{\|z\|^2\le R^5} \frac{1}{1+\|z\|^{d-2}} \\
\nonumber & \qquad + \sum_{\|z\|^2 \ge R^5} \frac{|F(z)|}{1+\|z-x\|^{d-2}}  \sum_{m=0}^M \sum_{i\in J_m\setminus I_m} \left(\frac{\1\{i\le \|z\|^2\}}{1+\|z\|^d} + \frac{\1\{i\ge \|z\|^2\}}{i^{d/2}}\right) \\
& \lesssim   \frac{R^5}{1+\|x\|^{d-2}} + \frac{1}{R^2} \cdot I_F(\|x\|),
\end{align}
using for the last inequality that the proportion of indices $i$ which are not in one of the $I_m$'s, is of order $1/R^2$.

\underline{Step 4.}
For $0\le m \le M+1$, set 
$$\RR^{(m)}:=\RR[mR^5,(m+1)R^5-1], \quad \text{and}\quad \tau_m:= \inf\{ n \ge 0 \, :\, \tilde S_n \in \RR^{(m)}\}.$$
Then let  
$$F_R := \sum_{m=0}^M \sum_{i\in I_m} \E_{0,x}\left[ \mathcal G_{i,R} F(S_i)Z_i^\ell \cdot \1\{\tau_m<\infty, \tilde S_{\tau_m} = S_i\} \right].$$ 
Since by definition $\tau\le \tau_m$, for any $m$, one has for any $i\in I_m$, 
\begin{align*}
& |\bP_{0,x}[\tau<\infty, \tilde S_{\tau} = S_i\mid S] - \bP_{0,x}[\tau_m<\infty, \tilde S_{\tau_m} = S_i\mid S] | \\
 & \le  \bP_{0,x}[\tau<\tau_m<\infty,  \tilde S_{\tau_m}=S_i\mid S] 
 \le  \sum_{j\notin J_m} \bP_{0,x}[\tau<\tau_m<\infty,  \tilde S_{\tau} = S_j,  \tilde S_{\tau_m} = S_i\mid S]\\
&\stackrel{\eqref{Green.hit}}{\le}  \sum_{j\notin J_m} G(S_j-x) G(S_i-S_j). 
\end{align*} 
Therefore, bounding again $\mathcal G_{i,R}$ by $(2R+1)G(0)$, we get 
\begin{align*}
|E_R-F_R| & \lesssim  R \, \sum_{m=0}^M \sum_{i\in I_m} \E\left[\left(\sum_{j\notin J_m}  G(S_i-S_j) G(S_j-x)\right)\cdot   |F(S_i)| \right]\\
&\lesssim R \, \sum_{i=0}^\infty \E\left[\left(\sum_{j\, :\, |j-i|\ge R^3} G(S_i-S_j) G(S_j-x)\right)\cdot |F(S_i)| \right] \\
& \stackrel{\eqref{lem.thm.asymp.2bis}}{\lesssim} \frac{1}{ R^{3\frac{d-4}{2}-1}} \cdot I_F(\|x\|)\lesssim \frac{1}{\sqrt R} \cdot I_F(\|x\|). 
\end{align*}

\underline{Step 5.} For $m\ge 0$ and $i\in I_m$, define 
$$e_i^m := \bP_{S_i} \left[H_{\RR^{(m)}}^+=\infty\mid S\right], \quad \text{and}\quad  \overline e_i^m:= \frac{e_i^m}{\cp(\RR^{(m)}) }.$$ 
Then let 
$$H_R: = \sum_{m=0}^M \sum_{i\in I_m} \E_{0,x}\left[ \mathcal G_{i,R}F(S_i)Z_i^\ell \overline e_i^m \cdot \1\{\tau_m<\infty\} \right].$$ 
Applying \eqref{harm.hit} to the sets $\Lambda_m:=\RR^{(m)}-S_{i_m}$, we get for any $m\ge 0$, and any $i\in I_m$, 
\begin{eqnarray}\label{harm.application}
\left| \bP_{0,x}[\tilde S_{\tau_m} = S_i\mid \tau_m<\infty, S] - \overline e_i^m \right| \le C\,  \frac{\text{rad}(\Lambda_m)}{1+\|x-S_{i_m}\|}.
\end{eqnarray}
By \eqref{cap.hitting}, it also holds 
\begin{align}\label{hit.cap.application}
\nonumber \bP_{0,x}[\tau_m<\infty \mid S ]  & \le \frac{ CR^5}{1+\|x-S_{i_m}\|^{d-2}} + \1\{\|x-S_{i_m} \| \le 2\text{rad}(\Lambda_m)\}\\
&  \lesssim  \frac{ R^5+\text{rad}(\Lambda_m)^{d-2}}{1+\|x-S_{i_m}\|^{d-2}},
\end{align}
using that $\cp(\Lambda_m)\le |\Lambda_m| \le R^5$. Note also that by \eqref{norm.Sn} and Doob's $L^p$-inequality (see Theorem 4.3.3 in \cite{Dur}), one has for any $1< p\le d$,  
\begin{equation}\label{Doob}
\E[\text{rad}(\Lambda_m)^p] = \OO(R^{\frac{5p}{2}}).
\end{equation} 
Therefore, 
\begin{align*}
& |F_R - H_R|  \stackrel{\eqref{harm.application}}{\lesssim} R  
\sum_{m=0}^M \sum_{i\in I_m} \E_{0,x}\left[ \frac{|F(S_i)|\cdot \text{rad}(\Lambda_m)}{1+\|x-S_{i_m}\|} \1\{ \tau_m<\infty\}\right] \\
&\stackrel{\eqref{cond.F}}{\lesssim} R^6  
\sum_{m=0}^M  \E_{0,x}\left[ \frac{|F(S_{i_m})|\cdot \text{rad}(\Lambda_m)^2}{1+\|x-S_{i_m}\|} \1\{ \tau_m<\infty\}\right] \\
& \stackrel{\eqref{hit.cap.application},  \eqref{Doob}}{\lesssim} R^{6+\frac{5d}{2}}  \sum_{m=0}^M \E\left[ \frac{|F(S_{i_m})|}{1+\|x-S_{i_m}\|^{d-1}}\right] \lesssim  R^{6+\frac{5d}{2}}\,  \sum_{z\in \Z^d}  \frac{|F(z)|G(z)}{1+\|x-z\|^{d-1}}  \\
& \stackrel{\eqref{Green}}{\lesssim}  \frac{R^{6+\frac{5d}{2}}}{1+\|x\|} \cdot  I_F(\|x\|).
\end{align*}

\underline{Step 6.}
Let 
$$K_R:= \sum_{m=0}^M \sum_{i\in I_m} \E \left[ \mathcal G_{i,R} Z_i^\ell \overline e_i^m\right] \cdot \E\left[F(S_{i_m}) \1\{\tau_m<\infty\} \right].$$
One has, using the Markov property and a similar argument as in the previous step, 
\begin{align*}
& |K_R-H_R|  \stackrel{\eqref{cond.F}}{\lesssim}  
R \sum_{m=0}^M \sum_{i\in I_m}  \E_{0,x}\left[\frac{|F(S_{i_m})|\cdot (1+\|S_i-S_{i_m}\|^2)}{1+\|S_{i_m}\|}\cdot  \1\{\tau_m<\infty\} \right] \\
& \stackrel{\eqref{hit.cap.application},  \eqref{Sn.large}}{\lesssim} R^{6+\frac{5d}{2}}  \sum_{m=0}^M \E\left[ \frac{|F(S_{i_m})|}{(1+\|S_{i_m}\|)(1+\|x-S_{i_m}\|^{d-2})}\right] \lesssim  R^{6+\frac{5d}{2}}\cdot  J_F(\|x\|).
\end{align*}

\underline{Step 7.} 
Finally we define 
$$\tilde A:= \frac{\kappa}{\gamma_d}   \cdot \E_{0,x}\left[F(\tilde S_\tau) \1\{\tau<\infty\} \right].$$
We recall that one has (see Lemmas 2.1 and 2.2 in \cite{AS19}), 
\begin{eqnarray}\label{easy.variance}
\E\left[\left(\cp(\RR_n) - \gamma_d n\right)^2\right] = \OO(n(\log n)^2). 
\end{eqnarray} 
It also holds for any nonempty subset $\Lambda\subseteq \Z^d$, 
\begin{eqnarray}\label{min.cap}
\cp(\Lambda) \ge c|\Lambda|^{1-\frac 2d}\ge c|\Lambda|^3,
\end{eqnarray}
using $d\ge 5$ for the second inequality (while the first inequality follows from \cite[Proposition 6.5.5]{LL} applied to the constant function equal to $c/|\Lambda|^{2/d}$, with $c>0$ small enough). As a consequence, for any $m\ge 0$ and any $i\in I_m$, 
\begin{align*}
& \left|\E \left[ \mathcal G_{i,R} Z_i^\ell \overline e_i^m\right]  -  \frac{\E \left[ \mathcal G_{i,R} Z_i^\ell e_i^m\right]}{\gamma_dR^5} \right| \lesssim  
\frac{1}{R^4}  \E\left[ \frac {|\cp(\RR^{(m)}) - \gamma_dR^5|}{\cp(\RR^{(m)})}\right] \\
 \stackrel{\eqref{easy.variance}}{\lesssim} & \frac{\log R}{R^{3/2}}  \E\left[\frac 1{\cp(\RR^{(m)})^2}\right]^{1/2} 
 \stackrel{\eqref{min.cap}}{\lesssim} \frac{\log R}{R^{3/2}} \left(\frac{\bP[\cp(\RR^{(m)}) \le \gamma_d R^5/2]}{R^6} + \frac{1}{R^{10}}\right)^{1/2} \\ 
  \stackrel{\eqref{easy.variance}}{\lesssim} &  \frac{\log R}{R^{3/2}} \left(\frac{(\log R)^2}{R^{11}} + \frac{1}{R^{10}}\right)^{1/2} \lesssim \frac{1}{R^6}.
\end{align*}
Next, recall that $Z(i)=\1\{S_j\neq S_i,\,  \forall j>i\}$, and note that 
$$|\E \left[ \mathcal G_{i,R} Z_i^\ell e_i^m\right] - \E \left[ \mathcal G_{i,R} Z(i) e_i^m\right]| \stackrel{\eqref{Green.hit},\, \eqref{Green}}{\lesssim} \frac{1}{R^{7/2}}.$$
Moreover, letting $e_i:=\bP_{S_i}[H^+_{\overline \RR_\infty} = \infty\mid \overline \RR_\infty]$ (where we recall $\overline \RR_\infty$ is the range of a two-sided random walk), one has
$$ |\E \left[ \mathcal G_{i,R} Z_i e_i^m\right] - \E \left[ \mathcal G_{i,R} Z_i e_i \right]| \stackrel{\eqref{lem.hit.3}}{\lesssim} \frac{1}{\sqrt R}, $$
$$| \E \left[ \mathcal G_{i,R} Z_i e_i \right] - \kappa| \le 2\, \E\left[\sum_{j>R} G(S_j)\right] \stackrel{\eqref{exp.Green}}{\lesssim}\frac 1{\sqrt R}.$$ 
Altogether this gives for any $m\ge 0$ and any $i\in I_m$, 
$$\left|\E \left[ \mathcal G_{i,R} Z_i^\ell \overline e_i^m\right]  - \frac{\kappa}{\gamma_dR^5}\right|\lesssim \frac{1}{R^{5+\frac 12}},$$
and thus for any $m\ge 0$, 
$$\left| \left(\sum_{i\in I_m} \E \left[ \mathcal G_{i,R} Z_i^\ell \overline e_i^m\right] \right) - \frac{\kappa}{\gamma_d}\right|\lesssim \frac{1}{\sqrt R}.$$
Now, a similar argument as in Step 6 shows that 
$$\sum_{m=0}^M \left|\E_{0,x}\left[F(S_{i_m}) \1\{\tau_m<\infty\} \right] - \E_{0,x}\left[F(\tilde S_{\tau_m}) \1\{\tau_m<\infty\} \right]\right| \lesssim R^{\frac{5d}2} J_F(\|x\|). $$
Furthermore, using that 
\begin{align*}
F(\tilde S_\tau)\1\{\tau<\infty\} & = \sum_{m=0}^{M+1} F(\tilde S_{\tau_m})\1\{\tau=\tau_m<\infty\}\\
& =\sum_{m=0}^{M+1} F(\tilde S_{\tau_m})(\1 \{\tau_m<\infty\} - \1\{\tau<\tau_m<\infty\}), 
\end{align*}
(with the convention that the term corresponding to index $M+1$ is zero when $\ell =\infty$) we get, 
\begin{align*}
& \left| \sum_{m=0}^M \E_{0,x}\left[F(\tilde S_{\tau_m}) \1\{\tau_m<\infty\} \right] - \E_{0,x}\left[F(\tilde S_{\tau}) \1\{\tau<\infty\} \right]\right| \\
& \lesssim \bP_{0,x}[\tau_{M+1}<\infty] + \sum_{m=0}^M \E_{0,x}\left[|F(\tilde S_{\tau_m})| \1\{\tau<\tau_m<\infty\}\right].
\end{align*}
Using \eqref{hit.cap.application}, \eqref{Doob} and \eqref{exp.Green.x}, we get  
$$\bP_{0,x}[\tau_{M+1}<\infty] \lesssim \frac{R^{\frac{5(d-2)}2}}{1+\|x\|^{d-2} }.$$
On the other hand, for any $m\ge 0$, 
\begin{align*}
&\E\left[|F(\tilde S_{\tau_m})| \1\{\tau<\tau_m<\infty\}\right] \le \sum_{j\in J_m} \sum_{i\notin J_m} \E\left[|F(S_j)| G(S_i-S_j)G(S_i-x)\right]\\
& \le \sum_{j\in I_m} \sum_{|j-i|>R^3}  \E\left[|F(S_j)| G(S_i-S_j)G(S_i-x)\right] \\
& \qquad + \sum_{j\in J_m\setminus I_m} \sum_{i\notin J_m}  \E\left[|F(S_j)|G(S_i-S_j)G(S_i-x)\right]. 
\end{align*}
The first sum is handled as in Step 4. Namely,
\begin{align*}
& \sum_{m=0}^M  \sum_{j\in I_m} \sum_{|j-i|>R^3}  \E\left[|F(S_j)| G(S_i-S_j)G(S_i-x)\right] \\
& \le \sum_{j\ge 0}  \sum_{|j-i|>R^3}  \E\left[|F(S_j)| G(S_i-S_j)G(S_i-x)\right]  \stackrel{\eqref{lem.thm.asymp.2bis}}{\lesssim} \frac{I_F(\|x\|)}{R^{3/2}} .
\end{align*}
Similarly, defining $\tilde J_m:=\{mR^5,\dots, mR^5+R\} \cup \{(m+1)R^5-R,\dots,(m+1)R^5-1\}$, one has, 
\begin{align*}
&\sum_{m=0}^M \sum_{j\in J_m\setminus I_m} \sum_{i\notin J_m}  \E\left[|F(S_j)|G(S_i-S_j)G(S_i-x)\right]  \\
& \le \sum_{m=0}^M \sum_{j\in J_m\setminus I_m} \sum_{|i-j|>R}  \E\left[|F(S_j)|G(S_i-S_j)G(S_i-x)\right] \\
& \quad +\sum_{m=0}^M \sum_{j\in J_m\setminus I_m} \sum_{i\notin J_m,\, |i-j|\le R}  \E\left[|F(S_j)|G(S_i-S_j)G(S_i-x)\right] \\
& \stackrel{\eqref{lem.thm.asymp.2bis},  \eqref{cond.F}}{\lesssim} \frac{I_F(\|x\|)}{\sqrt R}  + \sum_{m=0}^M \sum_{j\in J_m\setminus I_m} \sum_{i\notin J_m,\, |i-j|\le R}  \E\left[|F(S_i)|G(S_i-x)\right]\\
& \lesssim \frac{ I_F(\|x\|)}{\sqrt R}  +  R \sum_{m= 0}^M \sum_{i \in \tilde J_m}  \E\left[|F(S_i)|G(S_i-x)\right] \lesssim \frac{I_F(\|x\|)}{\sqrt R}  + \frac{R^5}{1+\|x\|^{d-2}},
\end{align*}
using for the last inequality the same argument as in \eqref{final.step3}. 
Note also that 
$$\E[|F(\tilde S_\tau)|\1\{\tau<\infty\}] \stackrel{\eqref{lem.hit.1}}{\le} \sum_{i\ge 0} \E[|F(S_i)|G(S_i-x)]\lesssim I_F(\|x\|). $$
Therefore, putting all pieces together yields 
$$|K_R - \tilde A| \lesssim \frac{I_F(\|x\|)}{\sqrt{R}}  + R^{\frac{5d}2}\cdot J_F(\| x\|) + \frac{R^{\frac{5(d-2)}{2}}}{1+\|x\|^{d-2}}.$$

\underline{Step 8.}
Altogether the previous steps show that for any $R$ large enough, any $\ell \ge 1$, and any $x\in\Z^d$, satisfying $\ell\wedge \|x\|^2 \ge R^6$, 
$$|A-\tilde A| \lesssim \left(\frac{1}{\sqrt{R}} + \frac{R^{6+\frac{5d}{2}}}{1+\|x\|}\right) \cdot I_F(\|x\|) + \frac{R^{\frac{5(d-2)}{2}}}{1+\|x\|^{d-2}} + R^{6+\frac{5d}{2}}\cdot J_F(\|x\|). $$ 
The proof of the theorem follows by taking for $R$ a sufficiently small power of $\|x\|\wedge \ell$, and observing that 
for any function $F$ satisfying \eqref{cond.F}, one has $\liminf_{\|z\|\to \infty} |F(z)|/\|z\|>0$, and thus also $I_F(\|x\|) \ge \frac{c}{1+\|x\|^{d-3}}$. 
\end{proof}

It amounts now to give the proofs of Lemmas \ref{lem.thm.asymptotic} and \ref{lem.potential}.

\begin{proof}[Proof of Lemma \ref{lem.thm.asymptotic}] 
We start with the proof of \eqref{lem.thm.asymp.1}. Recall the definition of $\chi_d$ given just above Theorem \ref{thm.asymptotic}. One has for any $i\ge 0$, 
\begin{align*}
&\E\left[\sum_{j= i+1}^\infty G(S_j-S_i) \frac{\|S_j-S_i\|}{1+\|S_j\|} \mid S_i \right] \stackrel{\eqref{Green}}{\lesssim} 
 \E\left[\sum_{j= i+1}^\infty  \frac{1}{(1+\|S_j - S_i\|^{d-3})(1+\|S_j\|)} \mid S_i \right] \\
& \lesssim  \sum_{z\in \Z^d} G(z) \frac{1}{(1+\|z\|^{d-3})(1+\|S_i+z\|)}\stackrel{\eqref{Green}}{\lesssim}  \frac{\chi_d(\|S_i\|)}{1+\|S_i\|}, 
\end{align*}
and moreover, 
\begin{align}\label{lem.FG}
\nonumber &\sum_{i=0}^\infty \E\left[\frac{|F(S_i)|\chi_d(\|S_i\|)}{1+\|S_i\|}G(S_i-x)\right] = \sum_{z\in \Z^d} G(z) \frac{|F(z)|\chi_d(\|z\|)}{1+\|z\|} G(z-x) \\
 \nonumber \stackrel{\eqref{Green}}{\lesssim}&  \frac{ \chi_d(\|x\|)}{1+\|x\|^{d-2}} \sum_{\|z\|\le \frac{\|x\|}{2}} \frac{|F(z)|}{1+\|z\|^{d-1}} +  \sum_{\|z\|\ge \frac{\|x\|}{2}} \frac{|F(z)|\chi_d(\|z\|)}{1+\|z\|^{2d-3}} \\
\nonumber & \qquad + \frac{\chi_d(\|x\|)}{1+\|x\|^{d-1}} \sum_{\|z-x\|\le \frac{\|x\|}{2}} \frac{|F(z)|}{1+\|z-x\|^{d-2}} \\
 \stackrel{\eqref{cond.F}}{\lesssim} &  J_F(\|x\|/2) +  \frac{|F(x)|\chi_d(\|x\|)}{1+\|x\|^{d-3}}\lesssim  J_F(\|x\|),
\end{align}
where the last inequality follows from the fact that by \eqref{cond.F}, 
$$ \int_{\|x\|/2}^{\|x\|} \frac{\overline F(s)\chi_d(s)}{s^{d-2}} \, ds\, \asymp\, \frac{|F(x)|\chi_d(\|x\|)}{1+\|x\|^{d-3}}\, \asymp\, \frac{\chi_d(\|x\|)}{1+\|x\|^{d-2}} \int_{\|x\|/2}^{\|x\|} \overline F(s)\, ds.$$  
Thus 
$$\sum_{i=0}^\infty \sum_{j=i+1}^\infty \E\left[G(S_j-S_i) \frac{\|S_j-S_i\|}{1+\|S_j\|} |F(S_i)| G(S_i-x) \right] = \OO(J_F(\|x\|)).$$
On the other hand, for any $j\ge 0$, 
\begin{align}\label{lem.FG.2}
\nonumber &\E\left[\sum_{i= j+1}^\infty G(S_j-S_i) \|S_j-S_i\| \cdot |F(S_i)|G(S_i-x)  \mid S_j \right] \\
 \nonumber \stackrel{\eqref{Green}}{\lesssim} & 
\sum_{i=j+1}^\infty \E\left[   \frac{|F(S_i)|G(S_i-x)}{1+\|S_j - S_i\|^{d-3}} \mid S_j \right] 
 \stackrel{\eqref{Green}}{\lesssim}    \sum_{z\in \Z^d} \frac{|F(S_j+z)| G(S_j + z-x)}{1+\|z\|^{2d-5}}\\
\nonumber  \stackrel{\eqref{cond.F}, \eqref{Green}}{\lesssim} &   \sum_{z\in \Z^d}  
 \frac{|F(S_j)|}{(1+\|z\|^{2d-5})(1+\|S_j+z-x\|^{d-2})} + \frac{1}{1+\|S_j\|^{2d-5}}\sum_{\|u\|\le \|S_j\|}\frac{|F(u)| }{1+\|u-x\|^{d-2}} \\
  \nonumber \stackrel{\eqref{cond.F}}{\lesssim}  &  \frac{|F(S_j)|\chi_d(\|S_j-x\|)}{1+\|S_j-x\|^{d-2}} + \frac{\1\{\|S_j\|\le \|x\|/2\}\cdot |F(S_j)|}{(1+\|x\|^{d-2})(1+\|S_j\|^{d-5})}  \\
  \nonumber   & \qquad  + \frac{\1\{\|S_j\|\ge \|x\|/2\}}{1+\|S_j\|^{2d-5}}\left(|F(x)|(1+\| x\|^2) + |F(S_j)|(1+\|S_j\|^2)\right) \\
   \lesssim   & \frac{|F(S_j)|\chi_d(\|S_j-x\|)}{1+\|S_j-x\|^{d-2}} + \frac{\1\{\|S_j\|\le \|x\|\} |F(S_j)|}{1+\|x\|^{d-2}} + \frac{\1\{\|S_j\|\ge \|x\| \}  |F(S_j)|}{1+\|S_j\|^{d-2}},
\end{align}
where for the last two inequalities we used that  by \eqref{cond.F}, if $\|u\|\le \|v\|$, then $|F(u)|\lesssim | F(v)| (1+\|v\|)/(1+\|u\|)$, and also that $d\ge 5$ for the last one.  
Moreover, for any $r\ge 0$
$$\sum_{j=0}^\infty \E\left[\frac {\1\{\|S_j\|\le r\} \cdot |F(S_j)|  }{1+\|S_j\|} \right] = \sum_{\|z\| \le r} \frac{G(z)|F(z)|}{1+\|z\|} \stackrel{\eqref{Green}}{=}\OO \left(\int_0^{r} \overline F(s)\, ds\right),$$ 
$$\sum_{j=0}^\infty \E\left[\frac {\1\{\|S_j\|\ge r\} \cdot |F(S_j)|  }{1+\|S_j\|^{d-1}} \right] = \sum_{\|z\| \ge r} \frac{G(z)|F(z)|}{1+\|z\|^{d-1}} \stackrel{\eqref{Green}}{=}\OO \left(\int_{r}^\infty \frac{\overline F(s)}{s^{d-2}}\, ds\right).$$
Using also similar computations as in \eqref{lem.FG} to handle the first term in \eqref{lem.FG.2}, we conclude that 
$$\sum_{j=0}^\infty \sum_{i=j+1}^\infty \E\left[G(S_j-S_i) \frac{\|S_j-S_i\|}{1+\|S_j\|} |F(S_i)| G(S_i-x) \right] = \OO(J_F(\|x\|)),$$
which finishes the proof of \eqref{lem.thm.asymp.1}.

We then move to the proof of \eqref{lem.thm.asymp.2}. First note that for any $i\ge 0$, 
$$\E\left[\sum_{j\ge i+R} G(S_j-S_i)\mid S_i \right] = \E\left[\sum_{j\ge R} G(S_j) \right] \stackrel{\eqref{exp.Green}}{=} \OO\left(R^{\frac{4-d}{2}}\right),$$
and furthermore, 
\begin{equation}\label{lem.FG.3}
\sum_{i=0}^\infty \E[|F(S_i)|G(S_i-x)] = \sum_{z\in \Z^d} |F(z)| G(z-x)G(z) \stackrel{ \eqref{cond.F},\, \eqref{Green}}{=} \OO(I_F(\|x\|)),
\end{equation}
which together give the desired upper bound for the sum on the set $\{0\le i\le j-R\}$. 
On the other hand, for any $j\ge 0$, we get as for \eqref{lem.FG.2}, 
\begin{align*}
&\E\left[\sum_{i\ge j+R} G(S_j-S_i)|F(S_i)|G(S_i-x) \mid S_j \right] = \sum_{z\in \Z^d} G(z)|F(S_j+z)| G(S_j+z-x) G_R(z) \\
& \stackrel{\eqref{Green}}{\lesssim} \frac{1}{R^{\frac{d-4}{2}}} \cdot \sum_{z\in \Z^d} \frac{|F(S_j+z)|}{ (1+\|z\|^d)(1+\|S_j+z-x\|^{d-2})}\\
& \stackrel{ \eqref{cond.F}}{\lesssim} \frac{1}{R^{\frac{d-4}{2}}} \left\{  \sum_{z\in \Z^d} \frac{|F(S_j)|}{ (1+\|z\|^d)(1+\|S_j+z-x\|^{d-2})} + \frac 1{1+\|S_j\|^d} \sum_{\|u\|\le \|S_j\|} \frac{|F(u)|}{1+\|u-x\|^{d-2}}\right\}\\
&\lesssim  \frac{1}{R^{\frac{d-4}{2}}} \left\{ \frac{|F(S_j)|\log (2+\|S_j-x\|)}{1+\|S_j-x\|^{d-2}} + \frac{|F(S_j)|}{1+\|x\|^{d-2}+\|S_j\|^{d-2}} \right\}.
\end{align*}
Then similar computation as above, see e.g. \eqref{lem.FG.3}, give 
\begin{equation}\label{FSj}
\sum_{j\ge 0} \E\left[\frac{|F(S_j)|\log (2+\|S_j-x\|)}{1+\|S_j-x\|^{d-2}}\right] = \OO(I_F(\|x\|)),
\end{equation}
\begin{equation*}
\sum_{j\ge 0} \E\left[\frac{|F(S_j)|}{1+\|x\|^{d-2} + \|S_j\|^{d-2}}\right] = \OO(I_F(\|x\|)),
\end{equation*}
which altogether proves \eqref{lem.thm.asymp.2}.

The proof of \eqref{lem.thm.asymp.2bis} is entirely similar: on one hand, for any $i\ge 0$, 
\begin{align*}
 & \E\left[\sum_{j= i+R}^\infty G(S_j-S_i) |F(S_j)| \mid S_i \right] \stackrel{\eqref{cond.F}}{\lesssim}  \E\left[\sum_{j= i+R}^\infty G(S_j-S_i) \frac{\|S_j-S_i\|}{1+\|S_j\|} \mid S_i \right]  |F(S_i)| \\
 & \lesssim  \sum_{z\in \Z^d} G_R(z) \frac{|F(S_i)|}{(1+\|z\|^{d-3})(1+\|S_i+z\|)} \\
 & \lesssim \sum_{z\in \Z^d} \frac{|F(S_i)|}{(R^{\frac{d-2}{2}} + \|z\|^{d-2})(1+\|z\|^{d-3})(1+\|S_i+z\|)} \lesssim \frac{|F(S_i)|}{R^{\frac{d-4}{2}}},
 \end{align*}
and together with \eqref{FSj}, this yields  
 $$\sum_{i=0}^\infty \sum_{j=i+R}^\infty \E\left[G(S_j-S_i) |F(S_j)| G(S_i-x) \right]\lesssim \frac{I_F(\|x\|)}{R^{\frac{d-4}{2}} }.$$
On the other hand, for any $j\ge 0$, using \eqref{Green}, 
\begin{align*} 
\E\left[\sum_{i\ge j+R} G(S_j-S_i)G(S_i-x) \mid S_j \right] \lesssim  \sum_{z\in \Z^d} \frac{G(S_j + z-x)}{R^{\frac{d-4}{2}}(1+\|z\|^d)} 
 \lesssim  \frac{\log (2+ \|S_j-x\|)}{R^{\frac{d-4}{2}}(1+\|S_j-x\|^{d-2})},
\end{align*}
and we conclude the proof of \eqref{lem.thm.asymp.2bis} using \eqref{FSj} again. 
\end{proof}

\begin{proof}[Proof of Lemma \ref{lem.potential}] The first statement follows directly from \eqref{Green.asymp} and the last-exit decomposition (see Proposition 4.6.4 (c) in \cite{LL}): 
$$\bP_y[H_\Lambda<\infty] = \sum_{x\in \Lambda} G(y-x) e_\Lambda(x).$$ 
Indeed if $\|y\|>2 \text{rad}(\Lambda)$, using \eqref{Green} we get $G(y-x)\le C\|y\|^{2-d}$, for some constant $C>0$ independent of $x\in \Lambda$, which gives well \eqref{cap.hitting}, since by definition $\sum_{x\in \Lambda} e_\Lambda(x) = \cp(\Lambda)$.

The second statement is more involved. Note that one can always assume $\mathcal J(y)>C\text{rad}(\Lambda)$, for some constant $C>0$, for otherwise the result is trivial. We use similar notation as in \cite{LL}. In particular $G_A(x,y)$ denotes the Green's function restricted to a subset $A\subseteq \Z^d$, that is the expected number of visits to $y$ before exiting $A$ for a random walk starting from $x$, and $H_A(x,y)=\bP_x[H_{A^c}=y]$. We also let $\CC_n$ denote the (discrete) ball of radius $n$ for the norm $\mathcal J(\cdot)$. Then exactly as in \cite{LL} (see Lemma 6.3.3 and Proposition 6.3.5 thereof), one can see using \eqref{Green.asymp} that for all $n\ge 1$, 
\begin{equation}\label{GA}
\left| G_{\CC_{n}}(x,w) - G_{\CC_{n}}(0,w) \right| \le C \frac{\|x\|}{1+\|w\|} \, G_{\CC_{n}}(0,w),
\end{equation}
for all $x\in \CC_{n/4}$, and all $w$ satisfying $2\mathcal J(x) \le \mathcal J(w) \le n/2$. One can then derive an analogous estimate for the (discrete) derivative of $H_{\CC_n}$. Define  $A_n=\CC_n\setminus \CC_{n/2}$, and $\rho = H^+_{A_n^c}$. By the last-exit decomposition (see \cite[Lemma 6.3.6]{LL}), one has for $x\in \CC_{n/8}$ and $z\notin \CC_n$,
\begin{align*}
\nonumber & |H_{\CC_n}(x,z) - H_{\CC_n}(0,z)|\le \sum_{w\in \CC_{n/2}}|G_{\CC_n}(x,w) - G_{\CC_n}(0,w)|\cdot  \bP_w[S_\rho = z]\\
\nonumber  & \stackrel{\eqref{GA},  \eqref{Green}}{\lesssim} \frac{\|x\|}{n}\cdot H_{\CC_n}(0,z) +  \sum_{2\mathcal J(x)\le \mathcal J(w)\le \frac n4}  \frac{\|x\|}{\|w\|^{d-1}} \bP_w[S_\rho = z] \\
& \qquad + \sum_{\mathcal J(w) \le 2\mathcal J(x)} \left(\frac 1{1+\|w-x\|^{d-2}} + \frac 1{1+\|w\|^{d-2}}\right)\bP_w[S_\rho = z]. 
\end{align*}
Now, observe that for any $y\notin \CC_n$, any $w\in \CC_{n/4}$, and any $A\subseteq \Z^d$, 
\begin{align*}
 \sum_{z\notin \CC_n} G_{A} (y,z)  \bP_w[S_\rho = z]  \lesssim \sum_{z\notin \CC_n} \bP_w[S_\rho = z] \lesssim \bP_w[\mathcal J(S_1)>\frac n2] \lesssim \bP[\mathcal J(X_1)> \frac n4] \lesssim n^{-d}, 
\end{align*}
using that by hypothesis $\mathcal J(X_1)$ has a finite $d$-th moment. 
It follows from the last two displays that 
\begin{align}\label{potentiel.3}
\sum_{z\notin \CC_n} G_{A} (y,z) H_{\CC_n}(x,z) = \left(\sum_{z\notin \CC_n} G_{A} (y,z) H_{\CC_n}(0,z)\right)\left(1+\OO\Big(\frac{\|x\|}{n}\Big)\right)+ \OO\left( \frac{\|x\|}{n^{d-1}} \right). 
\end{align}
Now let $\Lambda$ be some finite subset of $\Z^d$ containing the origin, and let $m:=\sup\{\mathcal J(u)\, :\, \|u\|\le 2\text{rad}(\Lambda)\}$. Note that $m=\OO(\text{rad}(\Lambda))$, and thus one can assume $\mathcal J(y)>16m$. Set $n:=\mathcal J(y)-1$. 
Using again the last-exit decomposition and symmetry of the step distribution, we get for any $x\in \Lambda$,   
\begin{align}\label{potentiel.4}
\bP_y[S_{H_\Lambda} =x,\, H_\Lambda<\infty] = \sum_{z\notin \CC_n} G_{\Lambda^c}(y,z) \bP_x[S_{\tau_n} = z,\, \tau_n<H_\Lambda^+],
\end{align}
with $\tau_n:= H_{\CC_n^c}$. 
We then write, using the Markov property, 
\begin{align}\label{potentiel.5}
\nonumber \bP_x[S_{\tau_n} = z,\, \tau_n<H_\Lambda^+]& =\sum_{x'\in \CC_{n/8}\setminus \CC_m}\bP_x[\tau_m<H_\Lambda^+,\, S_{\tau_m} =x']\cdot \bP_{x'}[S_{\tau_n} = z,\, \tau_n<H_\Lambda^+] \\
& \qquad + \bP_x\left[\mathcal J(S_{\tau_m}) >\frac n8,\, S_{\tau_n}=z\right],
\end{align}
with $\tau_m:=H_{\CC_m^c}$. Concerning the last term we note that 
\begin{align}\label{potentiel.5.bis}
\nonumber & \sum_{z\notin \CC_n} G_{\Lambda^c}(y,z) \bP_x\left[\mathcal J(S_{\tau_m}) >\frac n8,\, S_{\tau_n}=z\right] \\
\nonumber & \stackrel{\eqref{Green.hit}}{\le} \sum_{z\notin \CC_n} G(z-y) \left\{\bP_x[S_{\tau_m}=z] + \sum_{u\in \CC_n\setminus \CC_{n/8}} \bP_x[S_{\tau_m} =u]G(z-u)\right\}\\
\nonumber & \stackrel{\text{Lemma }\ref{lem.upconvolG}}{\lesssim} \sum_{z\notin \CC_n} G(z-y) \bP_x[S_{\tau_m}=z] + \sum_{u\in \CC_n\setminus \CC_{n/8}} \frac{\bP_x[S_{\tau_m} =u]}{\|y-u\|^{d-4}} \\
\nonumber & \lesssim \bP_x[\mathcal J(S_{\tau_m})>n/8] \lesssim \sum_{u\in \CC_m} G_{\CC_m}(x,u) \bP[J(X_1)>\frac n8 - m] \\
& \stackrel{\eqref{Green}}{=}\OO\left(\frac{m^2}{n^d}\right) = \OO\left(\frac{m}{n^{d-1}}\right),
\end{align}
applying once more the last-exit decomposition at the penultimate line, and the hypothesis that $\mathcal J(X_1)$ has a finite $d$-th moment at the end.  
We handle next the sum in the right-hand side of \eqref{potentiel.5}. First note that \eqref{potentiel.3} gives for any $x'\in \CC_{n/8}$, 
\begin{align}\label{potentiel.6}
\nonumber  \sum_{z\notin \CC_n} & G_{\Lambda^c}(y,z)\bP_{x'}[S_{\tau_n} = z]  =  \sum_{z\notin \CC_n} G_{\Lambda^c}(y,z)H_{\CC_n}(x', z)\\
& = \left(\sum_{z\notin \CC_n} G_{\Lambda^c}(y,z)H_{\CC_n}(0, z)\right)\left(1+\OO\Big(\frac{\|x'\|}n\Big)\right) + \OO\left( \frac{\|x'\|}{n^{d-1}} \right). 
\end{align}
Observe then two facts. On one hand, by the last exit-decomposition and symmetry of the step distribution, 
\begin{equation}\label{pot.1}
\sum_{z\notin \CC_n} G_{\Lambda^c}(y,z)H_{\CC_n}(0, z) \le \sum_{z\notin \CC_n} G_{\Z^d\setminus \{0\}}(y,z)H_{\CC_n}(0, z) = \bP[H_y<\infty] \stackrel{\eqref{Green.hit}, \eqref{Green}}{\lesssim} n^{2-d},
\end{equation}
and on the other hand by Proposition 4.6.2 in \cite{LL}, 
\begin{align}\label{pot.2}
&   \sum_{z\notin \CC_n}  G_{\Lambda^c}(y,z) H_{\CC_n}(0, z) \\
\nonumber & =  \sum_{z\notin \CC_n} G_{\Z^d\setminus \{0\}}(y,z)H_{\CC_n}(0, z)  + \sum_{z\notin \CC_n} \left(G_{\Lambda^c}(y,z) - G_{\Z^d\setminus \{0\}}(y,z)\right) H_{\CC_n}(0, z) \\
\nonumber & \ge  \bP[H_y<\infty] - \OO\left(\bP_y[H_\Lambda<\infty]\sum_{z\notin \CC_n}  G(z) H_{\CC_n}(0, z)\right) \\
\nonumber & \stackrel{\eqref{hit.ball}}{\ge}  \bP[H_y<\infty] - \OO\left(n^{2-d}\sum_{z\notin \CC_n}  G(z)^2\right) 
 \stackrel{\eqref{Green.asymp}}{\ge} \frac{c}{n^{d-2}}. 
\end{align}
This last fact, combined with \eqref{potentiel.6} gives therefore, for any $x'\in \CC_{n/8}$, 
\begin{align}\label{potentiel.6.bis}
\sum_{z\notin \CC_n} G_{\Lambda^c}(y,z)\bP_{x'}[S_{\tau_n} = z]  = \left(\sum_{z\notin \CC_n} G_{\Lambda^c}(y,z)H_{\CC_n}(0, z)\right)\left(1+\OO\Big(\frac{\|x'\|}n\Big)\right). 
\end{align}
By the Markov property, we get as well  
\begin{align*}
 \sum_{z\notin \CC_n} G_{\Lambda^c}(y,z)\bP_{x'}[S_{\tau_n} = z\mid H_\Lambda<\tau_n ] = \left(\sum_{z\notin \CC_n} G_{\Lambda^c}(y,z)H_{\CC_n}(0, z)\right)\left(1+\OO\Big(\frac{m}n\Big)\right), 
\end{align*}
since by definition $\Lambda\subseteq \CC_m\subset \CC_{n/8}$, and thus 
\begin{align*}
 \sum_{z\notin \CC_n} & G_{\Lambda^c}(y,z)\bP_{x'}[S_{\tau_n} = z,\, H_\Lambda<\tau_n ] \\
 & = \bP_{x'}[H_\Lambda<\tau_n]\left( \sum_{z\notin \CC_n} G_{\Lambda^c}(y,z)H_{\CC_n}(0, z)\right)\left(1+\OO\Big(\frac{m}n\Big)\right).  
\end{align*}
Subtracting this from \eqref{potentiel.6.bis}, we get for $x'\in \CC_{n/8}\setminus \CC_m$,
\begin{align*}
 \sum_{z\notin \CC_n} & G_{\Lambda^c}(y,z)\bP_{x'}[S_{\tau_n} = z,\, \tau_n<H_\Lambda ] \\
 & = \bP_{x'}[\tau_n<H_\Lambda]\left( \sum_{z\notin \CC_n} G_{\Lambda^c}(y,z)H_{\CC_n}(0, z)\right)\left(1+\OO\Big(\frac{\|x'\|}n\Big)\right),
\end{align*}
since by \eqref{hit.ball}, one has $\bP_{x'}[\tau_n<H_\Lambda]>c$, for some constant $c>0$, for any $x'\notin \CC_m$ (note that the stopping time theorem gives in fact 
$\bP_{x'}[H_\Lambda<\infty] \le G(x') / \inf_{\|u\|\le \text{rad}(\Lambda)} G(u)$, and thus by using \eqref{Green.asymp}, one can ensure $\bP_{x'}[H_\Lambda<\infty]\le 1-c$,  by taking $\|x'\|$ large enough, which is always possible). 
Combining this with \eqref{potentiel.4}, \eqref{potentiel.5} and \eqref{potentiel.5.bis}, and using as well \eqref{pot.1} and \eqref{pot.2}, we get 
\begin{align*}
& \bP_y[S_{H_\Lambda} =x,\, H_\Lambda<\infty]   =  \bP_x[\tau_n<H_\Lambda]\left( \sum_{z\notin \CC_n} G_{\Lambda^c}(y,z)H_{\CC_n}(0, z)\right) \\
& \qquad + \OO\left(\frac 1{n^{d-1}}\sum_{x'\in\CC_{n/8} \setminus \CC_m} \bP_x[S_{\tau_m}=x'] \cdot \|x'\| \right) + \OO\left(\frac{m}{n^{d-1}}\right)\\
 \stackrel{\eqref{hit.ball}}{=} & e_\Lambda(x)\left( \sum_{z\notin \CC_n} G_{\Lambda^c}(y,z)H_{\CC_n}(0, z)\right)\left(1+\OO\Big(\frac{m}{n}\Big)\right)+ \OO\left(\frac 1{n^{d-1}}\sum_{r=2m}^{n/8}  \frac{m^2}{r^{d-1}} \right) + \OO\left(\frac{m}{n^{d-1}}\right)\\
 = & \ e_\Lambda(x)\left( \sum_{z\notin \CC_n} G_{\Lambda^c}(y,z)H_{\CC_n}(0, z)\right)\left(1+\OO\Big(\frac{m}{n}\Big)\right),
\end{align*}
using the same argument as in \eqref{potentiel.5.bis} for bounding $\bP_{x'}[\mathcal J(S_{\tau_m})\ge r]$, when $r\ge 2m$. 
Summing over $x\in \Lambda$ gives 
\begin{align*}
\bP_y[H_\Lambda<\infty]=  \cp(\Lambda) \left( \sum_{z\notin \CC_n} G_{\Lambda^c}(y,z)H_{\CC_n}(0, z)\right)\left(1+\OO\Big(\frac{m}{n}\Big)\right),
\end{align*}
and the proof of the lemma follows from the last two displays.  
\end{proof}

%%%%%%%%%%%%%%%%%%%%%%%%%%%%%%%%%%%%%%%%%%%%%%%%%%%%%%%%%%%%%%%%%%%%%%%

\section{Proof of Proposition \ref{prop.phipsi.2}} 
The proof is divided in four steps, corresponding to the next four lemmas. 

\begin{lemma} \label{lem.var.1}
Assume that  $\epsilon_k\to \infty$, and $\epsilon_k/k\to 0$. There exists a constant $\sigma_{1,3}>0$, such that  
$$\cov(Z_0\phi_3, Z_k\psi_1) \sim \frac{\sigma_{1,3}}{k}.$$
\end{lemma}

\begin{lemma} \label{lem.var.2}
There exist positive constants $\delta$ and $\sigma_{1,1}$, such that when $\epsilon_k\ge k^{1-\delta}$, and $\epsilon_k/k\to 0$,    
$$ \cov(Z_0\phi_1,Z_k\psi_1) \sim \cov(Z_0\phi_3, Z_k\psi_3) \sim \frac{\sigma_{1,1}}{k}.$$
\end{lemma}

\begin{lemma}\label{lem.var.3}
There exist positive constants $\delta$ and $\sigma_{1,2}$, such that when $\epsilon_k\ge k^{1-\delta}$, and $\epsilon_k/k\to 0$,     
$$ \cov(Z_0\phi_2,Z_k\psi_1) \sim \cov(Z_0\phi_3, Z_k\psi_2) \sim \frac{\sigma_{1,2}}{k}.$$
\end{lemma}

\begin{lemma}\label{lem.var.4}
There exist positive constants $\delta$ and $\sigma_{2,2}$, such that when $\epsilon_k\ge k^{1-\delta}$, and $\epsilon_k/k\to 0$,   
$$ \cov(Z_0\phi_2,Z_k\psi_2) \sim  \frac{\sigma_{2,2}}{k}.$$
\end{lemma}

\subsection{Proof of Lemma \ref{lem.var.1}}
We assume now to simplify notation that the distribution $\mu$ is aperiodic, but it should be clear from the proof that the case of a bipartite walk could be handled similarly.

The first step is to show that 
\begin{equation}\label{var.1.1}
 \cov(Z_0\phi_3, Z_k \psi_1) = \rho^2 \left\{\sum_{x\in \Z^5} p_k(x) \varphi_x^2 - \left(\sum_{x\in \Z^5} p_k(x) \varphi_x\right)^2\right\} + o\left(\frac 1k\right), 
 \end{equation}
where $\rho$ and $\varphi_x$ are defined respectively as   
\begin{eqnarray}\label{def.rho}
\rho:= \E\left[ \bP\left[H^+_{\overline \RR_\infty } = \infty \mid (S_n)_{n\in \Z} \right] \cdot \1\{S_\ell \neq 0,\, \forall \ell \ge 1\}\right] ,
\end{eqnarray}
and 
$$\varphi_x : = \bP_{0,x} [\RR_\infty \cap \tilde \RR_\infty \neq \emptyset].$$
To see this, one needs to dissociate $Z_0$ and $Z_k$, as well as the events of avoiding  $\RR[-\epsilon_k,\epsilon_k]$ and $\RR[k-\epsilon_k,k+\epsilon_k]$ by two independent walks starting respectively from the origin and from $S_k$,  
which are local events (in the sense that they only concern small parts of the different paths), from the events of hitting $\RR[k+1,\infty)$ and $\RR(-\infty,-1]$ by these two walks, which involve different parts of the trajectories.

To be more precise, consider $(S_n^1)_{n\ge 0}$ and $(S_n^2)_{n\ge 0}$, two independent random walks starting from the origin, and independent of $(S_n)_{n\in \Z}$. Then define 
$$\tau_1:= \inf\{n\ge \epsilon_k : S_n^1 \in \RR[k+\epsilon_k,\infty) \}, \ \tau_2:= \inf \{n\ge \epsilon_k  :  S_k + S_n^2 \in \RR(-\infty,-\epsilon_k]\}.$$
We first consider the term $\E[Z_0\phi_3]$. Let  
$$\tau_{0,1}:=\inf \left\{n\ge \epsilon_k \, :\, S_n^1 \in \RR[-\epsilon_k,\epsilon_k]\right\},$$
and  
$$\Delta_{0,3}:=  \E\left[Z_0\cdot \1\{\RR^1[1,\epsilon_k]\cap \RR[-\epsilon_k,\epsilon_k] =\emptyset\}\cdot \1\{ \tau_1 <\infty\}\right].$$
One has,  
\begin{align*}
\left| \E[Z_0\phi_3] - \Delta_{0,3}\right| &\le \bP\left[\tau_{0,1}<\infty,\, \tau_1<\infty \right]  + \bP\left[ \RR^1[0,\epsilon_k] \cap \RR[k,\infty)\neq \emptyset \right] \\
& \qquad +\bP[\RR^1_\infty \cap \RR[k,k+\epsilon_k]\neq \emptyset] \\
& \stackrel{\eqref{lem.hit.3}}{\le}  \bP[\tau_1\le \tau_{0,1}<\infty] +  \bP[\tau_{0,1}\le \tau_1<\infty ] + \OO\left(\frac{\epsilon_k}{k^{3/2}} \right). 
\end{align*}
Next, conditioning on $\RR[-\epsilon_k,\epsilon_k]$ and using the Markov property at time $\tau_{0,1}$, we get with $X=S_{\epsilon_k} - S^1_{\tau_{0,1}}$, 
\begin{align*}
\bP[\tau_{0,1}\le \tau_1<\infty ] &\le  \E\left[\bP_{0,X}[\RR[k,\infty)\cap \tilde \RR_\infty \neq \emptyset] \cdot \1\{\tau_{0,1}<\infty\}\right] \\
& \stackrel{\eqref{lem.hit.3}}{=} 
\OO\left(\frac{\bP[\tau_{0,1}<\infty] }{\sqrt{k}}\right) \stackrel{\eqref{lem.hit.3}}{=}  \OO\left(\frac 1{\sqrt {k \epsilon_k}} \right).
\end{align*}
Likewise, using the Markov property at time $\tau_1$, we get 
\begin{align*}
& \bP[\tau_1\le \tau_{0,1}<\infty ] \stackrel{\eqref{lem.hit.1}}{\le} \E\left[\left(\sum_{j=-\epsilon_k}^{\epsilon_k} G(S_j- S_{\tau_1}^1)\right) \1\{\tau_1<\infty\}\right] \\
& \stackrel{\eqref{lem.hit.1}}{\le}\sum_{i=k+\epsilon_k}^\infty \sum_{j=-\epsilon_k}^{\epsilon_k} \E\left[G(S_j- S_i)G(S_i-S^1_{\epsilon_k})\right]\\
& \le (2\epsilon_k+1) \sup_{x\in \Z^5} \sum_{i=k}^\infty \E\left[G(S_i)G(S_i-x)\right]  \\
& \le (2\epsilon_k+1) \sup_{x\in \Z^5} \sum_{z\in \Z^5} G(z) G(z-x) G_k(z)  \stackrel{\eqref{Green}, \,  \text{Lemma }\ref{lem.upconvolG}}{=}\OO\left(\frac{\epsilon_k}{k^{3/2}}\right).
\end{align*}
Now define for any $y_1,y_2\in \Z^5$, 
\begin{equation}\label{Hy1y2}
H(y_1,y_2):= \E\left[Z_0 \1\{\RR^1[1,\epsilon_k]\cap \RR[-\epsilon_k,\epsilon_k] =\emptyset, S_{\epsilon_k }= y_1, S^1_{\epsilon_k} = y_2 \}\right].
\end{equation}
One has by the Markov property 
\begin{align*}
\Delta_{0,3} =\sum_{x\in \Z^5} \sum_{y_1,y_2\in \Z^5} H(y_1,y_2) p_k(x+y_2-y_1) \varphi_x.
\end{align*}
Observe that typically $\|y_1\|$ and $\|y_2\|$ are much smaller than $\|x\|$, and thus $p_k(x+y_2-y_1)$ should be also typically close to $p_k(x)$. 
To make this precise, consider $(\chi_k)_{k\ge 1}$ some sequence of positive integers, such that $\epsilon_k \chi_k^3 \le k$, for all $k\ge 1$, and $\chi_k\to \infty$, as $k\to \infty$. 
One has using Cauchy-Schwarz at the third line, 
\begin{align*}
& \sum_{\|x\|^2\le k/\chi_k} \sum_{y_1,y_2\in \Z^5}  H(y_1,y_2) p_k(x+y_2-y_1)\varphi_{x} \\
& \le \sum_{\|x\|^2\le k/\chi_k } \sum_{y_2\in \Z^5} p_{\epsilon_k}(y_2) p_{k+\epsilon_k}(x) \varphi_{x-y_2} 
\stackrel{\eqref{lem.hit.2}}{\lesssim}  \E\left[\frac{\1\{\|S_{k+\epsilon_k}\|^2 \le k/\chi_k\}}{1+\|S_{k+\epsilon_k} - S^1_{\epsilon_k}\|}\right]\\
& \lesssim   \E\left[\frac 1{1+\|S_{k+2\epsilon_k}\|^2}\right]^{1/2}\cdot \bP\left[\|S_{k+\epsilon_k}\|^2\le k/\chi_k\right]^{1/2}\stackrel{\eqref{pn.largex}}{\lesssim} \frac 1{\sqrt{k}\cdot \chi_k^{5/4}}. 
\end{align*}
Likewise, using just \eqref{Sn.large} at the end instead of \eqref{pn.largex}, we get 
$$\sum_{\|x\|^2 \ge k\chi_k} \sum_{y_1,y_2\in \Z^5}  H(y_1,y_2) p_k(x+y_2-y_1)\varphi_{x} \lesssim \frac 1{\sqrt{k}\cdot \chi_k^{5/4}},$$
and one can handle the sums on the sets $\{\|y_1\|^2\ge \epsilon_k\chi_k\}$ and $\{\|y_2\|^2\ge \epsilon_k\chi_k\}$ similarly. 
Therefore, it holds
$$\Delta_{0,3} =\sum_{k/\chi_k \le \|x\|^2 \le k \chi_k } \sum_{\|y_1\|^2\le \epsilon_k \chi_k} \sum_{\|y_2\|^2\le \epsilon_k \chi_k} H(y_1,y_2) p_k(x+y_2-y_1) \varphi_x +  \OO\left(\frac 1{\sqrt{k}\cdot \chi_k^{5/4}}  \right).$$ 
Moreover, Theorem \ref{LCLT} shows that for any $x,y_1,y_2$ as in the three sums above, one has  
$$|p_k(x+y_2-y_1)-p_k(x)| = \OO\left(\frac{\sqrt{\epsilon_k} \cdot \chi_k}{\sqrt k}\cdot p_k(x) + \frac 1{k^{7/2}}\right).$$
Note also that by \eqref{lem.hit.2}, one has 
\begin{equation}\label{sum.pkxphix}
\sum_{x,y_1,y_2\in \Z^5} H(y_1,y_2)p_k(x) \phi_x\le \sum_{x\in \Z^5} p_k(x) \varphi_x  = \OO\left(\frac 1{\sqrt{k}}\right).
\end{equation}
Using as well that $\sqrt{\epsilon_k}\chi_k \le \sqrt {k/\chi_k}$, and $\sum_{\|x\|^2 \le k\chi_k} \phi_x = \OO(k^2\chi_k^2)$, we get
\begin{align*}
\Delta_{0,3} =\rho_k  \sum_{x\in \Z^5  }  p_k(x) \varphi_x +  \OO\left(\frac{1}{\sqrt{k\cdot \chi_k}} + \frac{\chi_k^2}{k^{3/2}}\right),
\end{align*}
with  
$$\rho_k:= \sum_{y_1,y_2\in \Z^5} H(y_1,y_2) = \E\left[Z_0\cdot \1\{\RR^1[1,\epsilon_k] \cap \RR[-\epsilon_k,\epsilon_k] = \emptyset\}\right].$$
Note furthermore that one can always take $\chi_k$ such that $\chi_k=o( \sqrt k)$, and that by \eqref{Green.hit}, \eqref{Green} and \eqref{lem.hit.3}, one has 
$|\rho_k - \rho| \lesssim \epsilon_k^{-1/2}$. 
This gives 
\begin{eqnarray}\label{Z0phi3.final}
\E[Z_0\phi_3] = \rho \sum_{x\in \Z^5} p_k(x) \phi_x + o\left(\frac{1}{\sqrt{k}}\right).
\end{eqnarray}
By symmetry the same estimate holds for $\E[Z_k\psi_1]$, and thus using again \eqref{sum.pkxphix}, it entails
$$\E[Z_0\phi_3] \cdot \E[Z_k\psi_1] = \rho^2\left( \sum_{x\in \Z^5} p_k(x) \phi_x\right)^2 + o\left(\frac{1}{k}\right).$$ 
The estimate of $\E[Z_0\phi_3Z_k\psi_1]$ is done along the same line, but is a bit more involved. Indeed, let 
\begin{align*}
\Delta_{1,3}:= \E\left[\right. &Z_0Z_k  \1\{\RR^1[1,\epsilon_k]\cap \RR[-\epsilon_k,\epsilon_k] =\emptyset\}  \\
& \left. \times \1\{(S_k+\RR^2[1,\epsilon_k])\cap \RR[k-\epsilon_k,k+\epsilon_k] =\emptyset,  \tau_1 <\infty, \tau_2<\infty\}\right].
\end{align*}
The difference between $\E[Z_0\phi_3Z_k\psi_1]$ and $\Delta_{1,3}$ can be controlled roughly as above, but one needs additionally to handle the probability of $\tau_2$ being finite. 
Namely one has using symmetry, 
\begin{align}\label{Delta13}
& |\E[Z_0\phi_3Z_k\psi_1] - \Delta_{1,3} | \le  2\left(\bP\left[\tau_{0,1}<\infty,\, \tau_1<\infty, \, \overline \tau_2<\infty \right]  \right. \\
  \nonumber + & \bP\left[ \RR^1[0,\epsilon_k] \cap \RR[k,\infty)\neq \emptyset,\, \overline \tau_2<\infty \right]    + \left. \bP[\RR^1_\infty \cap \RR[k,k+\epsilon_k]\neq \emptyset, \, \overline \tau_2<\infty]\right),
\end{align}
with 
$$\overline \tau_2:= \inf\{n\ge 0 \, : \, S_k+S^2_n \in \RR(-\infty,0]\}.$$
The last term in \eqref{Delta13} is handled as follows: 
\begin{align*}
&\bP\left[\RR^1_\infty \cap \RR[k,k+\epsilon_k]\neq \emptyset,  \overline \tau_2<\infty\right] =\sum_{x\in \Z^5} 
\bP[\RR^1_\infty \cap \RR\left[k,k+\epsilon_k]\neq \emptyset,  \overline \tau_2<\infty, S_k=x\right] \\
& \stackrel{\eqref{lem.hit.1}}{\le} \sum_{x\in \Z^5} p_k(x)\phi_x \sum_{i=0}^{\epsilon_k} \E[G(S_i+x)]
\stackrel{\eqref{Green}, \eqref{lem.hit.2}, \eqref{exp.Green.x}}{\lesssim} \epsilon_k \sum_{x\in \Z^5} \frac{p_k(x)}{1+\|x\|^4}  \stackrel{\eqref{pn.largex}}{\lesssim} \frac{\epsilon_k}{k^2}. 
\end{align*}
The same arguments give as well  
$$\bP\left[ \RR^1[0,\epsilon_k] \cap \RR[k,\infty)\neq \emptyset,\, \overline \tau_2<\infty \right] \lesssim \frac{\epsilon_k}{k^2},$$
$$\bP\left[\tau_{0,1}<\infty,\, \tau_1<\infty, \, \overline \tau_2<\infty \right] = \bP\left[\tau_{0,1}<\infty,\, \tau_1<\infty, \, \tau_2<\infty \right] +  \OO\left(\frac{\epsilon_k}{k^2}\right).$$
Then we can write,  
\begin{align*}
&\bP\left[\tau_{0,1}\le  \tau_1<\infty,   \tau_2<\infty \right] =  \E \left[ \bP_{0,S_{k+\epsilon_k}-S_{\tau_{0,1}}}[\RR_\infty \cap \tilde \RR_\infty \neq \emptyset] 
\1\{\tau_{0,1}<\infty,  \tau_2<\infty\}\right] \\
& \stackrel{\eqref{lem.hit.1}, \eqref{lem.hit.2}}{\lesssim} \sum_{i=-\epsilon_k}^{\epsilon_k} \E\left[\frac {1}{1+\|S_{k+\epsilon_k}-S_i\|}\cdot \frac{G(S_i-S^1_{\epsilon_k})}{1+\|S_k-S_{-\epsilon_k}\|}\right]\\
& \stackrel{\eqref{exp.Green}}{\lesssim}\frac{1}{\epsilon_k^{3/2}} \sum_{i=-\epsilon_k}^{\epsilon_k}  \E\left[\frac {1}{1+\|S_k-S_i\|}\cdot \frac{1}{1+\|S_k-S_{-\epsilon_k}\|}\right] \\
& \lesssim \frac{1}{\sqrt{\epsilon_k}}  \max_{k-\epsilon_k\le j \le k+\epsilon_k}\, \sup_{u\in\Z^d}\,  \E\left[\frac {1}{1+\|S_j\|}\cdot \frac{1}{1+\|S_j+u\|}\right]
\lesssim \frac{1}{k\sqrt{\epsilon_k}},
\end{align*}
where the last equality follows from straightforward computations, using \eqref{pn.largex}. On the other hand,
\begin{align*}
&\bP\left[\tau_1\le \tau_{0,1}<\infty,   \tau_2<\infty \right] \stackrel{\eqref{lem.hit.1}, \eqref{lem.hit.2}}{\lesssim} \sum_{i=k+\epsilon_k}^\infty \sum_{j=-\epsilon_k}^{\epsilon_k} 
\E\left[\frac{G(S_j-S_i)G(S_i-S_{\epsilon_k}^1)}{1+\|S_k-S_{-\epsilon_k}\|}\right] \\
&  \stackrel{\eqref{Green}, \eqref{exp.Green.x}}{\lesssim}  \sum_{j=-\epsilon_k}^{\epsilon_k}  \sum_{i=k+\epsilon_k}^\infty 
\E\left[\frac{G(S_j-S_i)}{(1+\|S_i\|^3)(1+\|S_k-S_{-\epsilon_k}\|)}\right] \\
& \lesssim \sum_{j=-\epsilon_k}^{\epsilon_k}\sum_{z\in \Z^d} G_{\epsilon_k}(z) \E\left[\frac{G(z+ S_k-S_j)}{(1+\|z+S_k\|^3)(1+\|S_k-S_{-\epsilon_k}\|)}\right].
\end{align*}
Note now that for $x,y\in \Z^5$, by \eqref{Green} and Lemma \ref{lem.upconvolG},
\begin{align*}
 \sum_{z\in \Z^d} \frac{G_{\epsilon_k}(z)}{(1+\|z-x\|^3)(1+\|z-y\|^3)}  \lesssim \frac{1}{1+\|x\|^3} \left(\frac 1{\sqrt{\epsilon_k}} + \frac{1}{1+\|y-x\|}\right) .
\end{align*}
It follows that 
\begin{align*} 
&\bP\left[\tau_1\le \tau_{0,1}<\infty,   \tau_2<\infty \right]  \lesssim \sum_{j=-\epsilon_k}^{\epsilon_k}
\E\left[\frac{1}{(1+\|S_k\|^3)(1+\|S_k-S_{-\epsilon_k}\|)}\left(\frac 1{\sqrt{\epsilon_k}} + \frac 1{1+\|S_j\|}\right) \right]\\
& \stackrel{\eqref{exp.Green.x}}{\lesssim}  \E\left[\frac {\sqrt{\epsilon_k}}{1+\|S_k\|^4}\right] + \sum_{j=-\epsilon_k}^0 \E\left[\frac 1{(1+\|S_k\|^3)(1+\|S_k-S_j\|)(1+\|S_j\|)}\right] \\
& \qquad + \sum_{j=1}^{\epsilon_k} \E\left[\frac 1{(1+\|S_k\|^4)(1+\|S_j\|)}\right]\\
& \lesssim \frac{1}{k^2} \left(\sqrt{\epsilon_k} + \sum_{j=-\epsilon_k}^{\epsilon_k}  \E\left[\frac 1{1+\|S_j\|}\right] \right)\lesssim \frac{\sqrt{\epsilon_k}}{k^2},
\end{align*}
using for the third inequality that by \eqref{pn.largex}, it holds uniformly in $x\in \Z^5$ and $j\le \epsilon_k$, 
$$\E\left[\frac 1{1+\|S_k-S_j+x\|^4}\right] \lesssim  k^{-2}, \quad    \E\left[\frac 1{(1+\|S_k\|^3) (1+\|S_k+x\|)}\right] \lesssim k^{-2}.$$
Now we are left with computing $\Delta_{1.3}$. This step is essentially the same as above, so we omit to give all the details. We first define for $y_1,y_2,y_3 \in \Z^5$,  
$$H(y_1,y_2,y_3):= \E\left [Z_0 \1\{\RR^1[1,\epsilon_k]\cap \RR[-\epsilon_k,\epsilon_k] =\emptyset,  S_{\epsilon_k} =y_1, S^1_{\epsilon_k} = y_2, S_{-\epsilon_k} = y_3\}\right],$$
and note that 
$$\Delta_{1,3}= \sum_{\substack{y_1,y_2,y_3\in \Z^5 \\ z_1,z_2,z_3\in \Z^5 \\ x\in \Z^5}} H(y_1,y_2,y_3) H(z_1,z_2,z_3) p_{k-2\epsilon_k}(x-y_1+z_3) \varphi_{x+z_1-y_2} \varphi_{x + z_2-y_3}.$$
Observe here that by Theorem C, $\phi_{x+z_1-y_2}$ is equivalent to $\phi_x$, when $\|z_1\|$ and $\|y_2\|$ are small when compared to $\|x\|$, and similarly for $\phi_{x+z_2-y_3}$. Thus using similar arguments as above, and in particular that by \eqref{pn.largex} and \eqref{lem.hit.2},  
\begin{equation}\label{sum.pkxphix.2}
\sum_{x\in \Z^5} p_k(x) \phi_x^2 = \OO\left(\frac 1k\right), 
\end{equation}
we obtain
$$\Delta_{1,3} = \rho^2 \sum_{x\in \Z^5} p_k(x) \phi_x^2 + o\left(\frac 1k\right).$$
Putting all pieces together gives \eqref{var.1.1}. 
Using in addition \eqref{pn.largex}, \eqref{lem.hit.2} and Theorem \ref{LCLT}, we deduce that 
$$\cov(Z_0\phi_3, Z_k \psi_1) = \rho^2  \left\{\sum_{x\in \Z^5}  \overline p_k(x) \varphi_x^2 - \left(\sum_{x\in \Z^5} \overline p_k(x) \varphi_x\right)^2\right\} + o\left(\frac 1k\right).$$
Then Theorem C, together with \eqref{sum.pkxphix} and \eqref{sum.pkxphix.2} show that 
$$\cov(Z_0\phi_3, Z_k \psi_1) = \sigma  \left\{\sum_{x\in \Z^5}  \frac{\overline p_k(x)}{1+\mathcal J(x)^2}  - \left(\sum_{x\in \Z^5} \frac{\overline p_k(x)}{1+\mathcal J(x)} \right)^2\right\} +o\left(\frac 1k\right), $$
for some constant $\sigma>0$. Finally an approximation of the series with an integral and a change of variables gives, with $c_0:=(2\pi)^{-5/2} (\det \Gamma)^{-1/2}$,  
\begin{align*} 
 \cov(Z_0\phi_3, Z_k \psi_1) & = \frac{\sigma c_0}{k}   \left\{\int_{\R^5} \frac{e^{- 5 \mathcal J(x)^2/2}}{\mathcal J(x)^2} \, dx  - c_0\left(\int_{\R^5}  \frac{e^{- 5 \mathcal J(x)^2/2}}{\mathcal J(x)} \, dx \right)^2\right\} +o\left(\frac 1k\right). 
\end{align*} 
The last step of the proof is to observe that the difference between the two terms in the curly bracket is well a positive real. This follows simply by Cauchy-Schwarz, once we observe that $c_0\int_{\R^5} e^{-5 \mathcal J(x)^2/2} \, dx = 1$, which itself 
can be deduced for instance from the fact that $1= \sum_{x\in \Z^5} p_k(x) \sim c_0\int_{\R^5} 
e^{- 5 \mathcal J(x)^2/2} \, dx$, by the above arguments. This concludes the proof of Lemma \ref{lem.var.1}. \hfill $\square$

%%%%%%%%%%%%%%%%%%%%%%%%%%%%%%%%%%%%%%%%%%%%%%%%%%%%%%%%%%%%%%%%%%%%%%%%%%%%%%%%%%%%

\subsection{Proof of Lemma \ref{lem.var.2}}
Let us concentrate on the term $\cov(Z_0\phi_3,Z_k\psi_3)$, the estimate of $\cov(Z_0\phi_1,Z_k\psi_1)$ being entirely similar. We also assume to simplify notation that 
the walk is aperiodic.

We consider as in the proof of the previous lemma $(S_n^1)_{n\ge 0}$ and $(S_n^2)_{n\ge 0}$ two independent random walks starting from the origin, independent of $(S_n)_{n\in \Z}$, and define this time  
$$\tau_1:= \inf\{n\ge k+\epsilon_k : S_n \in \RR^1[\epsilon_k,\infty) \}, \ \tau_2:= \inf \{n\ge k+\epsilon_k  :  S_n \in S_k+\RR^2[\sqrt \epsilon_k,\infty)\}.$$
Define as well 
$$\overline \tau_1:= \inf\{n\ge k+\epsilon_k : S_n \in \RR^1_\infty \}, \  \overline \tau_2:= \inf \{n\ge k+\epsilon_k  :  S_n \in S_k+\RR^2_\infty\}.$$
\underline{Step $1$.} Our first task is to show that 
\begin{equation}\label{cov.33.first}
\cov(Z_0\phi_3,Z_k\psi_3) = \rho^2 \cdot \cov\left(\1\{\overline \tau_1<\infty\} ,\, \1\{\overline \tau_2<\infty\}\right) + o\left(\frac 1k\right), 
\end{equation}
with $\rho$ as defined in \eqref{def.rho}. This step is essentially the same as in the proof of Lemma \ref{lem.var.1}, but with some additional technical difficulties, so let us give some details. 
First, the proof of Lemma \ref{lem.var.1} shows that (using the same notation), 
$$\E[Z_0\phi_3] = \Delta_{0,3} + \OO\left(\frac{1}{\sqrt{k\epsilon_k}} + \frac{\epsilon_k}{k^{3/2}}\right),$$
and that for any sequence $(\chi_k)_{k\ge 1}$ going to infinity with $\epsilon_k \chi_k^{2+\frac 14} \le k$,  
$$\Delta_{0,3} = \sum_{k/\chi_k \le \|x\|^2 \le k \chi_k } \sum_{\substack{\|y_1\|^2\le \epsilon_k \chi_k \\  \|y_2\|^2\le \epsilon_k \chi_k}} H(y_1,y_2) p_k(x+y_2-y_1) \varphi_x +  \OO\left(\frac 1{\sqrt{k}\cdot \chi_k^{5/4}}  \right).$$
Observe moreover, that by symmetry $H(y_1,y_2) = H(-y_1,-y_2)$, and that by Theorem \ref{LCLT}, for any $x$, $y_1$, and $y_2$ as above, for some constant $c>0$, 
$$\left|p_k(x+y_2-y_1) + p_k(x+y_1-y_2) - p_k(x) \right| = \OO\left(\frac{\epsilon_k\chi_k}{k} \overline p_k(cx) + \frac 1{k^{7/2}}\right),$$
It follows that one can improve the bound \eqref{Z0phi3.final} into 
\begin{align}\label{Z0phi3.bis}
\nonumber \E[Z_0\phi_3] & = \rho \sum_{x\in \Z^5} p_k(x) \phi_x + \OO\left(\frac{\epsilon_k\chi_k}{k^{3/2}} + \frac{\chi_k^2}{k^{3/2}} +\frac 1{\sqrt{k}\cdot \chi_k^{5/4}} + \frac{1}{\sqrt{k\epsilon_k}} + \frac{\epsilon_k}{k^{3/2}} \right)\\
& = \rho \, \bP[\overline\tau_1<\infty] + \OO\left(\frac{\epsilon_k\chi_k}{k^{3/2}} + \frac{\chi_k^2}{k^{3/2}} +\frac 1{\sqrt{k}\cdot \chi_k^{5/4}} + \frac{1}{\sqrt{k\epsilon_k}} + \frac{\epsilon_k}{k^{3/2}} \right). 
\end{align} 
Since by \eqref{lem.hit.3} one has 
$$\E[Z_k\psi_3] \le \E[\psi_3] = \OO\left(\frac 1{\sqrt{\epsilon_k}}\right),$$ 
this yields by taking $\chi_k^{2+1/4} := k/\epsilon_k$, and $\epsilon_k\ge k^{2/3}$ (but still $\epsilon_k= o(k)$), 
\begin{align}\label{Z03.1}
\E[Z_0\phi_3] \cdot  \E[Z_k\psi_3] = \rho\, \bP[\overline\tau_1<\infty]\cdot \E[Z_k\psi_3] + o\left(\frac 1k \right).
\end{align}
We next seek an analogous estimate for $\E[Z_k\psi_3]$. 
Define $Z'_k:=1\{S_{k+i}\neq S_k,\, \forall i=1,\dots,\epsilon_k^{3/4}\}$, and 
$$\Delta_0:= \E\left[Z'_k\cdot \1\left\{\RR[k-\epsilon_k,k+\epsilon_k^{3/4}] \cap  (S_k+\RR^2[1,\sqrt{\epsilon_k}])=\emptyset, \, \tau_2<\infty \right\}\right].$$
Note that (with $\RR$ and $\tilde \RR$ two independent walks), 
\begin{align*}
 \left| \E[Z_k\psi_3] - \Delta_0\right| & \le \bP\left[0\in \RR[\epsilon_k^{3/4},\epsilon_k]\right] +  \bP\left[\tilde \RR[0,\sqrt {\epsilon_k}]\cap \RR[\epsilon_k,\infty)\neq \emptyset\right]  \\
& + \bP\left[\tilde \RR_\infty\cap \RR[\epsilon_k^{3/4},\epsilon_k] \neq \emptyset, \tilde \RR_\infty\cap \RR[\epsilon_k,\infty) \neq \emptyset \right] \\
  &  + \bP\left[\tilde \RR[\sqrt{\epsilon_k},\infty) \cap \RR[-\epsilon_k,\epsilon_k]\neq \emptyset,  \tilde \RR[\sqrt{\epsilon_k},\infty) \cap \RR[\epsilon_k,\infty)\neq \emptyset \right].
\end{align*}
Moreover, 
\begin{equation}\label{ZkZ'k}
\bP\left[0\in \RR[\epsilon_k^{3/4},\epsilon_k]\right]  \stackrel{\eqref{Green.hit}, \eqref{exp.Green}}{\lesssim} \epsilon_k^{-9/8},\quad  \bP\left[\tilde \RR[0,\sqrt {\epsilon_k}]\cap \RR[\epsilon_k,\infty)\neq \emptyset\right]  \stackrel{\eqref{lem.hit.3}}{\lesssim} \epsilon_k^{- 1}.
\end{equation}
Using also the same computation as in the proof of Lemma \ref{lem.123}, we get 
\begin{equation*}
\bP\left[\tilde \RR_\infty\cap \RR[\epsilon_k^{3/4},\epsilon_k] \neq \emptyset,\,  \tilde \RR_\infty\cap \RR[\epsilon_k,\infty) \neq \emptyset \right] \lesssim  \epsilon_k^{-\frac 38 - \frac 12}, 
\end{equation*}
\begin{equation} \label{tau01tau2}
  \bP\left[\tilde \RR[\sqrt{\epsilon_k},\infty) \cap \RR[-\epsilon_k,\epsilon_k]\neq \emptyset,  \tilde \RR[\sqrt{\epsilon_k},\infty) \cap \RR[\epsilon_k,\infty)\neq \emptyset \right] \lesssim \epsilon_k^{-\frac 14 - \frac 12}.
\end{equation}  
As a consequence 
\begin{align}\label{Zk3.1}
\E[Z_k\psi_3] = \Delta_0 + \OO\left(\epsilon_k^{-3/4}\right).
\end{align}
Introduce now
\begin{align*}
\tilde H(y_1,y_2) := \E\left[Z'_k\cdot  \right. & \1\{\RR[k-\epsilon_k,k+\epsilon_k^{3/4}]\cap (S_k+ \RR^2[1,\sqrt{\epsilon_k}])=\emptyset\}  \\
& \times \left. \1\{S_{k+\epsilon_k^{3/4}}-S_k = y_1,  S^2_{\sqrt{\epsilon_k}} = y_2\}\right],
\end{align*}
and note that  
$$\Delta_0 = \sum_{x\in \Z^d} \sum_{y_1,y_2\in \Z^d} \tilde H(y_1,y_2) p_{\epsilon_k-\epsilon_k^{3/4}}(x+y_2-y_1) \phi_x.$$
Let $\chi_k:= \epsilon_k^{1/8}$. As above, we can see that 
\begin{align*}
\Delta_0 & = \sum_{\substack{\epsilon_k/\chi_k\le \|x\|^2\le \epsilon_k\chi_k \\  \|y_1\|^2\le \epsilon_k^{3/4}\chi_k \\  \|y_2\|^2 \le \sqrt{\epsilon_k}\chi_k}} \tilde H(y_1,y_2) p_{\epsilon_k-\epsilon_k^{3/4}}(x+y_2-y_1) \phi_x + \OO\left(\frac{1}{\sqrt{\epsilon_k} \chi_k^{5/4}}\right)\\
 &= \left(\sum_{y_1,y_2\in \Z^d}  \tilde H(y_1,y_2)\right) \left(\sum_{x\in \Z^d} p_{\epsilon_k}(x)\phi_x\right) + \OO\left( \frac{\chi_k}{\epsilon_k^{3/4}} + \frac{\chi_k^2}{\epsilon_k^{3/2}} + \frac{1}{\sqrt{\epsilon_k} \chi_k^{5/4}}\right) \\
 & = \rho\cdot \bP[\overline \tau_2<\infty] + \OO(\epsilon_k^{-5/8}). 
\end{align*}
Then by taking $\epsilon_k\ge k^{5/6}$, and recalling \eqref{Z03.1} and \eqref{Zk3.1}, we obtain 
\begin{align}\label{Z03.2}
\E[Z_0\phi_3] \cdot  \E[Z_k\psi_3] = \rho^2\cdot \bP[\overline\tau_1<\infty]\cdot \bP[\overline \tau_2<\infty]+ o\left(\frac 1k \right).
\end{align}
Finally, let 
\begin{align*}
\Delta_{3,3}:= & \E [Z_0Z'_k  \1\{\RR^1[1,\epsilon_k]\cap \RR[-\epsilon_k,\epsilon_k] =\emptyset\}\\
& \times \1\{  (S_k+\RR^2[1,\sqrt{\epsilon_k}])\cap \RR[k-\epsilon_k^{\frac 34},k+\epsilon_k^{\frac 34}] =\emptyset,  \tau_1 <\infty, \tau_2<\infty \}]. 
\end{align*} 
It amounts to estimate the difference between $\Delta_{3,3}$ and $\E[Z_0Z_k\phi_3\psi_3]$. 
Define
$$\tilde \tau_1:=\inf\{n\ge k+\epsilon_k : S_n\in \RR^1[0,\epsilon_k]\},  \   \tilde \tau_2:=\inf\{n\ge k+\epsilon_k : S_n\in S_k+\RR^2[0,\sqrt{\epsilon_k}]\}.$$
Observe first that  
\begin{align}\label{tilde1bar2}
\nonumber & \bP[\tilde \tau_1\le \overline \tau_2<\infty] \stackrel{\eqref{lem.hit.2}}{\lesssim}   \E\left[\frac{\1\{\tilde \tau_1<\infty\}}{1+\|S_{\tilde \tau_1}-S_k\|} \right]
 \stackrel{\eqref{lem.hit.1}}{\lesssim}  \sum_{i=0}^{\epsilon_k} \E\left[\frac{G(S_i^1-S_{k+\epsilon_k})}{1+\|S_i^1-S_k\|} \right]\\
\nonumber & \lesssim  \sum_{i=0}^{\epsilon_k}  \sum_{z\in \Z^5} p_i(z) \E\left[\frac{G(z-S_{k+\epsilon_k})}{1+\|z-S_k\|} \right]\stackrel{\eqref{pn.largex}}{\lesssim}
\sum_{z\in \Z^5} \frac {\sqrt{\epsilon_k}}{1+\|z\|^4}\, \E\left[\frac{G(z-S_{k+\epsilon_k})}{1+\|z-S_k\|} \right]   \\
&\stackrel{\eqref{Green}}{\lesssim}  \E\left[\frac{\sqrt{\epsilon_k}}{(1+\|S_{k+\epsilon_k}\|^2)(1+\|S_k\|)} \right] 
\stackrel{\eqref{exp.Green.x}}{\lesssim}  \E\left[\frac{\sqrt{\epsilon_k}}{1+\|S_k\|^3} \right] \stackrel{\eqref{exp.Green}}{\lesssim} \frac{\sqrt{\epsilon_k}}{k^{3/2}}, 
\end{align}
and likewise, 
\begin{align*}
& \bP[\overline \tau_1\le \tilde \tau_2<\infty] \stackrel{\eqref{lem.hit.1}}{\le} \sum_{j\ge 0}\sum_{i=0}^{\epsilon_k}\E\left[G(S_k + S_i^2 - S_j^1) G(S_j^1- S_{k+\epsilon_k})\right] \\
& = \sum_{i=0}^{\epsilon_k} \sum_{z\in \Z^5}\E\left[G(z) G(S_k + S_i^2 - z) G(z- S_{k+\epsilon_k})\right]  \\
& \le C\sum_{i=0}^{\epsilon_k} \E\left[\frac 1{1+\|S_k + S_i^2\|^3}\left(\frac 1{1+\|S_{k+\epsilon_k}\|} +\frac 1{1+ \|S_{k+\epsilon_k}-S_k-S_i^2\|}\right)\right]\\
& \stackrel{\eqref{exp.Green},\, \eqref{exp.Green.x}}{\le} C\E\left[\frac {\epsilon_k}{1+\|S_k\|^4}\right] +C \E\left[\frac {\sqrt{\epsilon_k}}{1+\|S_k\|^3} \right]= \OO\left(\frac{\sqrt{\epsilon_k}}{k^{3/2}}\right).
\end{align*}
Additionally, it follows directly from \eqref{lem.hit.3} that 
$$\bP[\overline \tau_2 \le \tilde \tau_1<\infty] \lesssim \frac{\sqrt{\epsilon_k}}{k^{3/2}}, \quad \text{and} \quad \bP[\tilde \tau_2 \le \overline \tau_1<\infty] \lesssim \frac{1}{\epsilon_k \sqrt k},$$ 
which altogether yields
$$|\bP[\overline \tau_1<\infty,\, \overline \tau_2<\infty] - \bP[\tau_1<\infty,\,  \tau_2<\infty]| \lesssim \frac{\sqrt{\epsilon_k}}{k^{3/2}} + \frac{1}{\epsilon_k \sqrt k}.$$
Similar computations give also    
\begin{equation}\label{tau1tau2}
\bP[\overline \tau_1<\infty, \, \overline \tau_2<\infty] \lesssim \frac 1{\sqrt{k \epsilon_k}}.
\end{equation}
Next, using \eqref{ZkZ'k} and the Markov property, we get 
$$\E[|Z_k-Z'_k|\1\{\tau_1<\infty\}] \lesssim \frac 1{\epsilon_k^{9/8}\sqrt k}.$$
Thus, for $\epsilon_k \ge k^{5/6}$, 
\begin{align*}
 \left|\E[Z_0Z_k\phi_3\psi_3] - \Delta_{3,3}\right| & \le \bP[\tau_{0,1}<\infty, \tau_1<\infty, \tau_2<\infty] +\bP[\tau_{0,2}<\infty, \tau_1<\infty, \tau_2<\infty] \\
& \qquad + \bP[\tilde \tau_{0,2}<\infty, \tau_1<\infty, \tau_2<\infty] + o\left(\frac 1k\right),
\end{align*}
where $\tau_{0,1}$ is as defined in the proof of Lemma \ref{lem.var.1}, 
$$\tau_{0,2} : =\inf\{n\ge \sqrt{\epsilon_k} : S_k + S_n^2 \in \RR[k-\epsilon_k,k+\epsilon_k]\},$$
and  
$$\tilde \tau_{0,2} : =\inf\{n\le \sqrt{\epsilon_k}  : S_k + S_n^2 \in \RR[k-\epsilon_k,k-\epsilon_k^{3/4}]\cup \RR[k+\epsilon_k^{3/4},k+\epsilon_k]\}.$$ 
Applying \eqref{lem.hit.3} twice already shows that 
$$\bP[\tilde \tau_{0,2}<\infty,\, \tau_1<\infty] \lesssim \frac{1}{\sqrt k} \cdot \bP[\tilde \tau_{0,2}<\infty] \lesssim \frac{1}{\sqrt{k}\epsilon_k^{5/8}} = o\left(\frac 1k\right). $$
Then, notice that \eqref{tilde1bar2} entails
$$\bP[\RR[k+\epsilon_k,\infty) \cap \RR^1[0,\tau_{0,1}]\neq \emptyset, \, S^1_{\tau_{0,1}} \in \RR[-\epsilon_k,0]] \lesssim \frac{\sqrt{\epsilon_k}}{k^{3/2}}.$$
On the other hand,  
\begin{align*}
&\bP[\RR[k+\epsilon_k,\infty) \cap \RR^1[0,\tau_{0,1}]\neq \emptyset,  S^1_{\tau_{0,1}} \in \RR[0,\epsilon_k]] \\
 \stackrel{\eqref{lem.hit.1}}{\le} & \sum_{i=0}^{\epsilon_k} \sum_{j=k+\epsilon_k}^\infty 
\E[G(S_i-S_{k+j})G(S_{k+j} - S_k)]   =  \sum_{i=0}^{\epsilon_k} \sum_{z\in \Z^5} 
\E[G(S_i-S_k + z)G(z)G_{\epsilon_k}(z)] \\
\stackrel{\eqref{exp.Green}}{\lesssim} & \frac{\epsilon_k}{k^{3/2}} \sum_{z\in \Z^5} G(z)G_{\epsilon_k}(z)\stackrel{\text{Lemma }\ref{lem.upconvolG}}{\lesssim} 
\frac{\sqrt{\epsilon_k}}{k^{3/2}}. 
\end{align*}
By \eqref{lem.hit.1} and \eqref{exp.Green}, one has with $\tilde \RR_\infty$ an independent copy of $\RR_\infty$, 
\begin{align*}
& \bP[\tau_{0,1}<\infty,  \tau_2<\infty, \RR[k+\epsilon_k,\infty) \cap \RR^1[\tau_{0,1},\infty)\neq \emptyset ] \\
\lesssim & \frac{1}{\sqrt{\epsilon_k}} 
\max_{-\epsilon_k\le i\le \epsilon_k}\bP[\tau_2<\infty,\, \RR[k+\epsilon_k,\infty) \cap (S_i+ \tilde \RR_\infty) \neq \emptyset ] \lesssim \frac{1}{\epsilon_k\sqrt{k}}, 
\end{align*}
where the last equality follows from \eqref{tau1tau2}. Thus 
$$\bP[\tau_{0,1}<\infty, \tau_1<\infty, \tau_2<\infty]  =o\left(\frac 1k\right).$$ 
In a similar fashion, one has 
$$\bP[\tau_{0,2}<\infty, \tau_2\le \tau_1<\infty] \stackrel{\eqref{lem.hit.3}}{\lesssim} \frac{1}{\sqrt k}  \bP[\tau_{0,2}<\infty,  \tau_2<\infty] 
\stackrel{\eqref{tau01tau2}}{\lesssim} \frac{1}{\epsilon_k^{3/4}\sqrt{k}},$$
as well as,  
\begin{align*}
& \bP\left[\tau_{0,2}<\infty, \, \tau_1\le \tau_2<\infty,\, S_{\tau_2}\in (S_k+\RR^2[0, \tau_{0,2}])\right] \\
 \stackrel{\eqref{lem.hit.1}}{\le} & \sum_{i=k-\epsilon_k}^{k+\epsilon_k} \sum_{j\ge 0} \sum_{\ell \ge 0} \E[G(S_i-\tilde S_j - S^1_\ell) G(\tilde S_j + S^1_\ell-S_k) G(S^1_\ell - S_{k+\epsilon_k})] \\ 
\le & \sum_{i=k-\epsilon_k}^{k+\epsilon_k} \sum_{\ell \ge 0} \sum_{z\in\Z^5} \E[G(z) G(S_i- S^1_\ell - z) G(z + S^1_\ell-S_k) G(S^1_\ell - S_{k+\epsilon_k})]\\
 \stackrel{\text{Lemma }\ref{lem.upconvolG}}{\lesssim} & \sum_{i=k-\epsilon_k}^{k+\epsilon_k} \sum_{\ell \ge 0}  
\E\left[\frac{G(S^1_\ell - S_{k+\epsilon_k})}{1+\|S^1_\ell - S_k\|^3} \left(\frac 1{1+\|S_\ell^1- S_i\|} + \frac 1{1+\|S_i-S_k\|}\right)\right]\\
 \stackrel{\eqref{exp.Green}, \eqref{exp.Green.x}}{\lesssim} &  \sum_{i=0}^{\epsilon_k} \sum_{\ell \ge 0} 
\left\{\E\left[\frac{\epsilon_k^{-3/2}}{1+\|S^1_\ell - S_k\|^3} \left(\frac 1{1+\|S_\ell^1- S_{k-i}\|} + \frac 1{1+\|S_{k-i}-S_k\|}\right)\right] \right. \\
+&  \left.  \E\left[\frac{1}{(1+\|S^1_\ell - S_{k+i}\|^3)(1+\|S^1_\ell - S_k\|^3)} \left(\frac 1{1+\|S_\ell^1- S_{k+i}\|} + \frac 1{1+\|S_{k+i}-S_k\|}\right)\right] \right\}\\
 \stackrel{\eqref{pn.largex},  \eqref{exp.Green.x}}{\lesssim} & \sum_{i=0}^{\epsilon_k} \sum_{\ell \ge 0} 
\left\{\E\left[\frac{\epsilon_k^{-3/2}}{1+\|S^1_\ell - S_{k-i}\|^3} \left(\frac 1{1+\|S_\ell^1- S_{k-i}\|} + \frac 1{1+\sqrt{i}}\right)\right] \right.  \\ 
& \qquad + \left.  \E\left[\frac{(1+i)^{-1/2}}{1+\|S^1_\ell - S_k\|^6} \right] \right\} \\
   \lesssim &\  \frac {\sqrt{\epsilon_k}}{k^{3/2}} , 
\end{align*}
and 
\begin{align*}
& \bP\left[\tau_{0,2}<\infty, \, \tau_1\le \tau_2<\infty,\, S_{\tau_2}\in (S_k+\RR^2[\tau_{0,2},\infty))\right] \\
 \stackrel{\eqref{Green.hit}}{\le}  &
\sum_{i=-\epsilon_k}^{\epsilon_k} \E\left[G(S_{k+i}-S_k - S^2_{\sqrt{\epsilon_k}}) \1\{\tau_1<\infty,\, \RR[\tau_1,\infty) \cap (S_{k+i}+ \tilde \RR_\infty) \neq \emptyset\}\right]\\\
 \stackrel{\eqref{lem.hit.2}}{\lesssim} & \sum_{i=-\epsilon_k}^{\epsilon_k} \E\left[G(S_{k+i}-S_k - S^2_{\sqrt{\epsilon_k}})
\frac{\1\{\tau_1<\infty\}}{1+ \|S_{\tau_1} - S_{k+i}\|} \right]  \\
 \stackrel{\eqref{lem.hit.1}}{\lesssim} & \sum_{i=-\epsilon_k}^{\epsilon_k}\sum_{j\ge k+\epsilon_k} \E\left[\frac{G(S_{k+i}-S_k - S^2_{\sqrt{\epsilon_k}}) G(S_j)}{1+ \|S_j - S_{k+i}\|} \right] \\
\lesssim & \sum_{i=0}^{\epsilon_k} \sum_{z\in \Z^5} \left\{\E\left[\frac{G(S_{k-i}-S_k - S^2_{\sqrt{\epsilon_k}}) G(S_k+z)G(z)}{1+ \|z + S_k - S_{k-i}\|} \right] \right. \\
& \qquad \left. + \E\left[\frac{G(S_{k+i}-S_k - S^2_{\sqrt{\epsilon_k}}) G(S_{k+i}+z)G(z)}{1+ \|z\|} \right] \right\}\\
\lesssim & \sum_{i=0}^{\epsilon_k}  \left\{\E\left[\frac{G(S_{k-i}-S_k - S^2_{\sqrt{\epsilon_k}}) }{1+ \|S_k\|^2} \right] 
+ \E\left[\frac{G(S_{k+i}-S_k - S^2_{\sqrt{\epsilon_k}}) }{1+ \|S_{k+i}\|^2} \right] \right\}\\
\stackrel{\eqref{exp.Green}, \eqref{exp.Green.x}}{\lesssim} & \frac{1}{\epsilon_k^{3/4}}\sum_{i=0}^{\sqrt{\epsilon_k}} \E\left[\frac{1}{1+ \|S_k\|^2} +\frac{1}{1+ \|S_{k+i}\|^2}\right]  + \sum_{i = \sqrt{\epsilon_k}}^{\epsilon_k}  \E\left[\frac{G(S_{k-i}-S_k)}{1+ \|S_k\|^2} +\frac{G(S_{k+i}-S_k)}{1+ \|S_{k+i}\|^2}\right]\\
 \stackrel{\eqref{pn.largex},  \eqref{exp.Green}}{\lesssim} & \frac{1}{\epsilon_k^{1/4}k} + \sum_{i = \sqrt{\epsilon_k}}^{\epsilon_k} \frac 1{i^{3/2}}\cdot \E\left[\frac{1}{1+ \|S_{k-i}\|^2} +\frac{1}{1+ \|S_k\|^2}\right]   \lesssim  \frac 1{\epsilon_k^{1/4} k}. 
\end{align*}
Thus at this point we have shown that 
\begin{eqnarray}\label{approx.Delta3}
 \left|\E[Z_0Z_k\phi_3\psi_3] - \Delta_{3,3}\right|  =  o\left(\frac 1k\right). 
 \end{eqnarray}
Now define 
\begin{align*}
\tilde H(z_1,z_2,z_3) := \bP\left[0\notin \RR[1,\epsilon_k^{3/4}], \right. & \tilde \RR[1,\sqrt{\epsilon_k}]\cap \RR[-\epsilon_k^{ 3/4},\epsilon_k^{ 3/4}]=\emptyset,  \\
 & \left. S_{\epsilon_k^{3/4}}=z_1, S_{-\epsilon_k^{3/4}}=z_3, \tilde S_{\sqrt{\epsilon_k}} = z_3\right],
 \end{align*}
and recall also the definition of $H(y_1,y_2)$ given in \eqref{Hy1y2}. 
One has 
$$\Delta_{3,3} = \sum H(y_1,y_2) \tilde H(z_1,z_2,z_3) 
p_{k-\epsilon_k - \epsilon_k^{3/4}}(x - y_1+y_2+z_3-z_2) p_{\epsilon_k- \epsilon_k^{3/4}}(u-z_1+z_2) \varphi_{x,u},$$
where the sum runs over all $x,u,y_1,y_2,z_1,z_2,z_3\in \Z^5$, and 
\begin{align*}
\phi_{x,u}  := \bP[\overline \tau_1<\infty,\, \overline \tau_2<\infty\mid S_k = x,\, S_{k+\epsilon_k} = x+u]. 
\end{align*} 
Note that the same argument as for \eqref{tau1tau2} gives also 
\begin{eqnarray}\label{phixu}
\phi_{x,u} \lesssim \frac 1{1+\|u\|}\left(\frac 1{1+\|x+u\|} +\frac 1{1+\|x\|}\right). 
\end{eqnarray} 
Using this it is possible to see that in the expression of $\Delta_{3,3}$ given just above, one can restrict the sum to typical values of the parameters. 
Indeed, consider for instance the sum on atypically large values of $x$. More precisely, take $\chi_k$, such that $\epsilon_k \chi_k^{2+1/4} =k$, and note that by \eqref{phixu}, 
\begin{align*}
& \sum_{\substack{\|x\|^2\ge k\chi_k \\ u,y_1,y_2,z_1,z_2,z_3}} H(y_1,y_2) \tilde H(z_1,z_2,z_3) 
p_{k-\epsilon_k - \epsilon_k^{3/4}}(x - y_1+y_2+z_3-z_2) p_{\epsilon_k- \epsilon_k^{3/4}}(u-z_1+z_2) \varphi_{x,u}\\
& \le  \bP\left[ \|S_k-S_{\epsilon_k}^1\|\ge \sqrt{k\chi_k}, \tau_1<\infty, \tau_2<\infty\right]\le  \bP\left[ \|S_k-S_{\epsilon_k}^1\|\ge \sqrt{k\chi_k}, \tau_1<\infty, \overline \tau_2<\infty\right] \\
& \lesssim  \E\left[\frac{\1\{\|S_k-S_{\epsilon_k}^1\|\ge \sqrt{k\chi_k}\}}{1+\|S_{k+\epsilon_k}-S_k\|} \left(\frac {1}{1+\|S_k-S^1_{\epsilon_k}\|} + \frac 1{1+\|S_{k+\epsilon_k}-S^1_{\epsilon_k}\|}\right)\right] \\
& \lesssim \frac{1}{\chi_k^{5/4}\sqrt{k\epsilon_k} }, 
\end{align*}
where the last equality follows by applying Cauchy-Schwarz inequality and \eqref{Sn.large}. The other cases are entirely similar. Thus  
$\Delta_{3,3}$ is well approximated by the sums on typical values of the parameters (similarly as for $\Delta_0$ for instance), and then we can deduce with Theorem \ref{LCLT} and \eqref{phixu} that
$$\Delta_{3,3}  = \rho^2\cdot \bP[\overline \tau_1<\infty,\, \overline \tau_2<\infty]+  o\left(\frac 1k\right).$$
Together with \eqref{approx.Delta3} and \eqref{Z03.2} this proves \eqref{cov.33.first}.

\underline{Step $2$.} 
For a (possibly random) time $T$, set 
$$\overline \tau_1\circ T := \inf \{n\ge T\vee \epsilon_k : S_n \in \RR^1_\infty\},\ \overline \tau_2\circ T := \inf \{n\ge T\vee \epsilon_k : S_n \in (S_k+\RR^2_\infty)\}. $$
Observe that 
\begin{equation}\label{main.tau.1}
\bP[\overline \tau_1\le \overline \tau_2<\infty] = \bP[\overline \tau_1\le \overline \tau_2\circ \overline \tau_1<\infty] - \bP[\overline \tau_2\le \overline \tau_1\circ \overline \tau_2\le \overline \tau_2\circ \overline \tau_1\circ \overline \tau_2<\infty],
\end{equation}
and symmetrically, 
\begin{equation}\label{main.tau.2}
\bP[\overline \tau_2\le \overline \tau_1<\infty] = \bP[\overline \tau_2\le \overline \tau_1\circ \overline \tau_2<\infty] - \bP[\overline \tau_1\le\overline \tau_2\circ \overline \tau_1\le \overline \tau_1\circ \overline \tau_2\circ \overline \tau_1<\infty].
\end{equation}
Our aim here is to show that the two error terms appearing in \eqref{main.tau.1} and \eqref{main.tau.2} are negligible. 
Applying repeatedly \eqref{lem.hit.1} gives 
\begin{align*}
E_1& := \bP[\overline \tau_1\le\overline \tau_2\circ \overline \tau_1\le \overline \tau_1\circ \overline \tau_2\circ \overline \tau_1<\infty] \\
&\lesssim   \sum_{j\ge 0} \sum_{\ell \ge 0} \sum_{m\ge 0} \E\left[ G(S_j^1 - S_k - S_\ell^2) G(S_k + S_\ell^2 - S_m^1) G(S_m^1 - S_{k+\epsilon_k})\right]\\
& \stackrel{\eqref{exp.Green.x}}{\lesssim}  \sum_{j\ge 0} \sum_{\ell \ge 0} \sum_{m\ge 0} \E\left[ G(S_j^1 - S_k - S_\ell^2) G(S_k + S_\ell^2 - S_m^1) G(S_m^1 - S_k)\right]\\
& \lesssim \sum_{j\ge 0} \sum_{m\ge 0} G(z) \E\left[ G(S_j^1 - S_k - z) G(S_k + z - S_m^1) G(S_m^1 - S_k)\right].
\end{align*}
Note also that by using Lemma \ref{lem.upconvolG} and \eqref{Green}, we get  
$$\sum_{z\in \Z^5}G(z-x) G(z-y) G(z) \lesssim \frac 1{1+\|x\|^3} \left(\frac 1{1+\|y\|} + \frac {1}{1+ \|y-x\|} \right).$$
Thus, distinguishing also the two cases $j\le m$ and $m\le j$,  we obtain 
\begin{align*}
E_1& \lesssim  \sum_{j\ge 0} \sum_{m\ge 0}\E\left[\frac {G(S_m^1 - S_k)}{1+\|S_j^1-S_k\|^3} \left( \frac 1{1+\|S_m^1-S_k\|} + \frac 1{1+\|S_m^1- S_j^1\|}\right) \right] \\
& \lesssim \sum_{j\ge 0} \sum_{z\in \Z^5} G(z)\left\{ \E\left[\frac {G(z+S_j^1 - S_k)}{1+\|S_j^1-S_k\|^3} \left( \frac 1{1+\|z+S_j^1-S_k\|} + \frac 1{1+\|z\|}\right) \right] \right. \\
& \qquad \left. + \E\left[\frac {G(S_j^1 - S_k)}{1+\|z+S_j^1-S_k\|^3} \left( \frac 1{1+\|S_j^1-S_k\|} + \frac 1{1+\|z\|}\right) \right]\right\}\\
& \lesssim \sum_{j\ge 0}  \E\left[\frac {1}{1+\|S_j^1-S_k\|^5} \right] \lesssim  \E\left[\frac {\log (1+\|S_k\|)}{1+\|S_k\|^3}\right] \lesssim \frac{\log k}{k^{3/2}}. 
\end{align*}
Similarly, 
\begin{align*}
& \bP[\overline \tau_2\le \overline \tau_1\circ \overline \tau_2\le \overline \tau_2\circ \overline \tau_1\circ \overline \tau_2<\infty] \\
&\lesssim  \sum_{j\ge 0} \sum_{\ell \ge 0} \sum_{m\ge 0} \E\left[ G(S_j^2 + S_k - S_\ell^1) G(S_\ell^1 - S_k- S_m^2) G(S_m^2 + S_k - S_{k+\epsilon_k})\right]\\
& \stackrel{\eqref{exp.Green}, \eqref{exp.Green.x}}{\lesssim} \frac {1}{\sqrt{\epsilon_k}} \sum_{j\ge 0} \sum_{\ell \ge 0} \sum_{m\ge 0}\E\left[ \frac{G(S_j^2 + S_k - S_\ell^1) G(S_\ell^1 - S_k- S_m^2) }{1+\|S_m^2\|^2} \right]\\ 
&\lesssim \frac{1}{\sqrt{\epsilon_k}} \sum_{j\ge 0} \sum_{m\ge 0}\E\left[\frac 1{(1+\|S_m^2\|^2)(1+\|S_j^2+S_k\|^3)} \left( \frac 1{1+\|S_m^2+S_k\|} + \frac 1{1+\|S_m^2- S_j^2\|}\right) \right] \\
&\lesssim \frac{1}{\sqrt{\epsilon_k}} \sum_{j\ge 0}\E\left[\frac 1{(1+\|S_j^2\|)(1+\|S_j^2+S_k\|^3)} +\frac 1{(1+\|S_j^2\|^2)(1+\|S_j^2+S_k\|^2)}  \right]  \\
&\lesssim \frac{1}{\sqrt{\epsilon_k}}\cdot \E\left[\frac {\log (1+\|S_k\|)}{1+\|S_k\|^2}\right]  \lesssim \frac{\log k}{k \sqrt{\epsilon_k}}. 
\end{align*}

\underline{Step $3$.} We now come to the estimate of the two main terms in \eqref{main.tau.1} and \eqref{main.tau.2}. 
In fact it will be convenient to replace $\overline \tau_1$ in the first one by 
$$\hat \tau_1:= \inf \{n\ge k : S_n\in \RR_\infty^1\}.$$ 
The error made by doing this is bounded as follows: by shifting the origin to $S_k$, and using symmetry of the step distribution, we can write
\begin{align*}
& \left| \bP[\overline \tau_1\le \overline \tau_2\circ \overline \tau_1<\infty]  -  \bP[\hat \tau_1\le \overline \tau_2\circ \hat \tau_1<\infty]\right| \le \bP\left[\RR_\infty^1\cap \RR[k,k+\epsilon_k] \neq \emptyset, \overline \tau_2<\infty\right]  \\
& \stackrel{\eqref{Green.hit}}{\le} \E\left[\left(\sum_{i=0}^{\epsilon_k} G(S_i- \tilde S_k)\right) \left(\sum_{j=\epsilon_k}^\infty G(S_j)\right)\right] \\
&  =  \E\left[\left(\sum_{i=0}^{\epsilon_k} G(S_i- \tilde S_k)\right)  \left(\sum_{z\in \Z^5} G(z) G(z+S_{\epsilon_k})\right)\right] \\
&\stackrel{\text{Lemma }\ref{lem.upconvolG}}{\lesssim} \sum_{i=0}^{\epsilon_k} \E\left[ \frac{G(S_i- \tilde S_k)}{1+\|S_{\epsilon_k}\|}\right]  \stackrel{\eqref{exp.Green}}{\lesssim} \frac {\epsilon_k}{k^{3/2}}\cdot  \E\left[ \frac 1{1+\|S_{\epsilon_k}\|}\right] \lesssim  \frac{\sqrt{\epsilon_k}}{k^{3/2}}. 
\end{align*}
Moreover, using Theorem C, the Markov property and symmetry of the step distribution, we get for some constant $c>0$, 
\begin{align*}
& \bP[\hat \tau_1\le \overline \tau_2\circ \hat \tau_1<\infty] 
 = c \E\left[\frac {\1\{\hat \tau_1<\infty\}}{1+\JJ(S_{\hat \tau_1} - S_k) }\right] + o\left(\frac 1k\right) \\
 & = c \E\left[\frac {\1\{\hat \tau_1<\infty\}}{1+\JJ(S_{\hat \tau_1}) }\right] + o\left(\frac 1k\right) 
  = c  \sum_{x\in \Z^5}p_k(x) \, \E_{0,x} \left[F(S_\tau) \1\{\tau<\infty\}\right]  + o\left(\frac 1k\right),
\end{align*}
with $\tau$ the hitting time of two independent walks starting respectively from the origin and from $x$, and  $F(z) := 1/(1+ \JJ(z))$.  
Note that the bound $o(1/k)$ on the error term in the last display comes from the fact that 
$$\E\left[\frac {\1\{\hat \tau_1<\infty\}}{1+\JJ(S_{\hat \tau_1})} \right] \stackrel{\eqref{lem.hit.1}}{\lesssim} \sum_{j\ge 0} \E\left[\frac{G(\tilde S_j-S_k)}{1+\|\tilde S_j\|}\right] \lesssim \sum_{z\in \Z^5} 
\E\left[\frac{G(z)G(z-S_k)}{1+\|  z\|}\right] \lesssim \frac 1k. $$
Then by applying Theorem \ref{thm.asymptotic}, we get 
\begin{equation}\label{main.tau.1.2}
 \bP[\hat \tau_1\le \overline \tau_2\circ \hat \tau_1<\infty] = c_0 \sum_{x\in \Z^5} p_k(x) \sum_{z\in \Z^5} \frac{G(z)G(z-x)}{1+\mathcal J(z)} + o\left(\frac 1k\right),  
\end{equation}
for some constant $c_0>0$. 
Likewise, by Theorem \ref{thm.asymptotic} one has for some constant $\nu\in (0,1)$, 
\begin{align*}
\bP[\overline \tau_2\le \overline \tau_1\circ \overline \tau_2<\infty] & =  c\, \E\left[\frac{\1\{\overline \tau_2<\infty\} }{1+ \JJ(S_{\overline \tau_2})}\right] + 
\OO\left(\E\left[ \frac{\1\{\overline \tau_2<\infty\} }{1+ \JJ(S_{\overline \tau_2})^{1+\nu}} \right] \right). 
\end{align*}
Furthermore, 
\begin{align*}
& \E\left[ \frac{\1\{\overline \tau_2<\infty\} }{1+ \JJ(S_{\overline \tau_2})^{1+\nu}} \right] \lesssim \sum_{j\ge 0}  \E\left[ \frac {G(S_j^2+S_k-S_{k+\epsilon_k}) }{1+ \| S_j^2+S_k\|^{1+\nu}} \right] \\
& \stackrel{\eqref{exp.Green}, \eqref{exp.Green.x}}{\lesssim} \frac{1}{\sqrt{\epsilon_k}}\sum_{j\ge 0}  \E\left[ \frac {1 }{(1+\|S_j^2\|^2)(1+ \| S_j^2+S_k\|^{1+\nu})} \right] \\
 & \lesssim \frac{1}{\sqrt{\epsilon_k}} \E\left[\frac{\log (1+\|S_k\|)}{1+\|S_k\|^{1+\nu}}\right]  \lesssim \frac {\log k}{ k^{(1+ \nu)/2}\sqrt{\epsilon_k}}. 
 \end{align*}
Therefore, taking $\epsilon_k\ge k^{1-\nu/2}$, we get  
\begin{align}\label{main.tau.2.1}
 \nonumber \bP[\overline \tau_2\le \overline \tau_1\circ \overline \tau_2<\infty] & = c\, \E\left[\frac{\1\{\overline \tau_2<\infty\} }{1+ \JJ(S_{\overline \tau_2})}\right] + o\left(\frac 1k\right)  \\
\nonumber & =  c \sum_{u\in \Z^5} p_{\epsilon_k}(u) \E_{0,u} \left[\frac{\1\{\tau<\infty\} }{1+ \JJ(S_\tau - S_k)}\right] +  o\left(\frac 1k\right) \\
 & = c \sum_{u\in \Z^5} p_{\epsilon_k}(u) \E_{0,u} \left[\tilde F(S_\tau) \1\{\tau<\infty\} \right] +  o\left(\frac 1k\right), 
\end{align}
with $\tau$ the hitting time of two independent walks starting respectively from the origin and from $u$, and 
$$\tilde F(z):= \E\left[\frac 1{1+ \JJ(z-S_k)} \right]. $$ 
We claim that this function $\tilde F$ satisfies \eqref{cond.F}, for some constant $C_{\tilde F}$ which is independent of $k$. 
Indeed, first notice that 
$$\tilde F(z) \asymp \frac 1{1+\|z\| + \sqrt{k}},\quad \text{and}\quad \E\left[ \frac 1{1+\JJ(z-S_k)^2}\right] \asymp \frac 1{1+\|z\|^2 + k},$$
which can be seen by using Theorem \ref{LCLT}. Moreover, by triangle inequality, and Cauchy-Schwarz,  
\begin{align*}
|\tilde F(y) - \tilde F(z) | & \lesssim  \E\left[\frac {\|y-z\|}{(1+\|y-S_k\|)(1+\|z-S_k\|)}\right]  \\
& \lesssim \|y-z\| \, \E\left[\frac 1{1+\|y-S_k\|^2}\right]^{\frac 12} \E\left[\frac 1{1+\|z-S_k\|^2}\right]^{\frac 12} \\
& \lesssim \frac{\|y-z\|}{(1+\|y\|+\sqrt k)(1+\|z\| +\sqrt{k})} \lesssim  \frac{\|y-z\|}{1+\|y\|}\cdot \tilde F(z), 
\end{align*}
which is the desired condition \eqref{cond.F}. Therefore, coming back to \eqref{main.tau.2.1} and applying Theorem \ref{thm.asymptotic} once more gives, 
\begin{align}\label{main.tau.2.1.bis}
\nonumber  \bP[\overline \tau_2\le \overline \tau_1\circ \overline \tau_2<\infty]  & = c_0\sum_{u\in \Z^5} p_{\epsilon_k}(u) \sum_{z\in \Z^5} G(z)G(z-u)\tilde F(z) + o\left(\frac 1k\right) \\
 & = c_0\sum_{u\in \Z^5} \sum_{x\in \Z^5}p_{\epsilon_k}(u) p_k(x) \sum_{z\in \Z^5} \frac{G(z)G(z-u)}{1+\mathcal J(z-x)} + o\left(\frac 1k\right). 
\end{align}
Similarly, one has 
\begin{align}\label{main.tau.product} 
\nonumber &\bP[\overline \tau_1<\infty] \cdot \bP[\overline \tau_2<\infty] =\bP[\hat \tau_1<\infty] \cdot \bP[\overline \tau_2<\infty] + \OO\left(\frac{\sqrt{\epsilon_k}}{k^{3/2}}\right) \\
& =  c_0 \sum_{u\in \Z^5} \sum_{x\in \Z^5} p_{\epsilon_k}(u)p_k(x) \sum_{z\in \Z^5} \frac{G(z)G(z-u)}{1+\mathcal J(x)} +  o\left(\frac 1k\right). 
\end{align}
Note in particular that the constant $c_0$ that appears here is the same as in \eqref{main.tau.1.2} and \eqref{main.tau.2.1.bis}.

\underline{Step $4$.} We claim now that when one takes the difference between the two expressions in \eqref{main.tau.2.1.bis} and \eqref{main.tau.product}, 
one can remove the parameter $u$ from the factor $G(z-u)$ (and then absorb the sum over $u$). 
Indeed, note that for any $z$ with $\JJ(z)\le \JJ(x)/2$, one has 
$$\left| \frac 1{1+\JJ(z+x)} + \frac  1{1+\JJ(z-x)} - \frac 2{1+\JJ(x)}\right| \lesssim  \frac{\|z\|^2}{1+\|x\|^3}.$$ 
It follows that, for any $\chi_k\ge 2$, 
\begin{align*}
& \sum_{\substack{u,x\in \Z^5 \\ \JJ(z) \le \frac{\JJ(x)}{\chi_k}}} p_{\epsilon_k}(u) p_k(x)  G(z) G(z-u) \left|\frac  1{1+\JJ(z-x)}+\frac  1{1+\JJ(z+x)} - \frac 2{1+\JJ(x)}\right| \\ 
&\lesssim  \sum_{x\in \Z^5} \frac{p_k(x)}{1+\|x\|^3} \sum_{\JJ(z) \le \JJ(x)/\chi_k } \frac{\E[G(z-S_{\epsilon_k})]}{1+\|z\|}  \stackrel{\eqref{exp.Green.x}}{\lesssim} 
\frac{1}{k\chi_k}.  
\end{align*}
In the same way, for any $z$ with $\JJ(z) \ge 2\JJ(u)$, one has 
$$|G(z-u) -G(z)| \lesssim \frac{\|u\|}{1+\|z\|^4},$$ 
$$\left|\frac 1{1+\JJ(z-x)} - \frac 1{1+\JJ(x)}\right| \lesssim \frac{\|z\|}{(1+\|x\|)(1+\|z-x\|)}.$$
Therefore, for any $\chi_k\ge 2$, 
\begin{align*}
& \sum_{\substack{u,x\in \Z^5 \\ \JJ(z) \ge (\JJ(u)\chi_k)\vee \frac{\JJ(x)}{\chi_k}}} p_{\epsilon_k}(u) p_k(x)  G(z) |G(z-u)-G(z)| \left|\frac  1{1+\JJ(z-x)} - \frac 1{1+\JJ(x)}\right| \\ 
&\lesssim \sqrt{\epsilon_k} \sum_{x\in \Z^5} \frac{p_k(x)}{1+\|x\|} \sum_{\JJ(z) \ge \JJ(x)/\chi_k } \frac{1}{\|z\|^6(1+\|z-x\|)} 
 \stackrel{\eqref{exp.Green.x}}{\lesssim} \frac{\chi_k^2\sqrt{\epsilon_k}}{k^{3/2}}.  
\end{align*}
On the other hand by taking $\chi_k = (k/\epsilon_k)^{1/6}$, we get using \eqref{pn.largex} and \eqref{Sn.large},
\begin{align*}
 \sum_{\substack{x,z\in \Z^5 \\ \JJ(u)\ge \sqrt{\epsilon_k}\chi_k }}  p_{\epsilon_k}(u) p_k(x)  G(z) G(z-u) \left(\frac  1{1+\JJ(z-x)} + \frac 1{1+\JJ(x)}\right)   \lesssim  \frac{1}{\chi_k^5\sqrt{k\epsilon_k}} = o \left(\frac 1k\right),
\end{align*}
\begin{align*}
& \sum_{\substack{u,z\in \Z^5 \\ \JJ(x)\le \sqrt{k}/\chi_k }} p_{\epsilon_k}(u) p_k(x)  G(z) G(z-u) \left(\frac  1{1+\JJ(z-x)} + \frac 1{1+\JJ(x)}\right) = o \left(\frac 1k\right). 
\end{align*}
As a consequence, since $\JJ(u)\le \sqrt{\epsilon_k} \chi_k$ and $\JJ(x)\ge \sqrt{k}/\chi_k$, implies $\JJ(u)\le \JJ(x)/\chi_k$, with our choice of $\chi_k$, 
we get as wanted (using also symmetry of the step distribution) that 
\begin{align}\label{main.tau.combined.1}
& \bP[\overline \tau_2\le \overline \tau_1\circ \overline \tau_2<\infty]  - \bP[\overline \tau_1<\infty] \cdot \bP[\overline \tau_2<\infty] \\
\nonumber & = c_0 \sum_{x,z\in \Z^5} p_k(x)  G(z)^2 \left(\frac 1{1+\JJ(z-x)} - \frac{1}{1+\mathcal J(x)}\right) + o \left(\frac 1k\right)\\
\nonumber &=  \frac{c_0}{2} \sum_{x,z\in \Z^5} p_k(x) G(z)^2 \left(\frac 1{1+\JJ(z-x)} +\frac 1{1+\JJ(z+x)}- \frac{2}{1+\mathcal J(x)}\right) + o \left(\frac 1k\right). 
\end{align}

\underline{Step 5.} The previous steps show that 
$$\cov\left(\{\overline \tau_1<\infty\} , \{\overline \tau_2<\infty\}\right)  = c_0 \sum_{x,z\in \Z^5} p_k(x) \left(\frac{G(z)G(z-x)}{1+\JJ(z)} + \frac{G(z)^2}{1+\JJ(z-x)} - \frac{G(z)^2}{1+\JJ(x)}\right).$$
Now by approximating the series with an integral (recall \eqref{Green.asymp}), and doing a change of variables, we get with $u:=x/\JJ(x)$ and $v:=\Lambda^{-1} u$, and for some constant $c>0$ (that might change from line to line), 
\begin{align}\label{last.lem.integral}
\nonumber  & \sum_{z\in \Z^5}\left(\frac{G(z)G(z-x)}{1+\JJ(z)} +  \frac{G(z)^2}{1+\JJ(z-x)} - \frac{G(z)^2}{1+\JJ(x)}\right) \\
\nonumber & \sim c \int_{\R^5} \left\{\frac{1}{\JJ(z)^4\cdot \JJ(z-x)^3} + \frac{1}{\JJ(z)^6} \left(\frac 1{\JJ(z-x)} -\frac 1{\JJ(x)}\right)\right\}\, dz \\
\nonumber & = \frac{c}{\JJ(x)^2} \int_{\R^5}\left\{\frac{1}{\JJ(z)^4\cdot \JJ(z-u)^3} + \frac{1}{\JJ(z)^6} \left(\frac 1{\JJ(z-u)} -1\right)\right\}\, dz \\
& = \frac{c}{\JJ(x)^2} \int_{\R^5} \left\{\frac{1}{\|z\|^4\cdot \|z-v\|^3} + \frac{1}{\|z\|^6}\left(\frac 1{\|z-v\|} -1\right)\right\} \, dz. 
\end{align}
Note that the last integral is convergent and independent of $v$ (and thus of $x$ as well) by rotational invariance. Therefore, since $\sum_{x\in \Z^5} p_k(x) / \JJ(x)^2\sim \sigma/k$, for some constant $\sigma>0$ (for instance by applying Theorem \ref{LCLT}), it only remains to show that the integral above is positive. 
To see this, we use that the map $z\mapsto \|z\|^{-3}$ is harmonic outside the origin, and thus satisfies the mean value property on $\R^5\setminus \{0\}$. 
In particular, using also the rotational invariance, this shows (with $\BB_1$ the unit Euclidean ball and $\partial \BB_1$ the unit sphere), 
\begin{align}\label{last.lem2.a}
\int_{\BB_1^c} \frac{1}{\|z\|^4\cdot \|z-v\|^3}\, dz &=  \frac 1{|\partial \BB_1|}  \int_{\partial \BB_1} \, dv \int_{\BB_1^c} \frac{1}{\|z\|^4\cdot \|z-v\|^3}\, dz\\
\nonumber & = \int_{\BB_1^c} \frac 1{\|z\|^7} \, dz = c_1 \int_1^\infty \frac 1{r^3}\, dr = \frac {c_1}2, 
\end{align}
for some constant $c_1>0$. 
Likewise, 
\begin{equation}\label{last.lem2.b}
\int_{\BB_1} \frac{1}{\|z\|^4\cdot \|z-v\|^3}\, dz =  \frac{c_1}{|\partial \BB_1|} \int_0^1 \, dr \int_{\partial \BB_1} \frac{du}{\|ru - v\|^3}  = c_1, 
\end{equation} 
with the same constant $c_1$ as in the previous display. 
On the other hand 
\begin{equation}\label{last.lem2.c}
\int_{\BB_1^c} \frac 1{\|z\|^6}\, dz = c_1 \int_1^\infty \frac 1{r^2} \, dr = c_1. 
\end{equation}
Furthermore, using again the rotational invariance, 
\begin{align}\label{last.lem.2}
&  \int_{\BB_1} \frac{1}{\|z\|^6}\left(\frac 1{\|z-v\|} -1\right) \, dz = \int_{\BB_1} \frac{1}{\|z\|^6}\left(\frac 1{2\|z-v\|} + \frac 1{2\|z+v\|} -1\right) \, dz \\
\nonumber & =  \frac{c_1}{|\partial \BB_1|} \int_0^1 \frac {dr}{r^2} \int_{\partial \BB_1}\left(\frac 1{2\|v-ru\|} + \frac 1{2\|v+ru\|} -1\right)\, du. 
\end{align}
Now we claim that for any $u,v\in \partial \BB_1$, and any $r\in (0,1)$, 
\begin{equation}\label{claim.geom}
\frac 12\left(\frac 1{\|v-ru\|} + \frac 1{\|v+ru\|}\right) \ge \frac{1}{\sqrt{1+r^2}}.
\end{equation}
Before we prove this claim, let us see how we can conclude the proof. 
It suffices to notice that, if $f(s) = (1+s^2)^{-1/2}$, then $f'(s) \ge - s$, for all $s\in (0,1)$, and thus   
\begin{equation}\label{lower.sqrt}
\frac 1{\sqrt{1+r^2}} - 1 = f(r) - f(0) \ge - \int_0^r s\, ds \ge -r^2/2.
\end{equation} 
Inserting this and \eqref{claim.geom} in \eqref{last.lem.2} gives 
$$ \int_{\BB_1} \frac{1}{\|z\|^6}\left(\frac 1{\|z-v\|} -1\right) \, dz \ge -\frac {c_1}{2}. $$
Together with \eqref{last.lem2.a}, \eqref{last.lem2.b}, \eqref{last.lem2.c}, this shows that the integral in \eqref{last.lem.integral} is well positive. Thus all that remains to do is proving the claim \eqref{claim.geom}. Since the origin, $v$, $v+ru$, and $v-ru$ all lie in a common two-dimensional plane, one can 
always work in the complex plane, and assume for simplicity that $v= 1$, and $u=e^{i\theta}$, for some $\theta\in [0,\pi/2]$. In this case, the claim is equivalent to showing that
$$\frac 12 \left( \frac 1{\sqrt{1+ r^2 + 2r\cos \theta} }+ \frac 1{\sqrt{1+ r^2 - 2r\cos \theta}}\right) \ge \frac 1{\sqrt{1+r^2}},$$
which is easily obtained using that the left hand side is a decreasing function of $\theta$. 
This concludes the proof of Lemma \ref{lem.var.2}. 
 \hfill $\square$ 

\begin{remark}\emph{Note that the estimate of the covariance mentioned in the introduction, in case $(ii)$, can now be done as well. Indeed, denoting by 
$$\hat \tau_2:= \inf\{n\ge k+1\, :\, S_n\in S_k + \RR_\infty^2\},$$
it only remains to show that 
$$\left|\bP[\hat \tau_2\le k+\epsilon_k, \,  \overline  \tau_1<\infty ] - \bP[\hat \tau_2\le k+ \epsilon_k] \cdot \bP[\overline \tau_1<\infty]\right| = o\left(\frac 1k\right).$$
Using similar estimates as above we get, with $\chi_k = (k/\epsilon_k)^{4/5}$,   
\begin{align*}
& \left|\bP[\hat \tau_2\le k+ \epsilon_k,   \overline \tau_1<\infty ] - \bP[\hat \tau_2\le k+ \epsilon_k] \cdot \bP[\overline \tau_1<\infty]\right| \\
 \stackrel{\eqref{Sn.large}}{=} & \left|\bP[\hat \tau_2\le k+ \epsilon_k, \|S_{\hat \tau_2} - S_k\|\le \sqrt{\epsilon_k\chi_k},  \overline \tau_1<\infty ] - \bP[\hat \tau_2\le k+ \epsilon_k] \bP[\overline \tau_1<\infty]\right|  + \OO\left(\frac 1{\sqrt k \chi_k^{\frac 52}}\right)\\
  =& \sum_{\substack{x\in \Z^5  \\ \|y\|\le \sqrt{\epsilon_k \chi_k}} } \left|\frac{p_k(x+y) + p_k(x-y)}{2}-p_k(x)\right| \bP[\hat \tau_2\le k+\epsilon_k, S_{\hat \tau_2} -S_k= y] \varphi_x + \OO\left(\frac 1{\sqrt k \chi_k^{\frac 52}}\right)\\
 \lesssim & \frac 1{k^{\frac 32}} \E\left[\|S_{\hat \tau_2}-S_k\|^2 \1\{\|S_{\hat \tau_2}-S_k\|\le \sqrt{\epsilon_k\chi_k} \}\right] + \frac 1{\sqrt k \chi_k^{\frac 52}}\lesssim \frac 1{\sqrt k \chi_k^{\frac 52}}+ \frac {\sqrt{\epsilon_k\chi_k}}{k^{\frac 32}},
\end{align*}
using that by \eqref{lem.hit.2} and the Markov property, one has $\bP[\|S_{\hat \tau_2}-S_k\|\ge t] \lesssim \frac 1t $. 
 }
\end{remark}

%%%%%%%%%%%%%%%%%%%%%%%%%%%%%%%%%%%%%%%%%%%%%%%%%%%%%%%%%%%%%%%%%%%%%%%%%%%%%%%%%%%%%
\subsection{Proof of Lemma \ref{lem.var.3}}
We consider only the case of $\cov(Z_0\phi_2,Z_k\psi_1)$, the other one being entirely similar. Define 
$$\tau_1:= \inf\{n\ge 0  :  S_n^1\in \RR[\epsilon_k,k]\},\  \tau_2:=\inf\{n\ge 0 :  S_k+S_n^2\in \RR(-\infty, 0] \},$$
with $S^1$ and $S^2$ two independent walks, independent of $S$. 
The first step is to see that 
$$\cov(Z_0\phi_3,Z_k\psi_2)= \rho^2\cdot \cov(\1\{\tau_1<\infty\},\1\{\tau_2<\infty\}) +o\left(\frac 1k\right),$$
with $\rho$ as in \eqref{def.rho}. Since the proof of this fact has exactly the same flavor as in the two previous lemmas, we omit the details and directly move to the next step.

Let $\eta\in (0,1/2)$ be some fixed constant (which will be sent to zero later). Notice first that 
\begin{align}\label{eta.tau1}
 & \nonumber \bP\left[S^1_{\tau_1} \in \RR[(1-\eta) k,k],\, \tau_2<\infty\right]  \stackrel{\eqref{Green.hit}, \eqref{lem.hit.2}}{\lesssim} \sum_{i=\lfloor (1-\eta)k\rfloor }^k 
\E\left[\frac{G(S_i)}{1+\|S_k\|}\right] \\
&   \stackrel{\eqref{exp.Green}}{\lesssim}  \sum_{i=\lfloor (1-\eta)k\rfloor }^k 
\frac{\E\left[G(S_i)\right] }{1+\sqrt{k-i}} 
 \stackrel{\eqref{exp.Green}}{\lesssim}\frac{\sqrt{\eta}}{k}.   
\end{align}
Next, fix another constant $\delta\in (0,1/4)$ (which will be soon chosen small enough). 
Then let $N: = \lfloor (1-\eta)k/\epsilon_k^{1-\delta}\rfloor$, and for $i=1,\dots,N$,  define 
$$\tau_1^i:= \inf\{n\ge 0 \, :\, S_n^1 \in \RR[k_i,k_{i+1}]\},\quad \text{with}\quad k_i:= \epsilon_k + i\lfloor \epsilon_k^{1-\delta}\rfloor .$$
We claim that with sufficiently high probability, at most one of these hitting times is finite. Indeed, for $i\le N$, 
set $I_i := \{k_i,\dots,k_{i+1}\}$, and notice that  
\begin{align*}
&\sum_{1\le i< j\le N} \bP[\tau_1^i <\infty, \, \tau_1^j<\infty,\, \tau_2<\infty] \\
 \le & \sum_{1\le i< j\le N} \left(\bP[\tau_1^i \le \tau_1^j<\infty,\, \tau_2<\infty] + \bP[\tau_1^j \le \tau_1^i<\infty,\, \tau_2<\infty]\right) \\
 \stackrel{\eqref{lem.hit.1},  \eqref{lem.hit.2}}{\lesssim} & \sum_{\substack{i=1,\dots,N, j\neq i  \\ \ell \in I_i, m\in I_j}}  \E\left[\frac{G(S_\ell - S_m) G(S_m)}{1+\|S_k\|} \right] \lesssim \frac{1}{\sqrt{k}} \sum_{\substack{i=1,\dots,N, j\neq i  \\ \ell \in I_i, m\in I_j}} \E\left[G(S_\ell - S_m) G(S_m)\right]\\
 \stackrel{\eqref{exp.Green},  \eqref{exp.Green.x}}{\lesssim} & \frac{1}{\sqrt{k}} \sum_{\substack{i=1,\dots, N, j\neq i \\ \ell \in I_i, m\in I_j}}  \frac{1}{(1+ |m-\ell |^{3/2}) (m\wedge \ell)^{3/2}} \lesssim \frac {N\epsilon_k^{(1-\delta)/2}}{\epsilon_k^{3/2}\sqrt k} = o\left(\frac 1k\right), 
\end{align*}
where the last equality follows by assuming $\epsilon_k\ge k^{1-c}$, with $c>0$ small enough. 
Therefore, as claimed 
$$\bP[\tau_1<\infty,\, \tau_2<\infty ] = \sum_{i=1}^{N} \bP[\tau_1^i <\infty, \, \tau_2<\infty] + o\left(\frac 1k\right), $$
and one can show as well that,
$$\bP[\tau_1<\infty]\cdot \bP[ \tau_2<\infty ] = \sum_{i=1}^{N-2} \bP[\tau_1^i <\infty] \cdot \bP[\tau_2<\infty] +o\left(\frac 1k\right).$$
Next, observe that for any $i\le N$, using H\"older's inequality at the third line, 
\begin{align*}
&\bP\left[\tau_1^i <\infty,  \tau_2<\infty,  \|S_{k_{i+1}} - S_{k_i}\|^2\ge \epsilon_k^{1-\delta/2}\right]  \stackrel{\eqref{Green.hit}, \eqref{lem.hit.2}}{\lesssim} \sum_{j=k_i}^{k_{i+1}} \E\left[ \frac{G(S_j)\1\{\|S_{k_{i+1}} - S_{k_i}\|^2\ge \epsilon_k^{1-\delta/2}\} }{1+\|S_k\|}\right]\\
& \stackrel{\eqref{exp.Green}}{\lesssim} \frac {1}{\sqrt k} \sum_{j=k_i}^{k_{i+1}} \E\left[ G(S_j)\1\{\|S_{k_{i+1}} - S_{k_i}\|^2\ge \epsilon_k^{1-\delta/2}\} \right]\\
& \lesssim \frac {1}{\sqrt k}\left( \sum_{j=k_i}^{k_{i+1}} \E\left[\frac 1{1+\|S_j\|^4}\right]^{3/4}\right) \cdot \bP\left[\|S_{k_{i+1}} - S_{k_i}\|^2\ge \epsilon_k^{1-\delta/2}\right]^{1/4} \\
&  \stackrel{\eqref{Sn.large}}{\lesssim} \frac {\epsilon_k^{1-\delta}}{  k_i^{3/2}\sqrt{k}} \cdot \frac 1{\epsilon_k^{5\delta/16}}= o\left(\frac 1{Nk}\right), 
\end{align*}
by choosing again $\epsilon_k\ge k^{1-c}$, with $c$ small enough. 
Similarly, one has using Cauchy-Schwarz, 
\begin{align*}
& \bP\left[\tau_1^i <\infty,  \tau_2<\infty,  \|S_k - S_{k_{i+1}}\|^2\ge k \epsilon_k^{\delta/2}\right]   \lesssim \sum_{j=k_i}^{k_{i+1}} \E\left[ \frac{G(S_j)\1\{\|S_k - S_{k_{i+1}}\|^2 \ge k \epsilon_k^{\delta/2}\} }{1+\|S_k\|}\right]\\ 
& \lesssim \frac{1}{\epsilon_k^{5\delta/8}}\sum_{j=k_i}^{k_{i+1}} \E\left[ G(S_j) \E\left[\frac 1{1+\|S_k\|^2} \mid S_j \right]^{1/2}\right]\lesssim  \frac {\epsilon_k^{1-\delta}}{  k_i^{3/2}\sqrt{k}} \cdot \frac 1{\epsilon_k^{5\delta/8}}= o\left(\frac 1{Nk}\right).
\end{align*}
As a consequence, using also Theorem \ref{LCLT}, one has for $i\le N$, and with $\ell := k_{i+1}-k_i$, 
 \begin{align*}
&  \bP[\tau_1^i <\infty, \, \tau_2<\infty]  \\
& =\sum_{x\in \Z^5} \sum_{\substack{ \|z\|^2 \le  k \epsilon_k^{\delta/2} \\ \|y\|^2 \le \epsilon_k^{1-\delta/2} }}  p_{k_i}(x)\bP_{0,x}\left[ \RR_\infty \cap  \tilde \RR[0,\ell] \neq\emptyset,  \tilde S_{\ell} = y\right] p_{k-k_{i+1}}(z-y) \phi_{x+z} + o\left(\frac 1{Nk}\right)\\
& = \sum_{x\in \Z^5} \sum_{\underset{\|z\|^2 \le  k \epsilon_k^{\delta/2} }{\|y\|^2 \le \epsilon_k^{1-\delta/2}}  }  p_{k_i}(x) \bP_{0,x} 
\left[ \RR_\infty \cap \tilde \RR[0,\ell] \neq\emptyset, \tilde  S_\ell = y\right] p_{k-k_i}(z) \phi_{x+z} +  o\left(\frac 1{Nk}\right)\\
& = \sum_{x,z\in \Z^5}   p_{k_i}(x) \bP_{0,x} \left[ \RR_\infty \cap \tilde \RR[0,\ell] \neq\emptyset \right] p_{k-k_i}(z) \phi_{x+z} +  o\left(\frac 1{Nk}\right).
 \end{align*}
Moreover, Theorem \ref{thm.asymptotic} yields for any nonzero $x\in \Z^5$, and some $\nu>0$, 
 \begin{align}\label{RtildeRell}
\bP_{0,x} \left[ \RR_\infty \cap \tilde \RR[0,\ell] \neq\emptyset \right] = \frac{\gamma_5}{\kappa}\cdot \E\left[ \sum_{j=0}^{\ell} G(x+\tilde S_j) \right] + \OO\left(\frac {\log(1+ \|x\|)}{\|x\| (\|x\|\wedge \ell)^{\nu}}\right). 
\end{align}
Note also that for any $\epsilon \in [0,1]$, 
 \begin{align*}
 \sum_{x,z\in \Z^5}   \frac{p_{k_i}(x)}{1+\|x\|^{1+\epsilon}} p_{k-k_i}(z) \phi_{x+z} = \E\left[\frac 1{(1+\|S_{k_i}\|^{1+\epsilon})(1+\|S_k\|)}\right] 
 \lesssim  \frac 1{\sqrt{k_i}^{1+\epsilon} \sqrt k}, 
 \end{align*}
 and thus 
 $$\sum_{i=1}^N  \sum_{x,z\in \Z^5}   \frac{p_{k_i}(x)}{1+\|x\|^{1+\epsilon}} p_{k-k_i}(z) \phi_{x+z}  = \OO\left(\frac 1{\ell k^{\epsilon}}\right).$$
 In particular, the error term in \eqref{RtildeRell} can be neglected, as we take for instance $\delta=\nu/2$, and  $\epsilon_k\ge k^{1-c}$, with $c$ small enough.  
It amounts now to estimate the other term in \eqref{RtildeRell}. By \eqref{pn.largex}, for any $x\in \Z^5$ and $j\ge 0$, 
\begin{equation*}
\E[G(x+S_j)] = G_j(x) = G(x)-\mathcal O(\frac{j}{1+ \|x\|^d}).
\end{equation*}
As will become clear the error term can be neglected here. 
Furthermore, similar computations as above show that for any $j\in \{k_i,\dots,k_{i+1}\}$, 
$$\sum_{x,z\in \Z^5}   p_{k_i}(x) G(x) p_{k-k_i}(z) \phi_{x+z}  = \sum_{x,z\in \Z^5}   p_j(x) G(x) p_{k-j}(z) \phi_{x+z}+  o\left(\frac 1{Nk}\right),$$
Altogether, and applying once more Theorem \ref{thm.asymptotic}, this gives for some $c_0>0$, 
 \begin{equation}\label{cov.3.sum}
 \sum_{i=1}^N \bP[\tau_1^i <\infty,  \tau_2<\infty] = \sum_{j=\epsilon_k}^{(1-\eta) k} \E[G(S_j)\phi_{S_k}]+ o\left(\frac 1k\right)  =c_0 \sum_{j=\epsilon_k}^{\lfloor (1-\eta) k\rfloor} \E\left[\frac{G(S_j)}{1+\JJ(S_k)}\right]+ o\left(\frac 1k\right).
\end{equation}
We treat the first terms of the sum separately. Concerning the other ones notice that by \eqref{Green.asymp} and Donsker's invariance principle, one has 
\begin{align*}
 \sum_{j=\lfloor \eta k\rfloor}^{\lfloor (1-\eta)k\rfloor} \E\left[\frac{G(S_j)}{1+\JJ(S_k)}\right]  & = \frac 1k \int_\eta^{1-\eta} \E\left[\frac{G(\Lambda \beta_s)}{\JJ(\Lambda \beta_1)}\right] \, ds+ o\left(\frac 1k\right)  \\
 & = \frac {c_5}k \int_\eta^{1-\eta} \E\left[\frac{1}{\|\beta_s\|^3 \cdot \|\beta_1\|}\right] \, ds+ o\left(\frac 1k\right), 
\end{align*}
with $(\beta_s)_{s\ge 0}$ a standard Brownian motion, and $c_5>0$ the constant that appears in \eqref{Green.asymp}. 
In the same way, one has 
\begin{align*}  
\sum_{i=1}^N \bP[\tau_1^i <\infty] \cdot \bP[\tau_2<\infty]  & = c_0 \sum_{j=\epsilon_k}^{\lfloor \eta k\rfloor} \E[G(S_j)]\E\left[\frac{1}{1+\JJ(S_k)}\right] \\
& \quad + \frac{c_0c_5}{k} \int_\eta^{1-\eta}\E\left[\frac{1}{\|\beta_s\|^3}\right]  \E\left[ \frac 1{\|\beta_1\|}\right] \, ds+ o\left(\frac 1k\right), 
\end{align*}
with the same constant $c_0$, as in \eqref{cov.3.sum}. 
We next handle the  sum of the first terms in \eqref{cov.3.sum} and show that its difference with the sum from the previous display  is negligible. 
Indeed, observe already that with $\chi_k := k/(\eta\epsilon_k)$, 
$$\sum_{j=\epsilon_k}^{\lfloor \eta k\rfloor} \E\left[\frac{G(S_j)\1\{\|S_j\|\ge \eta^{1/4} \sqrt k\}}{1+\JJ(S_k)}\right]+ \E\left[\frac{G(S_j)\1\{\|S_k\|\ge \sqrt{k\chi_k}\}}{1+\JJ(S_k)}\right] \lesssim \frac{\eta^{1/4}}{k}. $$ 
Thus one has, using Theorem \ref{LCLT},
\begin{align}\label{eta.tau1.bis}
 &\sum_{j=\epsilon_k}^{\lfloor \eta k\rfloor}\left| \E\left[\frac{G(S_j)}{1+\JJ(S_k)}\right] -  \E[G(S_j)]\cdot \E\left[\frac{1}{1+\JJ(S_k)}\right]\right|  \\
  \nonumber & \lesssim \sum_{j=\epsilon_k}^{\lfloor \eta k\rfloor} \sum_{\underset{\|x\|\le \eta^{1/4}\sqrt{k}}{\|z\|\le \sqrt{k\chi_k}} } \frac{p_j(x)G(x)}{1+\|z\|} \left|\overline p_{k-j}(z-x) +\overline  p_{k-j}(z+x) -2 \overline p_k(z)\right|  + \frac {\eta^{1/4}}k \lesssim  \frac {\eta^{1/4}}k. 
\end{align}
Define now for $s\in (0,1]$, 
$$ H_s: = \E\left[\frac 1{\|\beta_s\|^3 \|\beta_1\|}\right]  - \E\left[\frac 1{\|\beta_s\|^3}\right] \cdot \E\left[\frac 1{\|\beta_1\|}\right].$$
Let $f_s(\cdot)$ be the density of $\beta_s$ and notice that as $s\to 0$, 
\begin{align*}
& H_s=\int_{\R^5} \int_{\R^5} \frac { f_s(x)  f_{1-s}(y)}{\|x\|^3 \|x+y\|} \, dx\, dy - \int_{\R^5} \int_{\R^5} \frac { f_s(x)  f_1(y)}{\|x\|^3 \|y\|} \, dx\, dy \\
& = \frac {1}{s^{3/2}} \int_{\R^5} \int_{\R^5}  \frac{f_1(x)f_1(y)}{\|x\|^3} \left(
\frac {1}{\|y\sqrt{1-s} +x \sqrt  s \| }  -\frac {1}{\|y\|}\right)   \, dx\, dy\\ 
& =  \frac {1}{s^{3/2}} \int_{\R^5} \int_{\R^5}  \frac{f_1(x)f_1(y)}{\|x\|^3\|y\|}  
\left\{\left(\frac 12 + \frac{\|x\|^2}{2\|y\|^2} + \frac{\langle x,y\rangle^2}{\|y\|^4} \right)s + \OO(s^{3/2})  \right\}\,  dx\,  dy =\frac c{\sqrt{s}} + \OO(1), 
\end{align*}
with $c>0$. 
Thus the map $s\mapsto H_s$ is integrable at $0$, and since it is also continuous on $(0,1]$, its integral on this interval is well defined. Since $\eta$ can be taken arbitrarily small
in \eqref{eta.tau1} and \eqref{eta.tau1.bis}, in order to finish the proof  it just remains to show that the integral of $H_s$ on $(0,1]$ is positive.

To this end, note first that $\tilde \beta_{1-s}:= \beta_1 - \beta_s$ is independent of $\beta_s$. We use then \eqref{claim.geom}, which implies, 
with $q=\E[1/\|\beta_1\|^3]$, 
\begin{align*}
 & \E\left[\frac 1{\|\beta_s\|^3 \|\beta_1\|}\right]   = \E\left[\frac 1{\|\beta_s\|^3 \|\beta_s + \tilde \beta_{1-s}\|}\right]
  \ge  \E\left[\frac 1{\|\beta_s\|^3 \sqrt{\|\beta_s\|^2 + \|\tilde \beta_{1-s}\|^2}}\right]\\
  & =  \frac{(5q)^2}{s^{3/2}} \int_0^\infty \int_0^\infty \frac{r e^{-\frac 52 r^2} u^4 e^{-\frac 52u^2}}{\sqrt{sr^2 + (1-s)u^2}} \, dr\,  du\\
   &   = \frac{q^2}{5s^{3/2}} \int_0^\infty \int_0^\infty \frac{r e^{-\frac{r^2}2} u^4 e^{-\frac{u^2}2}}{\sqrt{sr^2 + (1-s)u^2}} \, dr\,   du. 
 \end{align*}  
 We split the double integral in two parts, one on the set $\{sr^2\le (1-s)u^2\}$, and the other one on the complementary set $\{sr^2\ge (1-s)u^2\}$. 
 Call respectively $I_s^1$ and $I_s^2$ the integrals on these two sets. For $I_s^1$, \eqref{lower.sqrt} gives 
 \begin{align*} 
 I_s^1 & \ge \frac 1{\sqrt{1-s}} \int_0^\infty u^3 e^{-\frac{u^2}2} \int_0^{\sqrt{\frac{1-s}{s}}u} r e^{-\frac{r^2}2}\, dr  \, du\\
& \qquad - \frac{s}{2(1-s)^{3/2}}  \int_0^\infty u e^{-\frac{u^2}2} \int_0^{\sqrt{\frac{1-s}{s}}u} r^3 e^{-\frac{r^2}2}\, dr  \, du\\
& = \frac{2(1-s^2)}{\sqrt{1-s}} + \frac{s^2}{\sqrt{1-s}} - \frac{s}{\sqrt{1-s}}=  \frac{2-s - s^2}{\sqrt{1-s}}. 
 \end{align*}
For $I_s^2$ we simply use the rough bound: 
$$I_s^2 \ge \frac 1{\sqrt{2s}} \int_0^\infty \int_0^\infty e^{-\frac{r^2}2}  u^4 e^{-\frac{u^2}2}\1\{sr^2\ge (1-s)u^2\} \, dr  \, du, $$
which entails
\begin{align*}
& \int_0^1 \frac{I_s^2}{s^{3/2}}\, ds \ge  \frac 1{\sqrt 2} \int_0^\infty \int_0^\infty e^{-\frac{r^2}2}  u^4 e^{-\frac{u^2}2} 
\left(\int_{\frac{u^2}{u^2+r^2}}^1 \frac 1{s^2}\, ds\right)\, dr \, du \\
& = \frac 1{\sqrt 2} \left(\int_0^\infty  r^2 e^{-\frac{r^2}2} \, dr\right)^2= \frac 1{\sqrt 2} \left(\int_0^\infty  e^{-\frac{r^2}2}\, dr\right)^2  =\frac{\pi}{2\sqrt 2}>1, 
\end{align*}
where for the last inequality we use $\sqrt{2}<3/2$. 
Note now that 
 $$\E\left[\frac 1{\|\beta_s\|^3}\right] \cdot \E\left[\frac 1 {\|\beta_1\|}\right] = \frac {2q^2}{5s^{3/2}} , $$
 and 
\begin{align*}
& \int_0^1 \frac{I_s^1 - 2}{s^{3/2}}\, ds \ge \int_0^1 s^{-3/2} \left\{(2- s- s^2)(1+\frac s2 + \frac{3s^2}{8}) - 2\right\} \, ds  \\
& = - \int_0^1 (\frac 34 \sqrt s + \frac 78 s^{3/2} +\frac {3}{8} s^{5/2} ) \, ds = - (\frac 12 +  \frac 7{20} + \frac{3}{28}) = - \frac{134}{140} > -1. 
\end{align*}
Altogether this shows that the integral of $H_s$ on $(0,1]$ is well positive as wanted. This concludes the proof of the lemma. 
 \hfill $\square$ 

\begin{remark}\emph{The value of $H_1$ can be computed explicitely and one can check that it is positive. Similarly, by computing the leading order term in $I_s^2$, we could show that $H_s$ is also positive in a neighborhood of the origin, but it would be interesting to know whether $H_s$ is positive for all $s\in (0,1)$. }
\end{remark}

%%%%%%%%%%%%%%%%%%%%%%%%%%%%%%%%%%%%%%%%%%%%%%%%%%%%%%%%%%%%%%%%%%%%%%%%%%%%%%%%%%%%%%
\subsection{Proof of Lemma \ref{lem.var.4}}
We define here 
$$\tau_1:= \inf\{n\ge 0  :  S_n^1\in \RR[\epsilon_k,k-\epsilon_k]\},\  \tau_2:=\inf\{n\ge 0 :  S_k+S_n^2\in \RR[\epsilon_k,k-\epsilon_k]\},$$
with $S^1$ and $S^2$ two independent walks, independent of $S$. As in the previous lemma, we omit the details of the fact that  
$$\cov(Z_0\phi_2,Z_k\psi_2)= \rho^2\cdot \cov(\1\{\tau_1<\infty\},\1\{\tau_2<\infty\}) +o\left(\frac 1k\right).$$
Then we define $N:=\lfloor (k-3\epsilon_k)/\epsilon_k\rfloor$ and let $(\tau_1^i)_{i=1,\dots,N}$ be as in the proof of Lemma \ref{lem.var.3}. 
Define also $(\tau_2^i)_{i=1,\dots,N}$ analogously. Similarly as before one can see that 
\begin{eqnarray}\label{tau1i.tau2j}
\bP[\tau_1<\infty,  \tau_2<\infty ]= \sum_{i=1}^{N} \sum_{j=1}^{N} \bP[\tau_1^i <\infty,  \tau_2^j <\infty ] + o\left(\frac 1k\right).
\end{eqnarray}
Note also that for any $i$ and $j$, with $|i-j| \le 1$, by \eqref{Green.hit} and \eqref{exp.Green},
$$\bP[\tau_1^i<\infty, \, \tau_2^j<\infty] = \OO\left(\frac {\epsilon_k^{2(1-\delta)}}{k_i^{3/2} (k-k_i)^{3/2} }\right), $$
so that in \eqref{tau1i.tau2j}, one can consider only the sum on the indices $i$ and $j$ satisfying $|i-j|\ge 2$. 
Furthermore, when $i<j$, the events $\{\tau_1^i<\infty\}$ and $\{\tau_2^j<\infty\}$ are independent. Thus altogether this gives   
\begin{align*}
 \cov( &\1\{\tau_1<\infty\},  \1\{\tau_2<\infty\})  \\
 & = \sum_{i = 1}^{N-2} \sum_{j=i+2}^N \left( \bP[ \tau_1^j <\infty, \tau_2^i<\infty] - \bP[\tau_1^j<\infty]  \bP[\tau_2^i<\infty] \right) + o\left(\frac 1k\right). 
\end{align*}
Then by following carefully the same steps as in the proof of the previous lemma we arrive at 
$$\cov(\1\{\tau_1<\infty\}, \, \1\{\tau_2<\infty\})  = \frac {c}{k} \int_0^1 \tilde H_t\, dt + o\left(\frac 1k\right), $$
with $c>0$ some positive constant and,   
$$\tilde H_t := \int_0^t \left(\E\left[\frac 1{\|\beta_s-\beta_1\|^3 \cdot \|\beta_t\|^3} \right] - \E\left[\frac 1{\|\beta_s-\beta_1\|^3}\right] \cdot\E\left[\frac 1{ \|\beta_t\|^3} \right] \right) \, ds,$$
at least provided we show first that $\tilde H_t$ it is well defined and that its integral over $[0,1]$ is convergent. 
However, observe that for any $t\in (0,1)$, one has with $q=\E[\|\beta_1\|^{-3}]$, 
$$\int_0^t \E\left[\frac 1{\|\beta_s-\beta_1\|^3}\right] \cdot\E\left[\frac 1{ \|\beta_t\|^3} \right] =\frac{q^2}{t^{3/2}} \int_0^t \frac 1{(1-s)^{3/2}}\, ds = \frac{2q^2(1-\sqrt{1-t})}{t^{3/2}\sqrt{1-t}},$$ 
and therefore this part is integrable on $[0,1]$. This implies in fact that the other part in the definition of $\tilde H_t$ is also well defined and integrable, since we already know that 
$\cov(\1\{\tau_1<\infty\}, \, \1\{\tau_2<\infty\})=\OO(1/k)$.  Thus it only remains to show that the integral of $\tilde H_t$ on $[0,1]$ is positive.  
To this end, we write $\beta_t = \beta_s + \gamma_{t-s}$, and $\beta_1 = \beta_s + \gamma_{t-s} + \delta_{1-t}$, with $(\gamma_u)_{u\ge 0}$ and $(\delta_u)_{u\ge 0}$ 
two independent Brownian motions, independent of $\beta$. Furthermore, using that the map $z\mapsto 1/\|z\|^3$ is harmonic outside the origin, we can compute: 
\begin{align*}
& I_1:= \E\left[\frac {\1\{\|\beta_s\|\ge \|\gamma_{t-s}\|\ge \|\delta_{1-t}\| \} }{\|\beta_s-\beta_1\|^3 \cdot \|\beta_t\|^3} \right] = \E\left[\frac {\1\{\|\beta_s\|\ge \|\gamma_{t-s}\|\ge \|\delta_{1-t}\| \} }{\|\gamma_{t-s} + \delta_{1-t}\|^3 \cdot \|\beta_s\|^3} \right]  \\
= &  \frac{5q}{s^{3/2}} \E\left[\frac{\1\{\|\gamma_{t-s}\|\ge \|\delta_{1-t}\| \} }{\|\gamma_{t-s} + \delta_{1-t}\|^3} \int_{\frac{\|\gamma_{t-s}\|}{\sqrt s}}^\infty re^{-\frac 52 r^2} \, dr\right]  = \frac{q}{s^{3/2}}  \E\left[\frac{\1\{\|\gamma_{t-s}\|\ge \|\delta_{1-t}\| \}}{\|\gamma_{t-s} + \delta_{1-t}\|^3} e^{-\frac 5{2s} \|\gamma_{t-s}\|^2} \right] \\
 =& \frac{q}{s^{3/2}}  \E\left[\frac{\1\{\|\gamma_{t-s}\|\ge \|\delta_{1-t}\| \}}{\|\gamma_{t-s}\|^3} e^{-\frac 5{2s} \|\gamma_{t-s}\|^2} \right]   = \frac{5q^2}{s^{3/2}(t-s)^{3/2}}  \E\left[\int_{\frac {\|\delta_{1-t}\|}{\sqrt{t-s}}}^\infty r e^{-\frac 52 r^2 (1+  \frac {t-s}{s})}\, dr \right]  \\
 =&  \frac{q^2}{\sqrt{s}(t-s)^{3/2}t}\E\left[ e^{- \frac{\|\delta_{1-t}\|^2t}{s(t-s)}}\right] = \frac{5q^3}{\sqrt{s}(t-s)^{3/2}t} \int_0^\infty r^4 e^{-\frac 52r^2(1+ \frac{t(1-t)}{s(t-s)})}\, dr = \frac{q^2s^2(t-s)}{t\, \Delta^{5/2}}, 
\end{align*} 
with 
$$\Delta := t(1-t) + s(t-s)  = (1-t)(t-s) + s(1-s).$$
Likewise, 
\begin{align*}
I_2&:= \E\left[\frac {\1\{\|\beta_s\|\ge\|\gamma_{t-s}\|,\, \|\delta_{1-t}\| \ge  \|\gamma_{t-s}\| \} }{\|\beta_s-\beta_1\|^3 \cdot \|\beta_t\|^3} \right] = \frac{q}{s^{3/2}}  \E\left[\frac{\1\{\|\gamma_{t-s}\|\le \|\delta_{1-t}\| \}}{\|\gamma_{t-s} + \delta_{1-t}\|^3} e^{-\frac 5{2s} \|\gamma_{t-s}\|^2} \right] \\
 =& \frac{q}{s^{3/2}}  \E\left[\frac{\1\{\|\gamma_{t-s}\|\le \|\delta_{1-t}\| \}}{\|\delta_{1-t}\|^3} e^{-\frac 5{2s} \|\gamma_{t-s}\|^2} \right]   
= \frac{5q^2}{s^{3/2}(1-t)^{3/2}}  \E\left[ e^{-\frac 5{2s} \|\gamma_{t-s}\|^2} \int_{\frac {\|\gamma_{t-s}\|}{\sqrt{1-t}}}^\infty r e^{-\frac 52 r^2}\, dr \right]  \\
 =& \frac{q^2}{s^{3/2}(1-t)^{3/2}}  \E\left[ e^{-\frac 5{2} \|\gamma_{t-s}\|^2(\frac 1s + \frac {1}{1-t})}\right]  = \frac{q^2s(1-t)}{\Delta^{5/2}}. 
\end{align*} 
Define as well  
\begin{align*}
I_3 & := \E\left[\frac {\1\{\|\beta_s\|\le \|\gamma_{t-s}\|\le \|\delta_{1-t}\| \} }{\|\beta_s-\beta_1\|^3 \cdot \|\beta_t\|^3} \right], 
\end{align*}
\begin{align*}
I_4:=\E\left[\frac {\1\{ \|\delta_{1-t}\|\le\|\beta_s\|\le \|\gamma_{t-s}\| \} }{\|\beta_s-\beta_1\|^3 \cdot \|\beta_t\|^3} \right], \   I_5:= \E\left[\frac { \1\{\|\beta_s\|\le \|\delta_{1-t}\|\le \|\gamma_{t-s}\| \} }{ \|\beta_s-\beta_1\|^3 \cdot \|\beta_t\|^3} \right].
\end{align*}
Note that by symmetry one has 
$$\int_{0\le s\le t \le 1} I_1 \, ds \, dt = \int_{0\le s\le t \le 1} I_3 \, ds \, dt,\text{ and }  \int_{0\le s\le t \le 1} I_4 \, ds \, dt = \int_{0\le s\le t \le 1} I_5 \, ds \, dt. $$  
Observe also that, 
$$I_1+I_2 = \frac{q^2s}{t  \Delta^{3/2}}. $$ 
Moreover, using symmetry again, we can see that   
$$\int_0^t \frac {s-t/2}{ \Delta^{3/2}} \, ds = 0,$$
and thus 
$$\int_0^t (I_1+I_2)\, ds = \frac {q^2}{2} \int_0^t \frac{1}{ \Delta^{3/2}}\, ds. $$
Likewise, 
\begin{align*}
& \int_{0\le s\le t\le 1} I_1\, ds \, dt = \int_{0\le s\le t\le 1}  \frac{q^2s(t-s)^2}{t\Delta^{5/2}}\, ds \, dt = \frac 12\int_{0\le s\le t\le 1}  \frac{q^2s(t-s)}{\Delta^{5/2}}\, ds \, dt \\
 =&  \int_{0\le s\le t\le 1}  \frac{q^2(1-t)(t-s)}{2\Delta^{5/2}}\, ds \, dt  = \int_{0\le s\le t\le 1} \frac {q^2t(1-t)}{4\Delta^{5/2}}\, ds\, dt = \int_{0\le s\le t\le 1} \frac {q^2}{6\Delta^{3/2}}\, ds\, dt.
\end{align*}
It follows that 
$$\int_{0\le s\le t\le 1} (I_1+I_2+I_3) \, ds\, dt = \frac {2q^2}3 \int_{0\le s\le t\le 1} \Delta^{-3/2}\, ds\, dt.$$
We consider now the term $I_4$, which is a bit more complicated to compute, thus we only give a lower bound on a suitable interval. To be more precise, we first define for $r\ge 0$ and $\lambda\ge 0$,  
$$F(r):=\int_0^r s^4e^{-5s^2/2}\, ds,\quad \text{and}\quad F_2(\lambda, r):=\int_0^r F(\lambda s) s^4e^{-5s^2/2}\, ds,$$
and then we write, 
\begin{align*}
& I_4= \E\left[ \frac { \1\{ \|\delta_{1-t}\|\le\|\beta_s\|\le \|\gamma_{t-s}\| \} }{ \|\gamma_{t-s}\|^6} \right]
= 5q\cdot \E\left[\frac { \1\{ \|\beta_s\|\le \|\gamma_{t-s}\| \} }{ \|\gamma_{t-s}\|^6} F\left(\frac{\|\beta_s\|}{\sqrt{1-t}}\right)\right]\\
 =&   \E\left[\frac {(5q)^2}{ \|\gamma_{t-s}\|^6}  F_2\left(\frac{\sqrt s}{\sqrt{1-t}},\frac{\|\gamma_{t-s}\|}{\sqrt s}\right)\right]
= \frac{(5q)^3}{(t-s)^3}  \int_0^\infty \frac{e^{-\frac{5r^2}{2}}}{r^2}  F_2\left(\frac{\sqrt s}{\sqrt{1-t}},r\frac{\sqrt{t-s}}{\sqrt s}\right)\, dr \\
 =& \frac{(5q)^3}{(t-s)^3} \left\{  \frac{\sqrt{t-s}}{\sqrt s} \int_0^\infty F\left(r\frac{\sqrt{t-s}}{\sqrt{1-t}}\right) r^3e^{-\frac{5r^2}{2}}\, dr  -5\int_0^\infty F_2\left(\frac{\sqrt s}{\sqrt{1-t}},r\frac{\sqrt{t-s}}{\sqrt s}\right) e^{-\frac{5r^2}{2}}\, dr \right\}\\
 \ge & \frac{(5q)^3}{(t-s)^3} \left\{ \frac{(t-s)^{\frac 32}}{s^{3/2}} \int_0^\infty  F\left( r\frac{\sqrt{t-s}}{\sqrt {1-t}}\right) r^3 e^{-\frac {5r^2t}{2s}}\, dr +\frac{(2s-t)\sqrt{t-s}}{s^{3/2}} \int_0^\infty F\left(r\frac{\sqrt{t-s}}{\sqrt{1-t}}\right) r^3e^{-\frac {5r^2}{2}}\, dr \right\},
\end{align*}
using that 
$$F_2(\lambda,r)\le \frac 15 r^3F(\lambda r)(1-e^{-5r^2/2}).$$ 
Therefore, if $t/2\le s\le t$, 
\begin{align*}
I_4& \ge  \frac{(5q)^3}{[s(t-s)]^{3/2}}  \int_0^\infty r^3 F\left( r\frac{\sqrt{t-s}}{\sqrt {1-t}}\right) e^{-\frac {5r^2t}{2s}}\, dr \\
& = \frac{(5q)^3\sqrt s}{t^2(t-s)^{3/2}}   \int_0^\infty r^3 F\left( r\frac{\sqrt{s(t-s)}}{\sqrt {t(1-t)}}\right) e^{-5r^2/2}\, dr\\
& \ge \frac{2\cdot 5^2q^3\sqrt s}{t^2(t-s)^{3/2}}\int_0^\infty  F\left( r\frac{\sqrt{s(t-s)}}{\sqrt {t(1-t)}}\right) re^{-\frac{5r^2}{2}}\, dr 
=\frac{2\cdot 5q^3s^3(t-s)}{t^2[t(1-t)]^{5/2}}\int_0^\infty  r^4 e^{-\frac{5r^2\Delta}{2t(1-t)}}\, dr\\
& =  \frac{2 q^2s^3(t-s)}{t^2\Delta^{5/2}}\ge \frac{q^2s(t-s)}{2\Delta^{5/2}},
\end{align*}
and as a consequence,  
\begin{align*}
\int_{0\le s\le t\le 1} I_4\, ds\, dt & \ge  \int_{t/2\le s\le t\le 1}I_4\, ds\, dt \ge \frac {q^2}2 \int_{t/2\le s\le t\le 1}\,  \frac{s(t-s)}{\Delta^{5/2}}\, ds\, dt \\
&=\frac {q^2}4 \int_{0\le s\le t\le 1}\frac{s(t-s)}{\Delta^{5/2}}\, ds\, dt  = \frac {q^2}{12} \int_{0\le s\le t\le 1}\Delta^{-3/2}\, ds\, dt.
\end{align*}
Putting all these estimates together yields 
$$\int_{0\le s\le t\le 1} \E\left[\frac 1{\|\beta_s-\beta_1\|^3 \cdot \|\beta_t\|^3} \right] \, ds\, dt = \sum_{k=1}^5 \int_{0\le s\le t\le 1} I_k\, ds\, dt \ge \frac 56 \int_{0\le s\le t\le 1} \Delta^{-3/2}\, ds\,dt.$$
Thus it just remains to show that 
\begin{eqnarray}\label{finalgoal}
\int_{0\le s\le t\le 1} \Delta^{-3/2}\, ds\,dt \ge \frac 65 \int_{0\le s\le t\le 1} \tilde \Delta^{-3/2}\, ds\,dt, 
\end{eqnarray}
where $\tilde \Delta := t(1-s)$. 
Note that $\Delta = \tilde \Delta +(t-s)^2$. 
Recall also that for any $\alpha \in \R$, and any $x\in (-1,1)$, 
\begin{eqnarray}\label{DL1+x}
(1+x)^\alpha = 1+ \sum_{i\ge 1}\frac{\alpha (\alpha-1)\dots(\alpha-i+1)}{i!} x^i.
\end{eqnarray}
Thus 
$$\frac 1{\Delta^{3/2}} = \frac 1{\tilde \Delta^{3/2}} \left(1+\sum_{k\ge 1} \frac{(3/2)(5/2)\dots (k+1/2)}{k!} \cdot \frac{(t-s)^{2k}}{\tilde \Delta^k}\right).$$
One needs now to compute the coefficients $C_k$ defined by 
$$C_k := \frac{(3/2)(5/2)\dots (k+1/2)}{k!} \int_{0\le s\le t\le 1} \frac{(t-s)^{2k}}{\tilde \Delta^{k+3/2}}\, ds\,dt.$$
We claim that one has for any $k\ge 0$, 
\begin{eqnarray}\label{Ckformula}
C_k= \frac{2^{2k+2}}{2k+1}(-1)^k \Sigma_k, 
\end{eqnarray}
with $\Sigma_0=1$, and for $k\ge 1$, 
$$\Sigma_k = 1+ \sum_{i=1}^{2k}(-1)^i \frac{(k+1/2)(k-1/2)\dots(k-i +3/2)}{i!}.$$ 
We will prove this formula in a moment, but let us conclude the proof of the lemma first, assuming it is true. 
Straightforward computations show by \eqref{Ckformula} that 
$$C_0 = 4,\quad C_1= \frac 23, \quad \text{and}\quad C_2= \frac {3}{10},$$
and $C_0+C_1+C_2\ge 6C_0/5$, gives \eqref{finalgoal} as wanted. 

So let us prove \eqref{Ckformula} now. Note that one can assume $k\ge 1$, as the result for $k=0$ is immediate. By \eqref{DL1+x}, one has 
$$(1-s)^{-k-3/2}= 1+ \sum_{i\ge 1} \frac{(k+3/2)(k+5/2)\dots (k+ i +1/2)}{i!} s^i.$$
Thus by integrating by parts, we get  
$$\int_0^t \frac{(t-s)^{2k}}{(1-s)^{k+3/2}} \, ds = (2k)!  \sum_{i\ge 0} \frac{(k+3/2)\dots(k+i+1/2)}{(2k + i +1)!}\cdot t^{2k+i+1},$$
and then 
$$\int_0^1 \int_0^t \frac{(t-s)^{2k}}{t^{k+3/2}(1-s)^{k+3/2}} \, ds\, dt = (2k)!  \sum_{i\ge 0} \frac{(k+3/2)\dots(k+i-1/2)}{(2k + i +1)!}.$$
As a consequence, 
\begin{align*}
C_k& = \frac{(2k)!}{k!} \sum_{i\ge 0} \frac{(3/2)(5/2)\dots(k+i-1/2)}{(2k+i+1)!} \\
& =  \frac{(2k)!}{(k+1/2)(k-1/2)\dots(3/2)(1/2)^2\cdot k!} \sum_{i\ge 0} \frac{|(k+1/2)(k-1/2)\dots(-k-i+1/2)|}{(2k + i +1)!}  \\
& = \frac{2^{2k+2}}{2k+1} \sum_{i\ge 0}  \frac{|(k+1/2)(k-1/2)\dots(-k-i+1/2)|}{(2k + i +1)!}, 
\end{align*}
and it just remains to observe that the last sum is well equal to $\Sigma_k$. The latter is obtained by taking the limit as $t$ goes to $1$ in the 
formula \eqref{DL1+x} for $(1-t)^{k+1/2}$. This concludes the proof of Lemma \ref{lem.var.4}. \hfill $\square$ 

\begin{remark}
\emph{It would be interesting to show that the covariance between $1/\|\beta_s-\beta_1\|^3$ and $1/\|\beta_t\|^3$ itself is positive for all $0\le s\le t\le 1$, and not just its integral, as we have just shown. }
\end{remark} 
%%%%%%%%%%%%%%%%%%%%%%%%%%%%%%%%%%%%%%%%%%%%%%%%%%%%%%%%%%%%%%%%%%%%%%%%%%%%%%%%%%%%%% 
\section{Proof of Theorem B}
The proof of Theorem B is based on the Lindeberg-Feller theorem for triangular arrays, that we recall for convenience (see Theorem 3.4.5 in \cite{Dur}): 
\begin{theorem}[Lindeberg-Feller]\label{thm:lind}
For each $n$ let $(X_{n,i}: \, 1\leq i\leq n)$ be a collection of independent random variables with zero mean. Suppose that the following two conditions are satisfied
\newline
{\rm{(i)}} $\sum_{i=1}^{n}\E[X_{n,i}^2] \to \sigma^2>0$ as $n\to \infty$, and 
\newline
{\rm{(ii)}} $\sum_{i=1}^{n}\E\left[(X_{n,i})^2\1\{|X_{n,i}|>\epsilon\}\right] \to 0$, as $n\to \infty$, for all $\epsilon>0$.
\newline
Then, $S_n=X_{n,1}+\ldots + X_{n,n} \Longrightarrow  \NN(0,\sigma^2)$, as $n\to \infty$.
\end{theorem}
In order to apply this result, one needs three ingredients. The first one is an asymptotic estimate for the variance of the capacity of the range, 
which is given by our Theorem A. 
The second ingredient is a decomposition of the capacity of two sets as a sum of the capacities of the two sets minus some error term, in the spirit of the inclusion-exclusion 
formula for the cardinality of a set, which allows to decompose the capacity of the range up to time $n$ into a sum of independent pieces having the law of the capacity of the range up to a smaller time index, and finally the last ingredient is a sufficiently good bound on the centered fourth moment.

This strategy has been already employed successfully for the capacity of the range in dimension six and more in \cite{ASS18} (and for the size of the range as well, see  \cite{JO69, JP71}).  
In this case the asymptotic of the variance followed simply from a sub-additivity argument, but the last two ingredients are entirely similar in dimension $5$ and in higher dimension. In particular one has the following decomposition (see Proposition 1.6 in \cite{ASS19}): for any two subsets $A,B\subset \Z^d$, $d\ge 3$, 
\begin{eqnarray}\label{cap.decomp}
\cp(A\cup B) = \cp(A) + \cp(B) - \chi(A,B),
\end{eqnarray}
where $\chi(A,B)$ is some error term. Its precise expression is not so important here. All one needs to know is that  
$$|\chi(A,B)| \le 3\sum_{x\in A}\sum_{y\in B} G(x,y),$$ 
so that by \cite[Lemma 3.2]{ASS18}, if $\RR_n$ and $\tilde \RR_n$ are the ranges of two independent walks in $\Z^5$, then  
\begin{eqnarray}\label{bound.chin}
\E[\chi(\RR_n,\tilde \RR_n)^4] = \OO(n^2). 
\end{eqnarray}
We note that the result is shown for the simple random walk only in \cite{ASS18}, but the proof applies as well to our setting (in particular Lemma 3.1 thereof also follows from \eqref{exp.Green}). 
Now as noticed already by Le Gall in his paper \cite{LG86} (see his remark (iii) p.503), a good bound on the centered fourth moment follows from \eqref{cap.decomp} and \eqref{bound.chin}, and the triangle inequality in $L^4$. More precisely in dimension $5$, one obtains (see for instance the proof of Lemma 4.2 in \cite{ASS18} for some more details):   
\begin{eqnarray}\label{cap.fourth}
\E\left[\left(\cp(\RR_n)-\E[\cp(\RR_n)]\right)^4\right] = \OO(n^2(\log n)^4).
\end{eqnarray}
Actually we would even obtain the slightly better bound $\OO(n^2(\log n)^2)$, using our new bound on the variance $\var(\cp(\RR_n))=\OO(n\log n)$, but this is not needed here.  
Using next a dyadic decomposition of $n$, one can write with $T:=\lfloor n/(\log n)^4\rfloor$, 
\begin{eqnarray}\label{cpRn}
 \cp(\RR_n) = \sum_{i=0}^{\lfloor n/T\rfloor} \cp(\RR^{(i)}_T) - R_n,
 \end{eqnarray}
where the $(\RR^{(i)}_T)_{i=0,\dots,n/T}$ are independent pieces of the range of length either $T$ or $T+1$, and 
$$
R_n= \sum_{\ell =1}^L \sum_{i=0}^{2^{\ell-1}} \chi(\RR^{(2i)}_{n/2^\ell},\RR^{(2i+1)}_{n/2^\ell}),
$$
is a triangular array of error terms (with $L=\log_2(\log n)^4$). Then it follows from \eqref{bound.chin}, that 
\begin{align*}
\var(R_n) &\le L \sum_{\ell=1}^L \var\left(\sum_{i=1}^{2^{\ell-1}} \chi(\RR^{(2i)}_{n/2^\ell},\RR^{(2i+1)}_{n/2^\ell})\right)\le L\sum_{\ell=1}^L \sum_{i=1}^{2^{\ell-1}}  \var\left(\chi(\RR^{(2i)}_{n/2^\ell},\RR^{(2i+1)}_{n/2^\ell})\right)\\
& = \OO(L^2 n)=\OO(n(\log \log n)^2).
\end{align*}
In particular $(R_n-\E[R_n])/\sqrt{n\log n}$ converges in probability to $0$. Thus one is just led to show the convergence in law of the remaining sum in \eqref{cpRn}. 
For this, one can apply Theorem \ref{thm:lind}, with 
$$X_{n,i}:=\frac{\cp(\RR^{(i)}_T)-\E\left[\cp(\RR^{(i)}_T)\right]}{\sqrt{n\log n}}.$$ 
Indeed, Condition (i) of the theorem follows from Theorem A, and Condition (ii) follows 
from \eqref{cap.fourth} and Markov's inequality (more details can be found in \cite{ASS18}). This concludes the proof of Theorem B.  \hfill $\square$

\section*{Acknowledgments} We thank Fran{\c c}oise P\`ene for enlightening discussions at an early stage of this project, and Pierre Tarrago for Reference \cite{Uchiyama98}. We also warmly thank Amine Asselah and Perla Sousi for our many discussions related to the subject of this work, which grew out of it. The author was supported by the ANR SWIWS (ANR-17-CE40-0032) and MALIN (ANR-16-CE93-0003).

\end{document}